\documentclass[12pt]{article}

\usepackage{fullpage}
\usepackage{authblk}

\usepackage[square]{natbib}
\RequirePackage[colorlinks,citecolor=blue,urlcolor=blue]{hyperref}
\usepackage{mathrsfs,amsthm,amsmath}
\usepackage{amssymb,multirow}
\usepackage{amsfonts}
\usepackage{stmaryrd}
\usepackage{dsfont}
\usepackage[ruled,vlined]{algorithm2e}

\usepackage{ucs}
\usepackage[toc]{appendix}
\usepackage[utf8x]{inputenc}
\usepackage{amsmath}
\usepackage{amsfonts}
\usepackage{amssymb}
\usepackage{amsthm}
\usepackage{fontenc}
\usepackage{graphicx}

\usepackage{bbm}
\usepackage{bm}
\RequirePackage[OT1]{fontenc}
\usepackage[usenames,dvipsnames]{color}

\newcommand{\E}{\mathds{E}}

\renewcommand{\P}{\mathds{P}}

\newcommand{\ddr}{\mathrm{d}}

\def\simiid{\stackrel{\mbox{\scriptsize{iid}}}{\sim}}
\theoremstyle{plain}
\newtheorem{thm}{\textsc{Theorem}}

\newtheorem{lem}{\textsc{Lemma}}
\newtheorem{cor}{\textsc{Corollary}}
\newtheorem{assumption}{\textsc{Assumption}}
\newtheorem{rmk}{Remark}
\newtheorem{prp}{\textsc{Proposition}}

\newcommand{\Reals}{\mathbb{R}}
\newcommand{\EE}{\mathbb{E}}
\newcommand{\PP}{\mathbb{P}}
\newcommand{\QQ}{\mathbb{Q}}
\newcommand{\var}{\mathrm{var}}
\newcommand{\1}{\bm{1}}
\newcommand{\intd}{\mathrm{d}}
\newcommand{\gammaDist}{\mathrm{Gamma}}
\DeclareMathOperator*{\argmax}{arg\,max}

\allowdisplaybreaks

\begin{document}

\title{Bayesian Nonparametric Inference for ``Species-sampling" Problems}
\date{}

\author{Cecilia Balocchi\thanks{Cecilia Balocchi is Assistant Professor, School of Mathematics, University of Edinburgh, Edinburgh, United Kingdom (e-mail: \href{mailto:cecilia.balocchi@ed.ac.uk}{cecilia.balocchi@ed.ac.uk}).},
Stefano Favaro\thanks{Stefano Favaro is Professor, Department of Economics and Statistics, University of Torino and Collegio Carlo Alberto, Torino, Italy (e-mail:
\href{mailto:stefano.favaro@unito.it}{stefano.favaro@unito.it}) and is also affiliated to IMATI-CNR ``Enrico  Magenes" (Milan, Italy).},
Zacharie Naulet\thanks{Zacharie Naulet is Assistant Professor, Department of Mathematics, Universit\'e Paris-Saclay, Orsay, France (e-mail: \href{mailto:zacharie.naulet@universite-paris-saclay.fr}{zacharie.naulet@universite-paris-saclay.fr}).}
}






\maketitle
  \vspace{-1.2cm}

\bigskip

\begin{abstract} 
Given an observed sample from a population of individuals belonging to species, ``species-sampling'' problems (SSPs) call for estimating some features of the unknown species composition of additional unobservable samples from the same population. Within SSPs, the problems of estimating coverage probabilities, the number of unseen species and coverages of prevalences have emerged in the past three decades for being the subject of numerous methodological and applied works, mostly in biological sciences but also in statistical machine learning, electrical engineering, theoretical computer science, information theory and forensic statistics. In this paper, we focus on these popular SSPs, and present an overview of their Bayesian nonparametric (BNP) analysis under the Pitman--Yor process (PYP) prior. While reviewing the literature, we improve on computation and interpretability of existing posterior inferences, typically expressed through complicated combinatorial numbers, by establishing novel posterior representations in terms of simple compound Binomial and Hypergeometric distributions. We also consider the  problem of estimating the discount and scale parameters  of the PYP prior, showing a property of Bayesian consistency with respect to estimation through the hierarchical Bayes and empirical Bayes approaches, that is: the discount parameter can be estimated consistently, whereas the scale parameter cannot be estimated consistently, thus advising caution in posterior inference. We conclude our work by discussing some generalizations of SSPs, mostly in the field of biological sciences, which deal with ``feature-sampling", multiple populations of individuals sharing species and classes of Markov chains.
\end{abstract}



\section{Introduction}

The estimation of the number of unseen species is a classical problem in statistics, dating back to the seminal work of \citet{Fis(43)}. Consider a generic population of individuals, such that each individual takes a value in a (possibly infinite) space of species' labels or symbols. Assuming $n\geq1$ observed individuals to be modeled as a random sample $(X_{1},\ldots,X_{n})$ from an unknown discrete distribution $p$, the unseen-species problem calls for estimating
\begin{displaymath}
\mathfrak{u}_{n,m}=|\{X_{n+1},\ldots,X_{n+m}\}\setminus \{X_{1},\ldots,X_{n}\}|,
\end{displaymath}
namely the number of hitherto unseen (distinct) species that would be observed if $m\geq1$ additional samples $(X_{n+1},\ldots,X_{n+m})$ were collected from the same distribution. The unseen-species problem may be viewed as the $m$-step ahead generalization of the problem of estimating the missing mass, namely the probability of discovering at the $(n+1)$-th draw a species not observed in the sample $(X_{1},\ldots,X_{n})$ \citep{Goo(49),Goo(53)}. One may consider several refinements or generalizations of the unseen-species problem, by defining suitable discrete functionals of the $X_{i}$'s according to the features of interest on the unknown species’ composition of unobserved samples \citep{Goo(56),Efr(76)}. We refer to these problems as ``species-sampling" problems (SSPs), though SSPs may also refer to a broader class of statistical problems that deal with sampling from generic populations of species. SSPs first appeared in ecology for the estimation of the species richness or diversity of ecological populations, and their importance has grown in the most recent years driven by applications in biological and physical sciences, statistical machine learning, electrical engineering, theoretical computer science, information theory and forensic statistics.

Biological sciences are the field where SSPs have been most investigated over the past three decades, raising several challenges in both methods and applications. This is testified by the work of \citet{Den(19)}, which shows how large scale genomic data provide a fertile ground for SSPs. Although sequencing technologies have advanced the understanding of genome biology, observed samples may not be perfectly representative of the molecular heterogeneity or species composition of the underlying DNA library, often providing a poor representation due to low-abundance molecules that are hard to sample. Due to the impossibility of sequencing DNA libraries up to complete saturation, it is common to make use of the observed samples, typically collected under suitable budget constraints, to infer the molecular heterogeneity of additional unobserved samples from the library, as well as of the library itself. \citet{Den(19)} identified three major questions of interest:
\begin{itemize}
\item[Q1)] what is the expected population frequency of a species with frequency $r\geq1$ in the sample?
\item[Q2)] how many previously unobserved species in the sample will be observed in additional samples?
\item[Q3)] how many species with frequency $r\geq1$ in the sample will be observed in additional samples?
\end{itemize}
These are popular examples of SSPs, with Q2) being the unseen-species problem, and they all apply to statistical  analysis related to the study of sequencing complexity \citep{Dal(13)}, design of sequencing experiments \citep{Sim(14)} and estimation of genetic diversity \citep{Gao(07)}, immune receptor diversity \citep{Rob(09)} and genetic variation \citep{Ion(09)}.

Although nonparametric estimation of the number of unseen-species dates back to the 1950s, only recent works have set forth a rigorous and comprehensive treatment of such a problem \citep{Orl(17),Wu(19),Pol(20)}. Besides providing estimators of $\mathfrak{u}_{n,m}$ with provable (theoretical) guarantees, these works have introduced novel tools that pave the way to investigate other SSPs \citep{Wu(21)}. In general, nonparametric estimation of SSPs does not rely on any assumption on the underlying distribution $p$ of the $X_{i}$'s, with provable guarantees that hold uniformly over all discrete distributions. This assumption-free framework has led to develop solid theories in their greatest generality, though worst case distributions may severely hamper the empirical performance of the proposed estimators, leading to unreliable results in applications. It is therefore useful to take into account any knowledge about the nature of data, which typically results in placing assumptions on the tail behaviour of the distribution $p$. That is, the assumption-free framework of SSPs may be usefully complemented through suitable prior assumptions of regularity on the tail behaviour of $p$.  A common and flexible tail assumption is that of regular variation, which allows for $p$ to range from a geometric tail to an heavy power-law tail behaviour \citep{Gne(07)}. This is well-motivated by the ubiquitous power-law type distributions, which occur in many situations of scientific interest, and nowadays have significant consequences for the understanding of numerous natural and social phenomena \citep{Cla(09)}.

A Bayesian nonparametric (BNP) approach to SSPs has been set forth in \citet{Lij(07)}, and it relies on specifying a prior on the distribution $p$ of the $X_{i}$'s. This is arguably the most natural approach to complement the assumption-free framework of SSPs. In this respect, discrete random probability measures in the form of species sampling models \citep{Pit(96)} provide a broad class of nonparametric priors for $p$ \citep[Chapter 3 and Chapter 4]{Pit(06)}. Among species sampling models, \citet{Lij(07)} focussed on the Pitman--Yor process (PYP) prior \citep{Per(92),Pit(97)}, whose mathematical tractability and interpretability make it the natural candidate in applications \citep{Fav(09),Fav(12)}. The PYP prior is indexed by a discount parameter $\alpha\in[0,1)$ and a scale parameter $\theta>-\alpha$, such that for $\alpha=0$ it reduces to the Dirichlet process (DP) of \citet{Fer(73)}. Of special interest is the parameter $\alpha$, as it admits a clear interpretation in terms of controlling the tail behaviour of the PYP prior, which ranges from geometric tails to heavy power-law tails. In particular, the larger $\alpha$ the heavier the tail of the PYP prior, and as a limiting case for $\alpha\rightarrow0$ one recovers the geometric tail behaviour of the DP prior \citep[Chapter 3]{Pit(06)}. Such a peculiar parameterization makes the PYP a flexible prior choice, which allows for tuning of the tail behaviour of the prior with respect to the empirical distribution of the data. Among species sampling models, the PYP prior is certainly unique with respect to mathematical tractability, flexibility and interpretability \citep{Deb(15),Bac(17)}. 

\subsection{Our contributions}

In this paper, we present an overview of BNP inference for SSPs under the PYP prior. We focus on SSPs corresponding to the aforementioned questions Q1, Q2 and Q3. As for Q1, we consider the problem of estimating coverage probabilities, which include the missing mass and the coverage probability of order $r\geq1$, namely the probability mass of species observed with frequency $r$ in the sample. As for Q2, we consider the unseen-species problem and a generalization of it defined in terms of the estimation of the number of hitherto unseen species that would be observed with frequency $r\geq1$ in $m$ additional samples, here referred to as the unseen species' prevalences of order $r$. Finally, as for Q3, we consider the problem of estimating the coverage of prevalence of order $r\geq1$, namely the number of species with frequency $r$ in the sample that would be observed in $m$ additional samples. We introduce each SSP in the assumption-free framework, and then we present its analysis in the BNP framework under the PYP prior. While reviewing posterior analyses of SSPs from the recent BNP literature, we establish some novel representations of posterior distributions. In particular, we show that posterior distributions that are typically expressed in terms of complicated combinatorial numbers, which hamper both the computation and the interpretability of posterior inferences, admit straightforward representations in terms of compound Binomial and Hypergeometric distributions. This is a step forward in the BNP approach to SSPs, especially with respect to its use in problems of practical interest, contributing to simplify and make more interpretable the posterior inferences.

Critical in the BNP approach to SSPs under the PYP prior is the estimation of the prior's parameters $(\alpha,\theta)$. Two approaches for estimating $(\alpha,\theta)$ are: i) the hierarchical Bayes or fully Bayes approach, which relies on the posterior distribution of $(\alpha,\theta)$ with respect to a suitable prior specification; ii) the empirical Bayes approach, which relies on maximizing, with respect to $(\alpha,\theta)$, the (marginal) likelihood of the observed sample. The empirical Bayes approach is the most used in practice \citep{Lij(07),Fav(09),Deb(15)}, though no provable guarantees have been established for empirical Bayes estimates of $(\alpha,\theta)$. The lack of a theoretical understanding of the hierarchical Bayes approach and the empirical Bayes approach has precluded clear guidelines for choosing among them. Here, we investigate their properties of Bayesian consistency. Under moderate misspecification, we find that the empirical Bayes estimator of $(\alpha,\theta)$ converges to a limit that is interpretable in terms of the ``true" data generating mechanism. In particular, when the model is correct, we find that: i) both the empirical Bayes estimator and the hierarchical Bayes estimator of $\alpha$ are consistent; ii) $\theta$ can not be tested or estimated consistently, because of a curious anti-concentration result. Under the hierarchical Bayes approach, we characterize the large sample asymptotic behavior of the posterior distribution of the parameter $(\alpha,\theta)$; we find that the limiting posterior distribution of $\theta$ depends on the prior, thus particular caution should be used. As for $\alpha$, we prove a weak form of the Bernstein-von Mises theorem, finding its contraction rates.

\subsection{Organization of the paper}

The paper is structured as follows. In Section \ref{sec1} we introduce the BNP framework for SSPs under the PYP prior, and review the PYP, with emphasis on the interpretation of the prior's parameters $(\alpha,\theta)$. BNP inference for SSPs is presented in Section \ref{sec2} for the coverage probabilities, in Section \ref{sec3} for the number of unseen species, and in Section \ref{sec4} for the coverage of prevalences. In Section \ref{sec6} we consider the estimation of $(\alpha,\theta)$, showing properties of Bayesian consistency with respect to the hierarchical Bayes and empirical Bayes approaches. In Section \ref{sec7} we present some numerical illustrations for the estimation of $(\alpha,\theta)$ and for the estimation of the unseen, using synthetic data. Section \ref{sec8} contains a discussion on some emerging generalizations of SSPs, mostly in the field of biological and physical sciences, which deal with ``feature-sampling" problems, multiple populations of individuals sharing species and classes of Markov chains. Additional material, technical results, and proofs of the main results are deferred to the Supplementary Material \citep{BFN(23)}.


\section{The BNP species sampling framework }\label{sec1}

The assumption-free framework for SSPs, henceforth referred to as the ``classical framework", assumes that $n\geq1$ observed samples from the population are modeled as a random sample $\mathbf{X}_{n}=(X_{1},\ldots,X_{n})$ from an unknown discrete distribution $p$, i.e. $p=\sum_{j\geq1}p_{j}\delta_{s_{j}}$ with $p_{j}$ being the probability of  species' label $s_{j}$ for $j\geq1$, such that $\sum_{j\geq1}p_{j}=1$. Here, following \citet{Lij(07)}, we consider to endow $p$ with the PYP prior, namely we assume that
\begin{align}\label{eq:exchangeable_model}
X_i\,|\,P & \quad\simiid\quad P \qquad i=1,\ldots,n,\\[0.2cm]
\notag P& \quad\sim\quad\text{PYP}(\alpha,\theta),
\end{align}
with $\text{PYP}(\alpha,\theta)$ being the law of the PYP process with parameter $(\alpha,\theta)$. We refer to \eqref{eq:exchangeable_model} as the ``BNP framework" for SSPs. A simple and intuitive definition of the PYP follows from its stick-breaking construction \citep{Pit(95)}. For $\alpha\in[0,1)$ and $\theta>-\alpha$ let: i) $(V_{i})_{i\geq1}$ be independent r.v.s such that $V_{i}$ has  Beta distribution with parameter $(1-\alpha,\theta+i\alpha)$, for $i\geq1$; ii) let $(S_{j})_{j\geq1}$ be r.v.s, independent of the $V_{i}$'s, and i.i.d. as a non-atomic distribution $\nu$ on a measurable space $\mathbb{S}$. If we set $P_{1}=V_{1}$ and $P_{j}=V_{j}\prod_{1\leq i\leq j-1}(1-V_{i})$ for $j\geq1$, such that $P_{j}\in(0,1)$ for any $j\geq1$ and $\sum_{j\geq1}P_{j}=1$ almost surely, then the random probability measure $P=\sum_{j\geq1}P_{j}\delta_{S_{j}}$ is a PYP on $\mathbb{S}$ with base distribution $\nu$, discount $\alpha$ and scale $\theta$. The PYP generalizes the DP by means of the parameter $\alpha\in[0,1)$, which controls the tail behaviour of $P$. For any $\alpha\in(0,1)$ let $P\sim\text{PYP}(\alpha,\theta)$ and let $(P_{(j)})_{j\geq1}$ be the decreasingly ordered $P_{j}$'s. Then, as $j\rightarrow+\infty$ the $P_{(j)}$'s follow a power-law distribution of exponent $c=\alpha^{-1}$ \citep{Pit(97)}; that is, $\alpha\in(0,1)$ controls the tail behaviour of the PYP through the small $P_{(j)}$'s: the larger $\alpha$ the heavier the tail of $P$. As a limiting case for $\alpha\rightarrow0$, the DP features geometric tails \citep[Chapter 3 and Chapter 4]{Pit(06)}.

\subsection{Sampling formulae for the PYP prior}\label{sec11}

The random sample $\mathbf{X}_{n}$ from $P\sim\text{PYP}(\alpha,\theta)$ is regarded as part of the exchangeable sequence $(X_{n})_{n\geq1}$ satisfying \eqref{eq:exchangeable_model}, that is a sequence whose directing (de Finetti) measure is the law of the PYP. Because of the almost sure discreteness of $P$, $\mathbf{X}_{n}$ features $K_{n}\leq n$ distinct species, labelled by $\{S^{\ast}_{1},\ldots,S^{\ast}_{K_{n}}\}$, with frequencies $\mathbf{N}_{n}=(N_{1,n},\ldots,N_{K_{n},n})$ such that $N_{i}\geq1$, for $i=1,\ldots,k$, and $\sum_{1\leq i\leq k}N_{i}=n$. In other terms, $\mathbf{X}_{n}$ induces a random partition $\Pi_{n}$ of $\{1,\ldots,n\}$ whose blocks are the classes induced by the equivalence relation $i\sim j\iff X_{i}=X_{j}$. In particular, $\Pi_{n}$ is an exchangeable random partition, meaning that its distribution is such that the probability of any partition of $\{1,\ldots,n\}$ into $k$ blocks with frequency $\mathbf{n}=(n_{1},\ldots,n_{k})$ is a symmetric function $p_{n,k}$ of $\mathbf{n}$, i.e. 
\begin{equation}\label{eppf_pit}
p_{n,k}(\mathbf{n})=\frac{\left(\frac{\theta}{\alpha}\right)_{(k)}}{(\theta)_{(n)}}\prod_{i=1}^{k}\alpha(1-\alpha)_{(n_{i}-1)},
\end{equation}
where $(a)_{(u)}$ is the $u$-th rising factorial of $a$, i.e. $(a)_{(u)}=\prod_{0\leq i\leq u-1}(a+i)$, for $a\geq0$ and $u\in\mathbb{N}_{0}$. The function $p_{n,k}$ is referred to as the exchangeable partition probability function of $\Pi_{n}$ \citep{Kin(78),Pit(95)}. The sequence $(\Pi_{n})_{n\geq1}$ defines an exchangeable random partition $\Pi$ of $\mathbb{N}$, with exchangeability meaning that the distribution of $\Pi$ is invariant under finite permutations of its elements, provided that $\Pi_{m}$ is the restriction of $\Pi_{n}$ to the first $m$ elements, almost surely for all $m<n$. This implies that 
\begin{align}\label{EPPF_const}
&p_{n,k}(\mathbf{n})=p_{n+1,k+1}(\mathbf{n},1)+\sum_{i=1}^{k}p_{n+1,k}(\mathbf{n}+\mathbf{1}_{i}),
\end{align}
for all $n\geq 1$, with $\mathbf{1}_{i}$ being a vector of length $k$ with $1$ in the position $i$ and $0$ elsewhere \citep{Han(00),Nac(06)}. See \citet[Chapter 2]{Pit(06)} for details on exchangeable random partitions and generalizations thereof.

The construction of $\Pi$, and in particular \eqref{EPPF_const}, implies that its distribution is completely determined by the distribution of $(X_{i})_{i\geq1}$ through the conditional probability of $X_{n+1}$ given $\mathbf{X}_{n}$, i.e. the predictive probability. If $\mathbf{X}_{n}$ is a random sample from $P\sim\text{PYP}(\alpha,\theta)$ as described above, and $\text{Pr}[X_{1}\in\cdot]=\nu(\cdot)$, then the predictive probabilities of $(X_{i})_{i\geq1}$ are
\begin{align}\label{eq:pred}
&\text{Pr}[X_{n+1}\in\cdot\,|\,\mathbf{X}_{n}]=\frac{\theta+k\alpha}{\theta+n}\nu(\cdot)+\sum_{i=1}^{k}\frac{n_{i}-\alpha}{\theta+n}\delta_{S^{\ast}_{i}}(\cdot),
\end{align}
determining the distribution of $\Pi_{n+1}$ from that of $\Pi_{n}$, for $n\geq1$. The probability \eqref{eq:pred} is a linear combination between: i) the probability $(\theta+k\alpha)/(\theta+n)$ that $X_{n+1}$ is a new species, i.e. the probability of creating a new block in the random partition of $\{1,\ldots,n\}$; ii) the probability  $(n_{i}-\alpha)/(\theta+n)$ that $X_{n+1}$ takes value $S^{\ast}_{i}$, i.e. the probability of increasing by $1$ the size of the block $S^{\ast}_{i}$ in the random partition of $\{1,\ldots,n\}$, for $i=1,\ldots,k$. For $\theta>0$, an intuitive description of \eqref{eq:pred} was proposed in \citet[Chapter 11]{Zab(05)}. Consider an urn containing a black ball and colored (non-black) balls, where colored balls may be interpreted as individuals with their associated species (colors). Balls are drawn from the urn, and then returned to the urn, in such a way that the probability of a ball being drawn at any stage is proportional to its weight. Initially the urn contains a black ball with weight $\theta>0$, and at the $n$-th draw: i) if we pick a colored ball then the ball is returned to the urn together with a ball of the same color with weight $1$; ii) if we pick a black ball, then the weight of the black ball is increased by $\alpha\in[0,1)$ and a ball of a new color with weight $1-\alpha$, i.e. any color not present in the urn, is inserted in the urn. If $X_{n}$ is the color of the ball returned in the urn after the $n$-th draw, and such a color is generated as $\nu$, then the conditional distribution of $X_{n+1}$ given $\mathbf{X}_{n}$  is \eqref{eq:pred}.

\citet{Zab(92),Zab(97)} provided a characterization of the PYP through its predictive probability \eqref{eq:pred}, referred to as ``sufficientness postulate" \citep{Joh(32),Bac(17)}. For $\alpha\in(0,1)$ the PYP is characterized as the sole species sampling model whose predictive probability is such that: i) the probability that $X_{n+1}$ is of a new species depends on $\mathbf{X}_{n}$ only through $n$ and $K_{n}$; ii) the probability that $X_{n+1}$ belongs to $S^{\ast}_{i}$ depends on $\mathbf{X}_{n}$ only through $n$ and $N_{n,i}$. In particular, the DP is the sole species sampling model whose predictive probability is such that the probability that $X_{n+1}$ belongs is of a new species depends on $\mathbf{X}_{n}$ only through $n$ \citep{Reg(78),Lo(91)}. The  ``sufficientness postulate" and the P\'olya-like urn scheme show how the parameter $\alpha$ drives a combined effect in terms of a reinforcement mechanism and the increase in the rate at which new species are generated according to \eqref{eq:pred}. A new species $S^{\ast}$ entering in the sample produces two effects: i) it is assigned a mass proportional to $(1-\alpha)$ in the $S^{\ast}$ empirical component of \eqref{eq:pred}; ii) it is assigned a mass proportional to $\alpha$ in the probability of generating a new species. That is, the mass assigned to $S^{\ast}$ is less than proportional to its cluster size, i.e. $1$, and the remaining mass is added to the probability of generating new species. The first effect gives rise to the following reinforcement mechanism: if $S^{\ast}$ is re-observed then the mass of $S^{\ast}$ is increased by $1/(\theta+n+1)$, meaning that the sampling procedure tends to reinforce observed species with higher frequencies. The second effect implies that the probability of generating yet another new species, which overall still decreases in $n$, is increased by $\alpha/(\theta+n+1)$. The larger $\alpha$ the stronger the reinforcement mechanism, and the higher is the probability of generating new species. For $\alpha=0$ everything is proportional to species' frequencies, in such a way that the number of distinct species does not alter the probability of generating new species.

We conclude by recalling the sampling formula or frequency-of-frequency distribution induced by a random sample $\mathbf{X}_{n}$ from $P\sim\text{PYP}(\alpha,\theta)$, which follows from  \eqref{eppf_pit} \citep[Chapter 2]{Pit(06)}. Let $M_{r,n}$ be the number of distinct species with frequency $r$ in $\mathbf{X}_{n}$, for $1\leq r\leq n$, i.e. $M_{r,n}=\sum_{1\leq i\leq K_{n}}I(N_{i,n}=r)$, such that $\sum_{1\leq r\leq n}M_{r,n}=K_{n}$ and $\sum_{1\leq r\leq n}rM_{r,n}=n$. The distribution of $\mathbf{M}_{n}=(M_{1,n},\ldots,M_{n,n})$ is defined on $\mathcal{M}_{n,k}=\{k\in\{1,\ldots,n\}\text{ and }(m_{1},\ldots,m_{n})\text{ : }m_{i}\geq0,\,\sum_{1\leq i\leq n}m_{i}=k,\,\sum_{1\leq i\leq n}im_{i}=n\}$, and referred to as Ewens-Pitman sampling formula (EPSF). For $\mathbf{m}=(m_{1},\ldots,m_{n})\in \mathcal{M}_{n,k}$
\begin{equation}\label{eq_ewe_py}
\text{Pr}[\mathbf{M}_{n}=\mathbf{m}]=n!\frac{\left(\frac{\theta}{\alpha}\right)_{(\sum_{i=1}^{n}m_{i})}}{(\theta)_{(n)}}\prod_{i=1}^{n}\frac{\left(\frac{\alpha(1-\alpha)_{(i-1)}}{i!}\right)^{m_{i}}}{m_{i}!}.
\end{equation}
We refer to Section~S2 of the Supplementary Material \citep{BFN(23)} for a representation of the EPSF in terms of a compound Poisson sampling model \citep[Chapter 7]{Cha(05)}. The distribution of $K_{n}$ follows from \eqref{eq_ewe_py} \citep[Chapter 3]{Pit(06)}. In particular, for $a>0$, $b\geq0$ and $u,v\in\mathbb{N}_{0}$ with $v\leq u$, let $\mathscr{C}(u,v;a,b)$ be the $(u,v)$-th non-centered generalized factorial coefficient, i.e., $\mathscr{C}(u,v;a,b)=(v!)^{-1}\sum_{0\leq i\leq v}(-1)^{i}{v\choose i}(-ia-b)_{(u)}$; see Section~S1 of the Supplementary Material \citep{BFN(23)}. Then, for $x\in\{1,\ldots,n\}$ it holds that
\begin{equation}\label{eq_dist_py}
\text{Pr}[K_{n}=x]=\frac{\left(\frac{\theta}{\alpha}\right)_{(x)}}{(\theta)_{(n)}}\mathscr{C}(n,x;\alpha,0).
\end{equation}
For $\alpha=0$, i.e. under the DP, the distribution of $K_{n}$ follows directly from \eqref{eq_dist_py} by letting $\alpha\rightarrow0$. The resulting distribution is expressed in terms of the $(u,v)$-th signless Stirling number $|s(u,v)|$, which arises by means of $|s(u,v)|=\lim_{a\rightarrow0}a^{-v}\mathscr{C}(u,v;a,0)$; see Section~S1 of the Supplementary Material \citep{BFN(23)} for details.

At the sampling level, the power-law tail behaviour of the PYP emerges naturally from the analysis of the large $n$ asymptotic behaviour of $K_{n}$ and $M_{r,n}$. For $\alpha\in(0,1)$ we denote by $f_{\alpha}$ the density function of a positive $\alpha$-stable r.v., and for any $\theta>-\alpha$ let $S_{\alpha,\theta}$ be a r.v. with density function
\begin{equation}\label{adivers}
f_{S_{\alpha,\theta}}(s)\propto s^{\frac{\theta-1}{\alpha}-1}f_{\alpha}(s^{-1/\alpha}),
\end{equation}
that is a generalized Mittag-Leffler density function, such that $\E[S^{r}_{\alpha,\theta}]=(\theta/\alpha)_{(r)}\Gamma(\theta)/\Gamma(\theta+r\alpha)$ for $r\geq1$, from which the mean and the variance of $S_{\alpha,\theta}$ follows immediately. \citet[Theorem 3.8]{Pit(06)} shows that, as $n\rightarrow+\infty$,
\begin{align} \label{eq:sigma_diversity}
n^{-\alpha}K_{n}\rightarrow S_{\alpha,\theta}\quad\text{almost surely},
\end{align}
and
\begin{align} \label{eq:sigma_diversity_m}
n^{-\alpha}M_{r,n}\rightarrow \frac{\alpha(1-\alpha)_{(r-1)}}{r!}S_{\alpha,\theta}\quad\text{almost surely}.
\end{align}
The r.v. $S_{\alpha,\theta}$ is positive and finite (almost surely), and it is typically referred to as Pitman's $\alpha$-diversity \citep{Pit(03),DF(20a),DF(20b),Ber(24)}. For $\alpha=0$, we recall that as $n\rightarrow+\infty$, $K(n)/\log n\rightarrow\theta$ almost surely \citep{Kor(73)} and  $M_{r,n}\rightarrow\ P_{\theta/r}$ almost surely \citep{Ewe(72)}, where $P_{\theta/r}$ is a Poisson r.v. with parameter $\theta/r$. Equation \eqref{eq:sigma_diversity} shows that $K_{n}$, for large $n$, grows as $n^{\alpha}$. This is precisely the growth of the number of distinct species in $n\geq1$ random samples from a power-law distribution of exponent $c=\alpha^{-1}$. Moreover, by combining \eqref{eq:sigma_diversity} and \eqref{eq:sigma_diversity_m}, it holds that $p_{\alpha,r}=\alpha(1-\alpha)_{(r-1)}/r!$ is the large $n$ asymptotic proportion of the number of distinct species with frequency $r$. Therefore, $ p_{\alpha,r}\simeq c_{\alpha}r^{-\alpha -1}$ for large $r$, for a constant $c_{\alpha}$. This is precisely the distribution of the number of distinct species with frequency $r$  in $n\geq1$ random samples from a power-law distribution of exponent $c=\alpha^{-1}$.


\section{Coverage probabilities}\label{sec2}

The estimation of coverage probabilities, or rare probabilities, dates back to the early work of Alan M. Turing and Irving J. Good at Bletchley Park in 1940s \citep{Goo(53)}. Let $\mathbf{X}_{n}$ be a random sample under the ``classical framework" for SSPs, and denote by $(N_{j,n})_{j\geq1}$ the species' frequencies in $\mathbf{X}_{n}$. The coverage probability of order $r\geq0$ is
\begin{displaymath}
\mathfrak{p}_{r,n}=\sum_{j\geq1}p_{j}I(N_{j,n}=r),
\end{displaymath}
namely the probability mass of species observed with frequency $r\geq0$ in the sample. Of special interest is $\mathfrak{p}_{0,n}$, namely the coverage probability of order $0$, which is referred to as the missing mass.  The problem of estimating $\mathfrak{p}_{r,n}$ first appeared in ecology \citep{Fis(43),Goo(53),Cha(92),Bun(93)}, and over the past three decades its importance has grown in biological sciences \citep{Kro(99),Mao(04),Gao(07)}. Coverage probabilities arise in DNA sequencing data in the form of sample coverage, or saturation, and frequency estimation. Sample coverage, i.e. the proportion of molecules in an (infinite) population that are observed in the sample, is related to the estimation of the population abundance of unobserved molecules (the missing mass), as it is equal to $1-\mathfrak{p}_{0,n}$. An accurate estimation of the sample coverage allows to determine if the sample is saturated, i.e. all species have been sampled. Sequencing experiments with high sample coverage are crucial to avoid sampling bias and to produce robust findings. Given that high sample coverage is often an issue in degraded DNA samples, in metagenomics, such as the study of the microbiome, and in single-cell DNA sequencing, where sample coverage varies in each cell, a correct sample coverage estimation remains critical \citep[Section 4]{Den(19)}.

\begin{rmk}\label{rmk_cov}
Besides biological and physical sciences, the problem of estimating coverage probabilities, and in particular the estimation of the missing mass, has found application in many scientific fields: i) statistical machine learning, e.g., estimation of node degrees of networks based on source-destination data \citep{Zha(05)}, optimal discovery with probabilistic expert advice \citep{Bub(13)} and frequency recovery from sketched data through random hashing \citep{Cai(18)}; ii) theoretical computer science, in the context of recovering from sketches the number of distinct species through probabilistic counting algorithms \citep{Mot(06)}; iii) information theory, in the context of universal (lossless) compression of sequences over arbitrarily large alphabets \citep{Orl(04),Ben(18)}; iv) empirical linguistics, e.g., estimation of the size of a vocabulary \citep{Gal(95)} and $m$-gram language modeling in natural language processing \citep{Oha(12)}; v) forensic DNA analysis, in the context of rare-type matching problem \citep{Ane(17),Cer(17),FN(24)}.
\end{rmk}

\subsection{A nonparametric estimator of $\mathfrak{p}_{r,n}$}

Under the ``classical framework", the Good-Turing estimator is the most popular estimator of $\mathfrak{p}_{r,n}$ \citep{Goo(53),Goo(56),Rob(56),Rob(68)}. Denoting by $M_{r,n}=m_{r}$ the number of distinct species with frequency $r\geq1$ in $\mathbf{X}_{n}$, the Good-Turing estimator is given by
\begin{equation}\label{gt_estim}
\tilde{\mathfrak{p}}_{r,n}=(r+1)\frac{m_{r+1}}{n}.
\end{equation}
The Good-Turing estimator is a nonparametric estimator of $\mathfrak{p}_{r,n}$, in the sense that it does not rely on any distributional assumption on the unknown $p$. It has a straightforward heuristic derivation, which relies on a comparison between the expectations of $\mathfrak{p}_{r,n}$ and $M_{r,n}$, for $r\geq0$, that is
\begin{displaymath}
\E[\mathfrak{p}_{r,n}]=\frac{r+1}{n+1}\E[M_{r+1,n+1}],
\end{displaymath}
and then the approximation of $(n+1)^{-1}\E[M_{r+1,n+1}]$ with the observable $n^{-1}m_{r}$, assuming $n$ large enough \citep{Goo(53),Rob(68)}. The estimator $\tilde{\mathfrak{p}}_{r,n}$ also admits a nonparametric empirical Bayes interpretation \citep{Rob(56),Rob(64)}, that is $\tilde{\mathfrak{p}}_{r,n}$ may be viewed as a posterior expectation with respect to a nonparametric prior estimated from the sample \citep{Efr(76),Efr(03)}. The Good-Turing estimator has been the subject of numerous studies, e.g., central limit theorems, local limit theorems and large deviation principles \citep{Est(82),Est(83),Zha(09), Gao13, Gra(17)}, admissibility and concentration properties \citep{Rob(68),McA(00), Oha(12), Ben(17),Sko(20)}, consistency and convergence rates \citep{McA(03), Mos(15), Fad(18)}, and optimality through minimax lower bounds \citep{Orl(03),Ach(18),Fad(18)}.

\subsection{BNP inference for coverage probabilities}\label{sec21}

In the ``BNP framework" \eqref{eq:exchangeable_model}, \citet{Arb(17)} computed the posterior distribution of $\mathfrak{p}_{r,n}$, given $\mathbf{X}_{n}$. In particular, assume that the random sample $\mathbf{X}_{n}$ features $K_{n}=k$ distinct species with frequencies $\mathbf{N}_{n}=(n_{1},\ldots,n_{k})$, such that $M_{r,n}=m_{r}$ for $r\geq1$, and let $\text{Beta}(a,\,b)$ be the Beta distribution with parameter $(a,\,b)$, for $a,\,b>0$. Then,
\begin{equation}\label{post_miss}
\mathfrak{p}_{0,n}\,|\,\mathbf{X}_{n}\sim\text{Beta}(\theta+\alpha k,\,n-\alpha k),
\end{equation}
and for $r\geq1$
\begin{equation}\label{post_rare}
\mathfrak{p}_{r,n}\,|\,\mathbf{X}_{n}\sim\text{Beta}((r-\alpha)m_{r},\,\theta+n-(r-\alpha)m_{r}).
\end{equation}
According to \eqref{post_miss} and \eqref{post_rare}, for fixed $\alpha\in(0,1)$ and $\theta>-\alpha$, $K_{n}=k$ is a sufficient statistic to make inference on $\mathfrak{p}_{0,n}$, whereas $M_{r,n}=m_{r}$ is a sufficient statistic to make inference on $\mathfrak{p}_{r,n}$. Differently, if $\alpha=0$ then $n$ and $M_{r,n}$ are sufficient statistics to infer $\mathfrak{p}_{0,n}$ and $\mathfrak{p}_{r,n}$, respectively. Besides leading to BNP estimates of $\mathfrak{p}_{n,r}$,  \eqref{post_miss} and \eqref{post_rare} are critical to quantify uncertainty of estimates by means of credible intervals by means of, e.g., concentration inequalities  \citep{Sko(21)} and Monte Carlo sampling \citet{Arb(17)}. This is possible in practice because of the simple form of the posterior distribution, which allows to exploit well known properties of the Beta distribution to deal with BNP inference for coverage probabilities under the PYP prior.

Under the assumption of a squared loss function, BNP estimators of $\mathfrak{p}_{0,n}$ and $\mathfrak{p}_{r,n}$ follows directly from \eqref{post_miss} and \eqref{post_rare}, respectively, by taking corresponding expected values. That is,
\begin{equation}\label{est_miss}
\hat{\mathfrak{p}}_{0,n}=\E[\mathfrak{p}_{0,n}\,|\,\mathbf{X}_{n}]=\frac{\theta+k\alpha}{\theta+n},
\end{equation}
and for $r\geq1$
\begin{equation}\label{est_rare}
\mathfrak{\hat{p}}_{r,n}=\E[\mathfrak{p}_{r,n}\,|\,\mathbf{X}_{n}]=(r-\alpha)\frac{m_{r}}{\theta+n}.
\end{equation}
The BNP estimators \eqref{est_miss} and \eqref{est_rare} first appeared in \citet{Lij(07)} and \citet{Fav(12)}, where they are obtained as means of a direct application of the predictive probability \eqref{eq:pred} of the PYP prior. See also \citet{Deb(15)}. In particular, by combining the definition of $\mathfrak{p}_{0,n}$ and predictive probabilities, the BNP estimator of $\mathfrak{p}_{0,n}$ under a squared loss function is precisely the probability that the $(n+1)$-th draw belongs to a new species, i.e. a species not observed in the sample; this is  probability $(\theta+k\alpha)/(\theta+n)$ attached to $\nu$ in \eqref{eq:pred}. Similarly, by combining the definitions of $\mathfrak{p}_{r,n}$ and predictive probabilities, the BNP estimator of $\mathfrak{p}_{r,n}$ under a squared loss function is precisely the probability that the $(n+1)-th$ draw belongs to a species observed with frequency $r$ in the sample; this is the probability $(r-\alpha)/(\theta+n)$ attached to the empirical part of \eqref{eq:pred}, i.e. the probability of observing a specific species with frequency $r$ in the sample, multiplied by the number $m_{r}$ of species with frequency $r$ in the sample. 

A peculiar feature of the Good-Turing estimator \eqref{gt_estim}, which arises from its heuristic derivation, is that it depends on $m_{r+1}$, and not on $m_{r}$ as one would intuitively expect for an estimator of $\mathfrak{p}_{r,n}$. This is contrast with the BNP estimator. Such a feature, combined with the irregular behaviour of the $m_{r}$'s, may lead to absurd estimates, the most common being $\mathfrak{\tilde{p}}_{r,n}=0$ when $m_{r}>0$ and $m_{r+1}=0$.  To overcome this drawback, \citet{Goo(53)} suggested to smooth the estimator \eqref{gt_estim} by replacing the $m_{r}$'s with more regular $m^{\prime}_{r}$'s, with $m^{\prime}_{r}$ being, e.g., a suitable parabolic function of $r$, a proportion of the number $k$ of distinct species in the sample, the expectation with respect to a suitable parametric model. The BNP estimator $\mathfrak{\hat{p}}_{r,n}$ may be interpreted as a smoothed Good-Turing estimator, where the smoothing is induced by the PYP prior. In particular, let $a_{n}\simeq b_{n}$ mean that $\lim_{n\rightarrow+\infty}a_{n}/b_{n}=1$, namely $a_{n}$ and $b_{n}$ are asymptotically equivalent as $n$ tends to infinity. \citet[Theorem 1]{Fav(16)} show that, as $n\rightarrow+\infty$, for $r\geq0$
\begin{equation}\label{smooth_gt}
\mathfrak{\hat{p}}_{r,n}\simeq(r+1)\frac{m^{\prime}_{r+1}}{n},
\end{equation}
where
\begin{equation}\label{smoother}
m^{\prime}_{r+1}=\frac{\alpha(1-\alpha)_{(r)}}{(r+1)!}k.
\end{equation}
According to \eqref{smooth_gt}, the BNP estimator $\mathfrak{\hat{p}}_{r,n}$ is asymptotically equivalent, for large $n$, to a smoothed Good-Turing estimator, where the smoothed $m^{\prime}_{r}$ in \eqref{smoother} is the proportion $\alpha(1-\alpha)_{(r)}/(r+1)!$ of the number $k$ of species in the sample. While smoothing techniques for the Good-Turing estimator were introduced as an ad-hoc tool for post-processing the $m_{r}$'s in order to improve the performance of $\mathfrak{\tilde{p}}_{r,n}$, \citet{Fav(16)} show that smoothing emerges naturally from a BNP approach to estimate $\mathfrak{p}_{r,n}$. We refer to \citet{Arb(17)} for high-order refinements of \eqref{smooth_gt}.


\section{The number of unseen species}\label{sec3}

The unseen-species problem is an $m$-step ahead generalization the problem of estimating the missing mass. For $m\geq1$ let $\mathbf{X}_{n+m}=(X_{1},\ldots,X_{n},X_{n+1},\ldots,X_{n+m})$ be a random sample under the ``classical framework" for SSPs, of which only the first $n$ elements are assumed to be observed, and denote by $(N_{j,n})_{j\geq1}$ and $(N_{j,m})_{j\geq1}$ the species' frequencies in $\mathbf{X}_{n}$ and $(X_{n+1},\ldots,X_{n+m})$, respectively. The number of unseen (distinct) species is defined as
\begin{displaymath}
\mathfrak{u}_{n,m}=\sum_{j\geq1}I(N_{j,n}=0)I(N_{j,m}>0),
\end{displaymath}
namely the number of hitherto unseen (distinct) species that would be observed if $m$ additional samples were collected from the same population. As the missing mass, the unseen-species problem first appeared in ecology  \citep{Fis(43),Goo(56),Cha(84),Cha(92)}, and over the past three decades its importance has grown in biological and physical sciences. In molecular biological data, the unseen-species problem arises in the estimation of the complexity of sequencing libraries \citep{Dal(13),Ion(09)}. While in low-complexity libraries a large proportion of the sample is composed by only a small number of unique molecules, high-complexity libraries usually display a large number of molecules, providing more information for a fixed level of sequencing. Hence, high-complexity libraries are often preferred by researchers. To evaluate library complexity, the complexity curve, or Species Accumulation Curve (SAC), is defined as the number of additional species that are observed as the sampling effort increases. This is the number of unseen species, interpreted as a function of $m$. The problem of library complexity estimation plays a crucial role to predict the benefit of additional sequencing, and optimize resource in the planning stages of experiments \citep[Section 3]{Den(19)}.

\begin{rmk}\label{rmk_uns}
As a generalization of the problem of estimating the missing mass, the unseen-species problem has appeared in some of the fields described in Remark \ref{rmk_cov}, with interest in statistical machine learning and theoretical computer science \citep{Haa(95),Flo(07),Hao(20a)}, and in empirical linguistics and natural language processing  \citep{Thi(87),Gal(95), Oha(12)}. In principle, we may say that all the statistical problems described in Remark \ref{rmk_cov} admit a natural extension in which $m>1$ additional unobserved sample are considered.
\end{rmk}

\subsection{A nonparametric estimator of $\mathfrak{u}_{n,m}$}

Under  the ``classical framework", the Good-Toulmin estimator is the most popular estimator of $\mathfrak{u}_{n,m}$ \citep{Goo(56),Efr(76)}. If $\lambda=m/n$ and $M_{r,n}=m_{r}$ denotes the number of distinct species with frequency $r\geq1$ in $\mathbf{X}_{n}$, then the Good-Toulmin estimator is
\begin{equation}\label{gtou_estim}
\mathfrak{\tilde{u}}_{n,m}=\sum_{i\geq1}(-1)^{i+1}\lambda^{i}m_{i}.
\end{equation}
For $\lambda=n^{-1}$ the estimator \eqref{gtou_estim} reduces to the Good-Turing estimator of $\mathfrak{p}_{0,n}$. The Good-Toulmin estimator \eqref{gtou_estim} is a nonparametric estimator of $\mathfrak{u}_{n,m}$, as it does not rely on any distributional assumption. The estimator $\mathfrak{\tilde{u}}_{n,m}$ was first obtained by \citet{Goo(56)} through a comparison between the expectations of $\mathfrak{u}_{n,m}$ and $M_{r,n}$, in analogy with the Good-Turing estimator in \citet{Goo(53)}, whereas \citet{Efr(76)} proved that $\mathfrak{\tilde{u}}_{n,m}$ admits a nonparametric empirical Bayes derivation \citep{Mao(02)}. In particular, \citet{Efr(76)} observed that, for $\lambda\geq1$, the geometrically increasing magnitude of $\lambda^{i}$ produces some wild oscillations in \eqref{gtou_estim}, which are undesirable. To overcome this drawback, they proposed a modification of $\mathfrak{\tilde{u}}_{n,m}$ that relies on a Euler-type random truncation of the series \eqref{gtou_estim}. Motivated by the increasing interest in the range $\lambda>1$, especially in biological applications, Efron-Thisted estimator has been the subject of breakthrough studies \citep{Orl(17), Wu(19),Pol(20)}.  \citet{Orl(17)} showed that Efron-Thisted estimator provably estimates $\mathfrak{u}_{n,m}$ all of the way up to $\lambda\asymp\log n$, that such a range is the best possible, and that the estimator's mean-square error is minimax near-optimal for any $\lambda$. These theoretical guarantees do not rely on any assumption on the underlying unknown distribution $p$ for the $X_{i}$'s, and they hold uniformly over all discrete distributions, thus providing a theory in its greatest generality.

A generalization, or refinement, of the number of unseen species $\mathfrak{u}_{n,m}$ is the unseen species' prevalences. Formally, the unseen species' prevalence of order $r\geq1$ is defined as
\begin{displaymath}
\mathfrak{u}_{r,n,m}=\sum_{j\geq1}I(N_{j,n}=0)I(N_{j,m}=r).
\end{displaymath}
That is, $\mathfrak{u}_{r,n,m}$ is the number of hitherto unseen species that would be observed with frequency $r$ if $m$ additional samples were collected from the same population. For small values of $r$, the unseen species' prevalence of order $r$ is also referred to as the number of unseen rare species. This is a critical quantity for the understanding of the species composition of the unobserved individuals. In ecology and biology, for instance, conservation of biodiversity requires a careful control of the number of species with frequency less than a certain threshold, namely rare species \citep{Mag(03),Tho(04)}. In genomics, rare species represent a critical issue, the reason being that species that appear only once or twice are often associated with deleterious diseases \citep{Lai(10)}. For $\lambda<1$, a nonparametric estimator of $\mathfrak{u}_{r,n,m}$ may be obtained by comparing expectations of $\mathfrak{u}_{r,n,m}$ and $M_{r,n}$, which leads to
\begin{equation}\label{gtou_estim_gen}
\mathfrak{\tilde{u}}_{r,n,m}=\sum_{i\geq1}(-\lambda)^{i+r-1}{r+i-1\choose i-1}m_{i+r-1}.
\end{equation}
For $\lambda\geq1$, \citet{Hao(20)} proposed a modification of $\mathfrak{\tilde{u}}_{r,n,m}$ in the spirit of Efron-Thisted estimator. In particular, along the same lines of te original work of \citet{Orl(17)}, \citet{Hao(20)} show that their estimator provably estimates $\mathfrak{u}_{r,n,m}$ all of the way up to $\lambda\asymp r^{-1}\log n$, that such a range is the best possible, and that the estimator's mean-square error is minimax near-optimal.

\subsection{BNP inference for the number of unseen species}\label{sec31}

In the ``BNP framework" \eqref{eq:exchangeable_model}, \citet{Lij(07)} computed the posterior distribution of $\mathfrak{u}_{n,m}$ given $\mathbf{X}_{n}$. In particular, assume that the random sample $\mathbf{X}_{n}$ features $K_{n}=k$ distinct species with frequencies $\mathbf{N}_{n}=(n_{1},\ldots,n_{k})$, such that $M_{r,n}=m_{r}$ for $r\geq1$. If $\alpha\in(0,1)$ then for $x\in\{0,1,\ldots,m\}$
\begin{equation}\label{post_py_k}
\text{Pr}[\mathfrak{u}_{n,m}=x\,|\,\mathbf{X}_{n}]=\frac{\left(k+\frac{\theta}{\alpha}\right)_{(x)}}{(\theta+n)_{(m)}}\mathscr{C}(m,x;\alpha,-n+k\alpha).
\end{equation}
Under the assumption of a squared loss function, a BNP estimator of $\mathfrak{u}_{n,m}$ \citep{Fav(09)} is the expected value of \eqref{post_py_k}, i.e.,
\begin{equation}\label{est_py_k}
\mathfrak{\hat{u}}_{n,m}=\E[\mathfrak{u}_{n,m}\,|\,\mathbf{X}_{n}]=\left(k+\frac{\theta}{\alpha}\right)\left(\frac{(\theta+n+\alpha)_{(m)}}{(\theta+n)_{(m)}}-1\right).
\end{equation}
For $\alpha=0$, i.e. under the DP, the posterior distribution and the BNP estimator of $\mathfrak{u}_{n,m}$ are obtained from \eqref{post_py_k} and \eqref{est_py_k}, respectively, by letting $\alpha\rightarrow0$ \citep{Lij(07)}. In particular, for $b\geq0$ let $|s(u,v;b)|$ be the non-centered signless Stirling number of the first type, which arises from $\mathscr{C}(u,v;a,b)$ by means of $|s(u,v;b)|=\lim_{a\rightarrow0}a^{-v}\mathscr{C}(u,v;a,b)$. If $\alpha=0$, then for $x\in\{0,1,\ldots,m\}$
\begin{equation}\label{post_dp_k}
\text{Pr}[\mathfrak{u}_{n,m}=x\,|\,\mathbf{X}_{n}]=\frac{\theta^{k}}{(\theta+n)_{(m)}}|s(m,x;n)|
\end{equation}
and
\begin{equation}\label{est_dp_k}
\mathfrak{\hat{u}}_{n,m}=\E[\mathfrak{u}_{n,m}\,|\,\mathbf{X}_{n}]=\sum_{i=1}^{m}\frac{\theta}{\theta+n+i-1}.
\end{equation}
For $m=1$ the estimator \eqref{est_py_k} reduces to the estimator \eqref{est_miss} of $\mathfrak{p}_{0,n}$. As a generalization of the estimators \eqref{est_py_k} and \eqref{est_dp_k}, \citet{Fav(13)} introduced a BNP estimator of the unseen species' prevalence $\mathfrak{u}_{r,n,m}$ \citep{Deb(15)}. However, to date no closed-form expressions are available for the posterior distribution of $\mathfrak{u}_{r,n,m}$, given the sample $\mathbf{X}_{n}$.

According to \eqref{post_py_k}, for any fixed $\alpha\in(0,1)$ and $\theta>-\alpha$, $K_{n}=k$ is a sufficient statistic for estimating $\mathfrak{u}_{n,m}$. If $\alpha=0$, then the sample size $n$ is a sufficient statistic to infer $\mathfrak{u}_{n,m}$. The posterior distributions \eqref{post_py_k} and \eqref{post_dp_k} are critical to quantify uncertainty of  \eqref{est_py_k} and \eqref{est_dp_k} by means of credible intervals. In practice, Monte Carlo sampling of \eqref{post_py_k} and \eqref{post_dp_k} is doable for small values of $n$ and $m$, and it becomes impossible for large values of $n$ or $m$. This is because for large $n$ and $m$, and even only for large $m$, the computational burden for evaluating generalized factorial coefficients and Stirling numbers becomes overwhelming. To overcome this drawback, \citet{Fav(09)} proposed large $m$ approximations of the posterior distributions \eqref{post_py_k} and \eqref{post_dp_k}. Let $\mathfrak{u}_{n,m}(k)$ denote a r.v. whose distribution is \eqref{post_py_k} for $\alpha\in(0,1)$ and \eqref{post_dp_k} for $\alpha=0$. For $\alpha\in(0,1)$, \citet[Proposition 2]{Fav(09)} show that, as $m\rightarrow+\infty$,
\begin{equation}\label{asimp_py}
\frac{\mathfrak{u}_{n,m}(k)}{m^{\alpha}}\rightarrow B_{\frac{\theta}{\alpha}+k,\,\frac{n}{\alpha}-k}S_{\alpha,\theta+n}\quad\text{almost surely},
\end{equation}
where $B_{\theta/\alpha+k,n/\alpha-k}$ is a Beta r.v. with parameter $(\theta/\alpha+k,n/\alpha-k)$ and $S_{\alpha,\theta+n}$ is the Pitman's $\alpha$-diversity with parameter $(\alpha,\theta+n)$, with $B_{\theta/\alpha+k,n/\alpha-k}$ being independent of $S_{\alpha,\theta+n}$ \citep{DF(20a)}. For $\alpha=0$, as $m\rightarrow+\infty$ 
\begin{equation}\label{asimp_dp}
\frac{\mathfrak{u}_{n,m}(k)}{\log(m)}\rightarrow(\theta+n) \quad\text{almost surely}.
\end{equation}
The large $m$ asymptotics \eqref{asimp_py} of $\mathfrak{u}_{n,m}(k)$ may be viewed as a posterior counterpart of the large $n$ asymptotic behaviour of $K_{n}$ in \eqref{eq:sigma_diversity}; the limiting r.v. \eqref{asimp_py} is referred to as posterior Pitman's $\alpha$-diversity. See \citet{Fav(13)} for a generalization of \eqref{asimp_py} to the unseen species' prevalences $\mathfrak{u}_{r,n,m}$.

\citet{Fav(09)} introduced a Monte Carlo scheme for sampling the posterior $\alpha$-diversity, and applied it to obtain a large $m$ approximation of credible intervals for the BNP estimate \eqref{est_py_k}. The critical step consists in sampling the Pitman's $\alpha$-diversity $S_{\alpha,\theta+n}$, whose density function is \eqref{adivers} with $\theta+n$ in place of $\theta$. This problem boils down to sample from a polynomially tilted positive $\alpha$-stable distribution. That is, if $f_{\alpha}$ denotes the density function of a positive $\alpha$-stable r.v., then the distribution of $S_{\alpha,\theta+n}^{-1/\alpha}$ has the polynomially tilted positive $\alpha$-stable density function
\begin{displaymath}
f_{S_{\alpha,\theta+n}^{-1/\alpha}}(s)\propto s^{-(\theta+n)}f_{\alpha}(s).
\end{displaymath}
The Monte Carlo approach of \citet{Fav(09)} is suitable for scenarios where $n$ is not large, and $m$ is much more large than $n$. This is because: i) for large $n$ the computational burden for the sampling from a polynomially tilted positive $\alpha$-stable distribution becomes overwhelming \citep{Dev(09),Hof(11)}; ii) the large $m$ asymptotics \eqref{asimp_py} is of a qualitative nature, in the sense that it does not quantify the error in approximating the posterior distribution \eqref{post_py_k} with the distribution of the limiting r.v. \eqref{asimp_py}. In the next proposition, we introduce a representation of the posterior distributions \eqref{post_py_k} and \eqref{post_dp_k}. Besides providing an intuitive interpretation of \eqref{post_py_k} in terms of compound Binomial distributions, it leads to a straightforward Monte Carlo sampling from \eqref{post_py_k} and \eqref{post_dp_k}, for arbitrarily large $n$ and $m$. An analogous representation is introduced for the posterior distribution of $\mathfrak{u}_{r,n,m}$, thus filling a gap in the literature. We denote by $\text{Binomial}(n,\,p)$ a Binomial distribution with parameter $(n,p)$, $n\in\mathbb{N}$ and $p\in(0,1)$.

\begin{prp}\label{prp1}
Let $\mathbf{X}_{n}$ be a random sample from $P\sim\text{PYP}(\alpha,\theta)$, for $\alpha\in[0,1)$ and $\theta>-\alpha$, such that $\mathbf{X}_{n}$ features $K_{n}=k$ and $\mathbf{N}_{n}=(n_{1},\ldots,n_{k})$. Moreover, denote by $K^{\ast}_{m}$ and $M^{\ast}_{r,m}$ the random numbers of distinct species and distinct species with frequency $r\geq1$, respectively, in $m\geq1$ random samples from $P\sim\text{PYP}(\alpha,\theta+n)$. Then, 
\begin{itemize}
\item[i)] for $\alpha\in(0,1)$
\begin{equation}\label{eq_id_py}
\mathfrak{u}_{n,m}\,|\,\mathbf{X}_{n}\sim\text{Binomial}\left(K^{\ast}_{m},\,B_{\frac{\theta}{\alpha}+k,\,\frac{n}{\alpha}-k} \right),
\end{equation}
with $B_{\theta/\alpha+k,n/\alpha-k}$ being a Beta r.v. with parameter $(\theta/\alpha+k,n/\alpha-k)$, which is independent of the r.v. $K^{\ast}_{m}$;
\item[ii)] for $\alpha=0$ 
\begin{equation}\label{eq_id_dp}
\mathfrak{u}_{n,m}\,|\,\mathbf{X}_{n}\sim\text{Binomial}\left(K^{\ast}_{m},\,\frac{\theta}{\theta+n}\right);
\end{equation}
\item[iii)] for $\alpha\in(0,1)$
\begin{equation}\label{eq_id_py_r}
\mathfrak{u}_{r,n,m}\,|\,\mathbf{X}_{n}\sim\text{Binomial}\left(M^{\ast}_{r,m},\,B_{\frac{\theta}{\alpha}+k,\,\frac{n}{\alpha}-k}\right),
\end{equation}
with $B_{\theta/\alpha+k,n/\alpha-k}$ being a Beta r.v. with parameter $(\theta/\alpha+k,n/\alpha-k)$, which is independent of the r.v. $M^{\ast}_{r,m}$;
\item[iv)] for $\alpha=0$
\begin{equation}\label{eq_id_dp_r}
\mathfrak{u}_{r,n,m}\,|\,\mathbf{X}_{n}\sim \text{Binomial}\left(M^{\ast}_{r,m},\,\frac{\theta}{\theta+n}\right).
\end{equation}
\end{itemize}
\end{prp}

See Section~S3 of the Supplementary Material \citep{BFN(23)} for the proof of Proposition \ref{prp1}. Proposition \ref{prp1} is a consequence of the conjugacy and quasi conjugacy properties of the DP and the PYP, respectively \citep{Fer(73),Pit(96)}. From \eqref{eq_id_py} the posterior distribution \eqref{post_py_k} is the distribution of the number of successes in a random number $K^{\ast}_{m}$ of independent Bernoulli trials, each trial with a Beta random probability $B_{\theta/\alpha+k,n/\alpha-k}$ of success. Note that the expectation of $B_{\theta/\alpha+k,n/\alpha-k}$ is the BNP estimator of the missing mass \eqref{est_miss}. We write
\begin{displaymath}
\mathfrak{\hat{u}}_{n,m}=\frac{\theta+k\alpha}{\theta+n}\E[K^{\ast}_{m}],
\end{displaymath}
where a simple expression for $\E[K^{\ast}_{m}]$ is available in \citet[Chapter 3]{Pit(06)}. From \eqref{eq_id_dp}, the posterior distribution \eqref{post_dp_k} coincides with the distribution of the number of successes in a random number $K^{\ast}_{m}$ of independent Bernoulli trials, each trial with probability $\theta/(\theta+n)$ of success, i.e. the BNP estimator of the missing mass \eqref{est_miss}. Then,
\begin{displaymath}
\mathfrak{\hat{u}}_{n,m}=\frac{\theta}{\theta+n}\E[K^{\ast}_{m}].
\end{displaymath}
Along similar lines, from \eqref{eq_id_py_r} and \eqref{eq_id_dp_r} we obtain an alternative representation of the BNP estimator $\mathfrak{u}_{r,n,m}$ \citep{Fav(13)}, i.e.,
\begin{displaymath}
\mathfrak{\hat{u}}_{r,n,m}=\frac{\theta+k\alpha}{\theta+n}\E[M^{\ast}_{r,m}],
\end{displaymath}
where a simple closed-form expression for $\E[M^{\ast}_{r,m}]$ is available in \citet[Proposition 1]{Fav(13)}. According to \eqref{eq_id_py_r}, for fixed $\alpha\in(0,1)$ and $\theta>-\alpha$ the number $K_{n}$ of distinct species in the sample is a sufficient statistic to make inference on $\mathfrak{u}_{r,n,m}$. Moreover, if $\alpha=0$, i.e. under the DP, then the sample size $n$ is a sufficient statistic to infer $\mathfrak{u}_{r,n,m}$. 

The representations \eqref{eq_id_py} and \eqref{eq_id_dp} are useful for Monte Carlo sampling from the posterior distributions \eqref{post_py_k} and \eqref{post_dp_k}. They allow to sample from \eqref{post_py_k} and \eqref{post_dp_k} for arbitrarily large values of $n$ and $m$, thus avoiding the use of the large $m$ approximation proposed in \citet{Fav(09)}. For fixed $\alpha\in(0,1)$ and $\theta>-\alpha$, Monte Carlo sampling from \eqref{post_py_k} consists of three steps: i) sample from $\text{Beta}(\theta/\alpha+k,\,n/\alpha-k)$; ii) independently of step i), sample the r.v. $K^{\ast}_{m}$ under $P\sim\text{PYP}(\alpha,\theta+n)$; iii) given step i) and step ii), sample from $\text{Binomial}(K^{\ast}_{m},\,B_{\theta/\alpha+k,\,n/\alpha-k})$. If $\alpha=0$, i.e. under the DP, then Monte Carlo sampling of \eqref{post_dp_k} consists of two steps: i) sample the r.v. $K^{\ast}_{m}$ under $P\sim\text{PYP}(0,\theta+n)$; ii) given step i), sample from $\text{Binomial}(K^{\ast}_{m},\,\theta/(\theta+n))$. Sampling from Beta and Binomial distributions is straightforward, for arbitrarily large $n$ and $m$, and routines are available in standard software. Sampling $K^{\ast}_{m}$ is also straightforward, for arbitrarily large $n$ and $m$, and it exploits the predictive probabilities of the PYP \eqref{eq:pred}. In particular, let $\text{Bernoulli}(p)$ be the Bernoulli distribution with parameter $p$, for $p\in(0,1)$. Sampling $K^{\ast}_{m}$ then reduces to sample $m-1$ Bernoulli r.v.s:
\begin{itemize}
\item[1)] Set $k=1$;
\item[2)] For $i=1$ to $m-1$
\begin{itemize}
\item[] Set $b$ to be a sample from $\text{Bernoulli}((\theta+n+\alpha k)/(\theta+n+i))$;
\item[] Set $k=k+b$;
\end{itemize}
\item[3)] Return $k$.
\end{itemize}
Along the same lines, the representations \eqref{eq_id_py_r} and \eqref{eq_id_dp_r} are useful for Monte Carlo sampling from the posterior distribution of $\mathfrak{u}_{r,n,m}$, given $\mathbf{X}_{n}$. In particular, they require to sample the r.v. $M^{\ast}_{r,m}$ rather than the r.v. $K^{\ast}_{m}$. Sampling of $M^{\ast}_{r,m}$ still exploits the predictive probabilities of the PYP although, regrettably, it does not reduce to sample Bernoulli r.v.s.


\section{Coverages of prevalences}\label{sec4}

The estimation of coverages of prevalences, or saturations, may be viewed as the $m$-step ahead generalization of the problem of estimating coverage probabilities. For $m\geq1$ let $\mathbf{X}_{n+m}=(X_{1},\ldots,X_{n},X_{n+1},\ldots,X_{n+m})$ be a random sample under the ``classical framework" for SSPs, of which only the first $n$ elements are assumed to be observed, and denote by $(N_{j,n})_{j\geq1}$ and $(N_{j,m})_{j\geq1}$ the species' frequencies in $\mathbf{X}_{n}$ and $(X_{n+1},\ldots,X_{n+m})$, respectively. The coverage of prevalence of order $r\geq0$ is defined as
\begin{displaymath}
\mathfrak{f}_{r,n,m}=\sum_{j\geq1}I(N_{j,n}=r)I(N_{j,m}>0),
\end{displaymath}
namely the number of species observed $r$ times that would be observed if $m$ additional samples were collected from the same population. Note that $\mathfrak{f}_{0,n,m}$, i.e. the coverage of prevalence of order $0$, is the number of unseen species $\mathfrak{u}_{n,m}$. The estimation of coverages of prevalences first appeared in linguistics to answer the question ``Did Shakespeare write a newly-discovered poem?" \citep{Thi(87)}. The estimated coverages of prevalences was applied to test the consistency of the word usage in a previously unknown poem attributed to Shakespeare with the word usage in the entire Shakespearean canon. Similar questions, though under different vests, appear in biological sciences. In genomics data, they appear in relation to coverage depth, i.e. the average number of reads that are aligned to known reference bases \citep{Den(19)}. Depending on the application, different levels of coverage might be required. This leads to determine whether additional sequencing is needed, which can be motivated by a required minimum coverage threshold, or investigating rare events. The estimation of coverages of prevalences can be valuable in such a setting, as it allows researchers to estimate how many molecules observed with a given frequency would be observed again if the sampling efforts were increased. 

\subsection{A nonparametric estimator of $\mathfrak{f}_{r,n,m}$}

Under  the ``classical framework", \citet{Thi(87)} introduced an estimator of $\mathfrak{f}_{r,n,m}$. In particular, if $\lambda=m/n$ and $M_{r,n}=m_{r}$ denotes the number of distinct species with frequency $r\geq1$ in $\mathbf{X}_{n}$, then the estimator is given by
\begin{equation}\label{est_prev}
\tilde{\mathfrak{f}}_{r,n,m}=\sum_{j\geq1}(-1)^{i+1}\lambda^{i}{r+i\choose i}m_{r+i}.
\end{equation}
For $\lambda=n^{-1}$, the estimator \eqref{est_prev} reduces to the Good-Turing estimator of coverage probability $\mathfrak{p}_{r,n}$. The estimator \eqref{est_prev} is a nonparametric estimator of $\mathfrak{f}_{r,n,m}$, in the sense that it does not rely on any distributional assumption on the unknown $p$. It was obtained in \citet{Thi(87)} through a nonparametric empirical Bayes derivation, though an heuristic derivation by a comparison between the expectations of $\mathfrak{u}_{n,m}$ and $M_{r,n}$ is also possible. The study of provable guarantees of $\tilde{\mathfrak{f}}_{r,n,m}$ has not yet been considered in the literature. \citet{Thi(87)} showed that $\tilde{\mathfrak{f}}_{r,n,m}$ empirically estimates $\mathfrak{f}_{r,n,m}$ for $\lambda<1$, but without provable guarantees. For $\lambda>1$ we expect that $\tilde{\mathfrak{f}}_{r,n,m}$ will suffer the same variance issue of $\tilde{\mathfrak{u}}_{n,m}$, that is the  geometrically increasing magnitude of ${r+i\choose i}\lambda^{i}$ may produce wild oscillations as the number of terms increases. Because of the additional Binomial term ${r+i\choose i}$, it is natural to expect that such a variance issue worsens as $r$ increases. As for the study of its provable guarantees, we expect that the theory developed to study minimax optimality $\tilde{\mathfrak{u}}_{n,m}$ for $\lambda\geq1$ \citep{Wu(16), Wu(19),Pol(20)} is not of a direct applicability to $\tilde{\mathfrak{f}}_{r,n,m}$.

\subsection{BNP inference for coverages of prevalences}\label{sec51}

In the ``BNP framework" \eqref{eq:exchangeable_model}, the problem of estimating the coverages of prevalences is open, and here we cover this gap. Before computing the posterior distribution of $\mathfrak{f}_{r,n,m}$, given $\mathbf{X}_{n}$, it is useful to recall the definition of  (general) hypergeometric distribution \citep[Chapter 6.2.5]{Joh(05)}, as well as the definition of generalized factorial distribution \citep[Chapter 2]{Cha(05)}. For any $u\in\mathbb{N}$ and $b,c>0$, we say that a r.v. $U_{b,c,u}$ on $\{1,\ldots,u\}$ has a generalized factorial distribution if, for $x \in \{1,\ldots,u\}$,
\begin{equation}\label{eq_urn}
\text{Pr}[U_{b,c,u}=x]=\frac{1}{(bc)_{(u)}}\mathscr{C}(u,x;b,0)(c)_{(x)};
\end{equation}
see Section~S1 of the Supplementary Material \citep{BFN(23)} for details. Moreover, for $u,v\in\mathbb{N}$ and $a>0$ such that $a>u$, a r.v. $H_{a,u,v}$ on $\{0,1,\ldots,u\}$ has a (general) hypergeometric distribution if, for $x\in{\{0,1,\ldots,u\}}$ if holds
\begin{equation}\label{gen_hyp}
\text{Pr}[H_{a,u,v}=x]=\frac{{a\choose x}{v\choose u-x}}{{a+v\choose u}}.
\end{equation}
In the next proposition, we show that the posterior distribution of $\mathfrak{f}_{r,n,m}$, given $\mathbf{X}_{n}$, admits a representation in terms of a compound (general) hypergeometric distribution.

\begin{prp}\label{prp2}
Let $\mathbf{X}_{n}$ be a random sample from $P\sim\text{PYP}(\alpha,\theta)$, for $\alpha\in[0,1)$ and $\theta>-\alpha$, such that $\mathbf{X}_{n}$ features $K_{n}=k$ and $\mathbf{N}_{n}=(n_{1},\ldots,n_{k})$. If $M_{r,n}=m_{r}$ is the number of distinct species with frequency $r\geq1$ in $\mathbf{X}_{n}$, then
\begin{equation}\label{eq_profile}
\mathfrak{f}_{r,n,m}\,|\,\mathbf{X}_{n}\,\stackrel{\text{d}}{=}\,m_{r}-\text{H}_{\frac{\theta+n}{r-\alpha}-1,m_{r},U_{r-\alpha,\frac{\theta+n}{r-\alpha},m}}.
\end{equation}
\end{prp}

See Section~S4 of the Supplementary Material \citep{BFN(23)} for the proof of Proposition \ref{prp2}. From \eqref{eq_profile}, for fixed $\alpha\in[0,1)$ and $\theta>-\alpha$, $M_{r,n}$ is a sufficient statistic to make inference on $\mathfrak{f}_{r,n,m}$. Under a squared loss function, a BNP estimator of $\mathfrak{f}_{r,n,m}$ is the expected value of \eqref{eq_profile}, i.e.,
\begin{equation}\label{eq_profile_est}
\hat{\mathfrak{f}}_{r,n,m}=\E[\mathfrak{f}_{r,n,m}\,|\,\mathbf{X}_{n}]=m_{r}\left(1-\frac{(\theta+n-r+\alpha)_{(m)}}{(\theta+n)_{(m)}}\right).
\end{equation}
See Section~S5 of the Supplementary Material \citep{BFN(23)} for the proof of Equation \ref{eq_profile_est}. For $m=1$ the estimator \eqref{eq_profile_est} reduces to the estimator \eqref{est_rare} of $\mathfrak{p}_{r,n}$. Interestingly, the estimator \eqref{eq_profile_est} may be interpreted as the proportion
\begin{displaymath}
w_{r,n,m}(\alpha,\theta)=1-\frac{(\theta+n-r+\alpha)_{(m)}}{(\theta+n)_{(m)}}\in(0,1),
\end{displaymath}
of the number $m_{r}$ of distinct species with frequency $r$. Uncertainty quantification of \eqref{eq_profile_est} is obtained by means of credible intervals via Monte Carlo sampling of the posterior distribution \eqref{eq_profile}. Monte Carlo sampling of \eqref{eq_profile} is doable for small values of $n$ and $m$, and it becomes impossible for large $n$ or $m$. This is because for large $n$ and $m$, and even only for large $m$, the computational burden for evaluating the generalized factorial coefficients in the distribution of $U_{r-\alpha,(\theta+n)/(r-\alpha),m}$. To overcome this drawback, we propose an approximation of the posterior distribution \eqref{eq_profile} for large $n$ and $m$. Let $a_{n,m}\simeq b_{n,m}$ mean that $\lim_{n\rightarrow+\infty}\lim_{m\rightarrow+\infty}a_{n,m}/b_{n,m}=1$, namely $a_{n,m}$ and $b_{n,m}$ are asymptotically equivalent as $n$ and $m$ tends to infinity. Then, as  $n,m\rightarrow+\infty$ with $m>0$, for $x\in\{0,1,\ldots,m_{r}\}$
\begin{align}\label{eq:approx_post}
&\text{Pr}[\mathfrak{f}_{r,n,m}=x\,|\,\mathbf{X}_{n}]\\
&\notag\quad\simeq {m_{r}\choose x}\left[1-\left(\frac{n}{n+m}\right)^{r-\alpha}\right]^{x}\left[\left(\frac{n}{n+m}\right)^{r-\alpha}\right]^{m_{r}-x},
\end{align}
and hence
\begin{displaymath}
\hat{\mathfrak{f}}_{r,n,m}\simeq m_{r}\left[1-\left(\frac{n}{n+m}\right)^{r-\alpha}\right].
\end{displaymath}
See Section~S6 of the Supplementary Material \citep{BFN(23)} for the proof of Equation \ref{eq:approx_post}. From \eqref{eq:approx_post}, for large $n$ and $m$ with $m>0$, the posterior distribution \eqref{eq_profile} admits a first order local approximation in terms of a Binomial distribution with parameters $(m_{r},1-(n/(n+m))^{r-\alpha})$, with $m_{r}$ being the number of trials and $1-(n/(n+m))^{r-\alpha}$ being the probability of success at the single trial.

\subsection{Disclosure risk assessment}

We conclude this section with a SSP related to the coverage of prevalence of order $1$, which first appeared in the context of disclosure risk for data confidentiality \citep{Wil(01)}. Consider a microdata sample of $n\geq1$ units (individuals) from a population of $N>n$ units, such that each unit contains identifying and sensitive information. Identifying information consists of categorical variables which might be matchable to known units of the population. A threat of disclosure results from the possibility of identifying an individual through such a matching, and hence disclose its sensitive information. To quantify disclosure risk, microdata units are partitioned according to a categorical variable formed by cross-classifying the identifying variables; that is, units are partitioned into non-empty cells containing individuals with the same combinations of values of identifying variables. Intuitively, a risk of disclosure arises from cells with frequency $1$ since, assuming no errors in matching processes or data sources, for these cells the match is guaranteed to be correct \citep{Bet(90),Ski(94),Ski(02)}. Under the ``classical framework" for SSPs, if microdata units are modeled as a random sample $\mathbf{X}_{N}$ in the classical species sampling framework, and $(N_{j,n})_{j\geq1}$ and $(N_{j,m})_{j\geq1}$ are the frequencies of cells (species) in $\mathbf{X}_{n}$ and $(X_{n+1},\ldots,X_{n+N})$, respectively, then a popular measure of disclosure risk is defined as
\begin{displaymath}
\mathfrak{d}_{n,m}=\sum_{j\geq1}I(N_{j,n}=1)I(N_{j,m}=0),
\end{displaymath}
namely the number of cells with frequency $1$ in the observed sample that are also of frequency $1$ in the whole population. We refer to \citet{Cam(20)} and \citet{Fav(21b)} for the estimation of $\mathfrak{d}_{n,m}$ within the ``classical framework" and in the ``BNP framework", respectively.


\section{Estimation of prior's parameters $(\alpha,\theta)$}\label{sec6}

In practice, a poor assessment of the prior's parameters $(\alpha,\theta)$ may harm any statistical inference based upon the $\mathrm{PYP}(\alpha,\theta)$ prior due to its rigidity. Indeed, many of the estimators derived in the previous sections critically depend on the parameter $\alpha$, and in a less extent on the parameter $\theta$. For instance, let us consider the estimator of the number of unseen species obtained in Section~\ref{sec31}. As a direct consequence of \eqref{asimp_py}, $\mathfrak{u}_{n,m} \sim Zm^{\alpha}$ as $m\to \infty$, for some r.v. $Z$. Hence, a bad choice for the parameter $\alpha$ may lead to a poor prediction. A natural way to add flexibility in the model and improve the inference is to estimate the prior's parameters $(\alpha,\theta)$. Denoting by $\Phi = \{(\alpha,\theta) \in (0,1) \times \Reals \,:\, \theta > -\alpha\}$, in this section we investigate the problem of estimating $(\alpha,\theta)$ over the set $\Phi$.

We start by considering the case where data come from the model, the so-called well-specified case. That is, we assume the ``BNP framework" \eqref{eq:exchangeable_model} for a choice of the prior's parameters $(\alpha,\theta)\in \Phi$. We study the empirical Bayes approach, in which the parameter $(\alpha,\theta)$ is estimated by maximizing the marginal likelihood function, and the hierarchical Bayes approach, where a prior distribution is placed over $\Phi$. We show that, as $n\rightarrow+\infty$, both approaches are consistent with respect to the estimation of $\alpha$. More surprising, we prove that both the approaches are inconsistent with respect to the estimation of $\theta$. These consistency and inconsistency results are also illustrated through simulations in Section~\ref{sec7}. Ultimately, we prove a minimax lower bound for the estimation of $\theta$, which proves that the inconsistency issue is fundamental, and no procedure can estimate $\theta$ with vanishing maximum risk. We also consider a more general point of view than the well-specified case, thus no longer assuming the ``BNP framework" for a choice of the prior's parameters. We consider instead the ``classical framework", namely data are modeled as a random sample from a fixed probability measure $p$ satisfying a regularity assumption enabling meaningful inference. In particular, we show that under this assumption, the sequence $((\hat\alpha_n,\hat\theta_n))_{n\geq 1}$ of marginal maximum likelihood estimators has an interpretable limit $(\alpha_{*},\theta_{*})$. We then establish the asymptotic shape of the sequence of posterior distributions in the hierarchical Bayes model. Our result shows that the sequence of posterior distributions is consistent for $\alpha_{*}$, but inconsistent for $\theta_{*}$. Thereby, we demonstrate that inference made upon the $\mathrm{PYP}$ can be robust to some form of misspecification of the model.

\subsection{Consistency and inconsistency in the well-specified case}
\label{sec:consis:wellspecified}

For the well-specified case, we consider the ``BNP framework" under the assumption that there exists a ``true" parameter $(\alpha_0,\theta_0)\in \Phi$. Specifically, $\mathbf{X}_{n}$ is random sample from $P \sim \mathrm{PYP}(\alpha_0,\theta_0)$. The goal is to analyze the large $n$ asymptotic behaviour of the empirical Bayes and hierarchical Bayes approaches under the assumption that $\mathbf{X}_n$ is distributed as the  the prior predictive distribution, here denoted $P_n^{\alpha_0,\theta_0}(\cdot) = \int P^{\otimes n}(\cdot)\intd \mathrm{PYP}_{\alpha_0,\theta_0}(P)$ in the sequel. For the empirical Bayes approach, the parameter $(\alpha,\theta)$ is estimated using a maximizer of the marginal likelihood function:
\begin{equation*}
    L_n(\alpha,\theta)%
    = n!\frac{\left(\frac{\theta}{\alpha}\right)_{\left(\sum_{i=1}^{n}M_{i,n}\right)}}{(\theta)_{(n)}}\prod_{i=1}^{n}\frac{\left(\frac{\alpha(1-\alpha)_{(i-1)}}{i!}\right)^{M_{i,n}}}{M_{i,n}!},
\end{equation*}
namely the EPSF \eqref{eq_ewe_py} as a function of the parameter $(\alpha,\theta)$ for the observed $\mathbf{M}_{n}=(M_{1,n},\ldots,M_{n,n})$. A maximizer of $L_n$ is known as the marginal maximum likelihood estimator (MMLE). The next theorem characterizes the asymptotic limit (in probability) of $((\hat\alpha_n,\hat\theta_n))_{n\geq 1}$ as $n\to \infty$.

\begin{thm}
\label{thm:wellspecified:mmle}
Assume the random sample to be such that $\mathbf{X}_{n}\sim P^{\alpha_0,\theta_0}_n$. Then, the set $\argmax_{(\alpha,\theta)\in \Phi} L_n(\alpha,\theta)$ is a non-empty set with probability $1 + o(1)$ as $n\to \infty$. Furthermore, $(\hat\alpha_n,\hat\theta_n) \in \argmax_{(\alpha,\theta)\in \Phi} L_n(\alpha,\theta)$ is such that
\begin{equation*}
    \hat\alpha_n = \alpha_0  + o_p(1)
\end{equation*}
and
\begin{equation*}
    \hat\theta_n = Z + o_p(1),
\end{equation*}
where $Z$ is a r.v. related to $S_{\alpha_0,\theta_0}$ via the relation $S_{\alpha_0,\theta_0} = \exp\{\psi(Z/\alpha_0  +1) - \alpha_0\psi(Z+1) \}$. Here, $\psi$ denotes the digamma function; that is the function $\psi$ is the derivative of $\log\Gamma$.
\end{thm}

See Section~S7 of the Supplementary Material \citep{BFN(23)} for the proof of Theorem~\ref{thm:wellspecified:mmle}. The theorem establishes that the sequence of MMLE is consistent to estimate $\alpha_0$. The limit in probability of $(\hat\theta_n)_{n\geq 1}$ is, however, a r.v. This shows that $\theta_0$ cannot be estimated consistently using the MMLE.

We continue our analysis of the well-specified case with the hierarchical Bayes approach. We put a prior distribution $G = G_{\alpha}\otimes G_{\gamma}$ over the
parameter $\alpha \in (0,1)$ and the shifted parameter $\gamma = \theta + \alpha \in (0,\infty)$. Namely, under $G$, the r.v. $\alpha$ and $\gamma$ are independent with respective marginals $G_{\alpha}$ and $G_{\gamma}$. We denote by $\Pi$ the joint distribution of $(\alpha,\gamma,P,X_1,X_2,\dots)$. The posterior distribution of $(\alpha,\gamma)$ given $\mathbf{X}_n$ is denoted $\Pi(\cdot \mid \mathbf{X}_n)$. See that by Bayes' rule:
\begin{equation*}
    \Pi\big( (\alpha,\gamma)\in A \mid \mathbf{X}_n\big)%
    = \frac{\int_A L_n(\alpha,\gamma-\alpha) \intd G(\alpha,\gamma) }{\int_{\mathbb{R}_+^2 } L_n(\alpha,\gamma-\alpha) \intd G(\alpha,\gamma)}.
\end{equation*}

\begin{thm}
  \label{thm:consistency-wellspecified-bayes}
 Assume the random sample to be such that $\mathbf{X}_{n}\sim P^{\alpha_0,\theta_0}_n$. Furthermore, assume that $G_{\alpha}$ and $G_{\gamma}$ have continuous and positive densities. Then for every $\varepsilon > 0$
  \begin{equation*}
    \Pi\big( |\alpha - \alpha_0| > \varepsilon \mid \mathbf{X}_n \big)%
    = o_p(1).
  \end{equation*}
  Moreover, there exists a r.v. $W > 0$ such that the following holds:
  \begin{equation*}
    \Pi\big(\theta \in [\theta_0 - W,\theta_0 + W]  \mid \mathbf{X}_n\big) \leq \frac{1}{2}  + o_p(1).
  \end{equation*}
\end{thm}

See Section~S8 of the Supplementary Material \citep{BFN(23)} for the proof of Theorem~\ref{thm:consistency-wellspecified-bayes}. The theorem establishes the consistency of the posterior near the ``true" $\alpha_0$, that is the posterior will eventually put all its mass on a small neighborhood of $\alpha_0$ when the sample size $n$ gets large enough. The posterior is however inconsistent at $\theta_0$: regardless of how large $n$ is taken the posterior will put mass outside of a neighborhood of fixed size [see also Section~\ref{sec7} and Figure~\ref{fig:par_estimation} for a numerical illustration]. Indeed, as demonstrated later in Theorem~\ref{thm:4sc-jzwh-mc1} the posterior distribution for $\gamma$ (equivalently $\theta$) depends on the choice of  $G_{\gamma}$ even in the asymptotic limit. Therefore, in the absence of strong prior belief on $\theta$ it is unclear if a hierarchical Bayes approach to estimate this parameter is meaningful.

Theorem~\ref{thm:wellspecified:mmle} and Theorem~\ref{thm:consistency-wellspecified-bayes} establish that the prior's parameter $\theta$ cannot be estimated consistently using the MMLE or the hierarchical Bayes approach. We conclude this section by demonstrating that the issue is more fundamental and no estimator can estimate $\theta$ with vanishing maximum risk.

\begin{thm}
    \label{thm:lecamtheta}
    For all $\alpha \in (0,1)$, for all $n \geq 1$, and for all $t > -\alpha$
    \begin{equation*}
        \inf_{\hat\theta_n}\sup_{-\alpha < \theta \leq t}\EE_{(\alpha,\theta)}\big( (\hat{\theta}_n(\mathbf{X}_n) - \theta)^2 \big)%
        \geq \frac{\min\big((t+\alpha)^2,\,t+\alpha \big)}{64}.
    \end{equation*}
    where the infimum is understood over all measurable functions of $\mathbf{X}_n$.
\end{thm}

See Section~S9 of the Supplementary Material \citep{BFN(23)} for the proof of Theorem~\ref{thm:lecamtheta}. The theorem shows that even with the knowledge of the prior's parameter $\alpha$, it is impossible for an estimator of $\theta$ to have a vanishing maximum risk over $(-\alpha,t]$ as $n\to \infty$. It also establishes that the minimax risk over $\theta \in (-\alpha,\infty)$ is infinite.

\subsection{Consistency and inconsistency in the misspecified case}
\label{sec:consis:misspecified}

Differently from the well-specified case, for the misspecified case we do not necessarily assume that $\mathbf{X}_n$ is a random sample from $P \sim \mathrm{PYP}(\alpha_0,\theta_0)$. Instead, we consider the ``classical framework", namely $\mathbf{X}_{n}$ is a random sample from a ``true" (fixed)  distribution $p$. To guarantee the existence of meaningful limits for our estimators, we require a moderate misspecification, as stated in the next assumption.

\begin{assumption}
  \label{ass:ucj-yl04-iyy}
  The distribution $p$ is discrete, such that $p = \sum_{j\geq 1}p_j \delta_{s_j}$. Furthermore, by defining the quantity
  $$\bar{F}_p(x) = \sum_{j\geq 1}I(p_j > x),$$ 
  there exist $L > 0$ and $\alpha_* \in (0,1)$ such that as $x\to 0$ it holds
  \begin{equation*}
    \bar{F}_p(x) = Lx^{-\alpha_{*}} + o\Big[\frac{1}{-x^{\alpha_*}\log (x)} \Big].
  \end{equation*}
\end{assumption}

The Assumption~\ref{ass:ucj-yl04-iyy} is motivated by the fact that if $P \sim \text{PYP}(\alpha,\theta)$, then results in \cite{Pit(03)} show that $\lim_{x \to 0}x^{\alpha}\bar{F}_P(x) = S_{\alpha,\theta}/\Gamma(1-\alpha)$ almost surely, with $\Gamma$ being the Gamma function \citep[Chapter 3 and Chapter 4]{Pit(06)}. We strengthen those results in Section~S12 of the Supplementary Material \citep{BFN(23)} and we establish a law of the iterated logarithm (LIL) as $x\to 0$ for $\bar{F}_P(x) - S_{\alpha,\theta}x^{-\alpha}/\Gamma(1-\alpha)$ when $P \sim \mathrm{PYP}(\alpha,\theta)$. In particular, the LIL implies that $P\sim \mathrm{PYP}(\alpha,\theta)$ satisfies almost-surely the Assumption~\ref{ass:ucj-yl04-iyy} with $\alpha_{*} = \alpha$ and $L = S_{\alpha,\theta}/\Gamma(1-\alpha)$. The Assumption~\ref{ass:ucj-yl04-iyy} is, however, much more general and allow for many more probability distributions. In the next theorem, we characterize the asymptotic limit of the sequence of MMLEs when the data $\mathbf{X}_n$ is independent from a distribution $p$ satisfying Assumption~\ref{ass:ucj-yl04-iyy}.

\begin{thm}
  \label{thm:misspecified:mle}
  
 Assume $\mathbf{X}_{n}$ to be a random sample from $p$, and assume that the distribution $p$ satisfies Assumption~\ref{ass:ucj-yl04-iyy}. Then the set $\argmax_{(\alpha,\theta)\in \Phi}L_n(\alpha,\theta)$ is a non-empty set with probability $1 + o(1)$ as $n\to \infty$. Furthermore, $(\hat{\alpha}_n,\hat{\theta}_{n}) \in \argmax_{(\alpha,\theta)\in \Phi}L_n(\alpha,\theta)$ is such that
  \begin{equation*}
    \hat{\alpha}_n = \alpha_{*} + o_p(1)
  \end{equation*}
  and
    \begin{equation*}
    \hat{\theta}_n = \theta_{*} + o_p(1),
  \end{equation*}
  with $\theta_{*}$ defined through
  \begin{equation*}
    L = \frac{\exp\{\psi(\theta_{*}/\alpha_{*} + 1) - \alpha_{*}\psi(\theta_{*}+1) \} }{\Gamma(1-\alpha_{*})}.
  \end{equation*}
  Here, $\psi$ is the digamma function; that is $\psi$ is the derivative of $\log\Gamma$.
\end{thm}

See Section~S10 of the Supplementary Material \citep{BFN(23)} for the proof of Theorem~\ref{thm:misspecified:mle}. Interestingly, Theorem~\ref{thm:misspecified:mle} establishes that if the
misspecification is moderate, then $(\alpha_{*},\theta_{*})$ can be directly interpreted in terms of key functionals of the ``true" data generative mechanism, that is in terms of $\alpha_{*}$ and $L$. This also shed some lights in the meaning of the parameter $(\alpha,\theta)$ when the model is correct. In the next theorem, we consider the hierarchical Bayes modeling and we characterize the asymptotic shape of the sequence of posterior distributions as $n\to \infty$.

\begin{thm}
  \label{thm:4sc-jzwh-mc1}
  Let $G = G_{\alpha} \otimes G_{\gamma}$ be a (proper or improper) prior distribution over
  the parameter $(\alpha,\gamma)$ such that $G_{\alpha}$ (respectively $G_{\gamma}$) has a density $g_{\alpha}$ (resp.
  $g_{\gamma}$) with respect to Lebesgue measure which is positive in a neighborhood
  of $\alpha_{*}$ (resp. $\gamma_{*} = \theta_{*} + \alpha_{*}$). Furthermore assume that $G_{\gamma}$ is
  such that $\int_0^{\infty}\frac{[L\Gamma(1-\alpha_{*})]^{z/\alpha_{*}}\Gamma(1-\alpha_{*} + z)}{\Gamma(z/\alpha_{*})}G_{\gamma}(\intd z) < \infty$\footnote{This is not a strong restriction since it allows to use all proper priors, as well as many improper priors.} and define a probability distribution $H_{*}$ on
  $(0,\infty)$ through
  \begin{equation*}
    H_{*}(A)%
    = \frac{\int_A \frac{[L\Gamma(1-\alpha_{*})]^{z/\alpha_{*}}\Gamma(1-\alpha_{*} + z)}{\Gamma(z/\alpha_{*})} G_{\gamma}(\intd z) }{\int_0^{\infty}\frac{[L\Gamma(1-\alpha_{*})]^{z/\alpha_{*}}\Gamma(1-\alpha_{*} + z)}{\Gamma(z/\alpha_{*})}G_{\gamma}(\intd z)}.
  \end{equation*}
  Then under Assumption~\ref{ass:ucj-yl04-iyy}, the posterior distribution of
  $(\alpha,\gamma)$ given $\mathbf{X}_n$, written here $\Pi(\cdot \mid \mathbf{X}_n)$, satisfies as $n \to \infty$
  \begin{multline*}
    \sup_{A,B}\Big|\Pi( \hat{V}_n^{1/2}(\alpha - \hat{\alpha}_n) \in A,\, \gamma \in B \mid \mathbf{X}_n  ) \\- \phi(A)H_{*}(B) \Big|%
    = o_p(1),
  \end{multline*}
  where the supremum is taken over all measurable sets, $\hat{V}_n = -\partial_{\alpha}^2\log L_n(\hat{\alpha}_n^0,0)$, and $\phi(A)$ is the
  probability that a standard normal r.v. lies in $A$.
\end{thm}

See Section~S11 of the Supplementary Material \citep{BFN(23)} for the proof of Theorem~\ref{thm:4sc-jzwh-mc1}. The theorem establishes that $(\alpha,\gamma)$  are asymptotically independent a posteriori. We see that the limiting marginal distribution for $\gamma$ is neither converging toward a point mass, nor even Gaussian. In contrast with the MMLE, the posterior distribution for $\gamma$ is difficult to interpret since it depends on the prior $G_{\gamma}$. For this reason, we shall be careful with posterior analysis involving $\gamma$ (equivalently $\theta$). The limiting marginal distribution for $\hat{V}_n^{1/2}(\alpha - \hat{\alpha}_n)$ is however Gaussian, which can be seen as a weak form of Bernstein-von Mises' (BvM) theorem. It is not a true BvM because the centering is taken at $\hat{\alpha}_n$ instead of $\alpha_{*}$ and the scaling is the ``empirical'' Fisher information instead of Fisher's information. Finally, let us mention that after writing the first draft of this paper, we have been made aware that \cite{Fra(22)} have independently obtained results similar to those presented in this section. Under an assumption resembling Assumption~\ref{ass:ucj-yl04-iyy},  \cite{Fra(22)} are able to establish a more precise asymptotic for the estimation of $\alpha$. They, however, obtain less precise results regarding the estimation of the parameter $\theta$.


\section{Numerical illustrations}\label{sec7}

In this section, we provide a numerical illustration of the prior parameters estimation and of the inference on the presented species sampling problems, using synthetic data.

We first compare the empirical Bayes approach and the fully Bayes approach by  using several prior distributions, and we demonstrate how the effect of the estimation method changes with increasing sample sizes. In particular, we compare different prior distributions on $\alpha$ and $\theta$, constructed using different combination of non-informative and (mis-specified) informative priors: (a) a non-informative prior for both parameters; (b) an informative prior for both parameters; (c) an informative prior for only $\theta$ and non-informative for $\alpha$; (d) an informative prior for only $\alpha$ and non-informative for $\theta$. See Section~S13 of the Supplementary Materials \citep{BFN(23)} for additional details.
Figure~\ref{fig:par_estimation} displays the mean error across several simulated datasets, for the different estimation methods as a function of the increasing sample size. 
While the mis-specified informative priors on $\alpha$ have a detrimental effect on the estimation, this effect gets less strong with the increasing sample size. On the contrary, the effect of mis-specified priors on $\theta$ remains large even with large sample size. Moreover, Figure~\ref{fig:par_estimation} shows that while the relative error for $\alpha$ vanishes with increasing sample size for all the estimation methods, the error for $\theta$ does not. This is a clear illustration of the findings in Section~\ref{sec:consis:wellspecified}. We then demonstrate the inference and prediction of the various SSP presented in this work, using synthetic data generated from both from the PYP and from a power law Zipf distribution. We compare the estimation and prediction under the Empirical Bayes (EB) approach, and the Full Bayes (FB) approach with non-informative priors on both $\theta$ and $\alpha$, and informative prior for both parameters.

\begin{figure}
\includegraphics[width=0.45\textwidth]{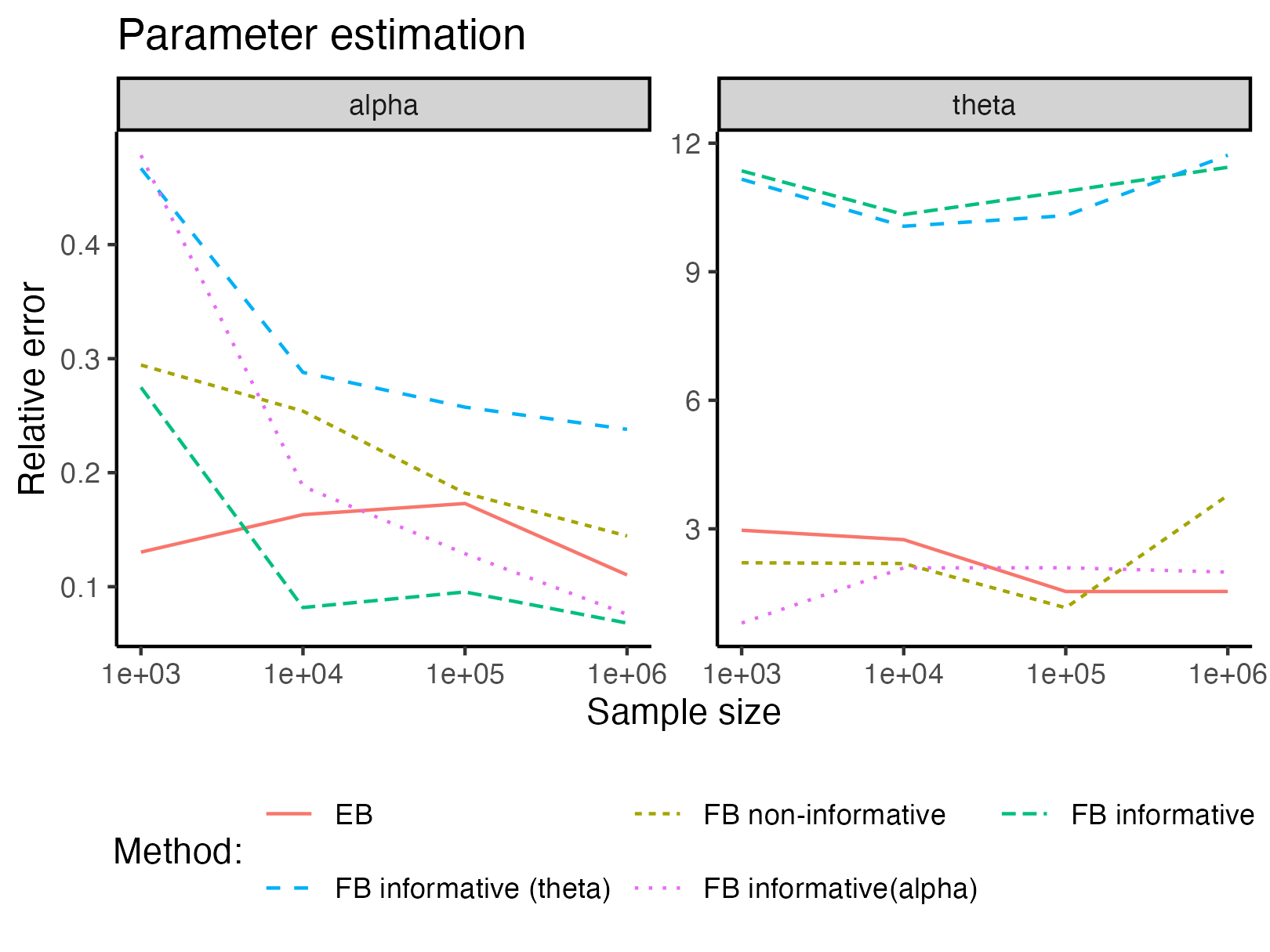}\\
\centering
\caption{Mean absolute relative (percentage) error for estimation of $\alpha$ and $\theta$ under the Empirical Bayes (EB) approach and the Full Bayes (FB) approach with several prior distributions. \label{fig:par_estimation}}
\end{figure}

We first analyze the SSPs that depend only on the initial sample and on the ``true" distribution $P$: the missing mass and the coverage probability $\mathfrak{p}_{r,n}$ for $r=1$. Figure~S1 of the Supplementary Materials \citep{BFN(23)} shows the relative (percentage) error for these functionals across several synthetic datasets. The median absolute percentage errors are less than $7\%$ for the missing mass and the coverage probability for the PY-generated data, and less than $12\%$ for the Zipf-generated data, demonstrating good recovery of the true functionals. 
We then study the ``predictive'' SSP (the ones that depend on an additional sample): the number of unseen species, the unseen prevalence $\mathfrak{u}_{r,n,m}$ for $r=1$ and the coverage of prevalence $\mathfrak{f}_{r,n,m}$ for $r=1$. The median absolute percentage error (across methods and additional sample sizes) is less then $30\%$ for the number of unseen and the unseen prevalence, and $7\%$ for the coverage of prevalences, when the data is generated from a PYP. 
Figure~\ref{fig:PYprediction} displays the predicted and actual value for a synthetic dataset that can be thought as ``representative'', as it achieved the median error across the several generated datasets. For all these SSP, the prediction is somewhat accurate, and most often the credible intervals contain the ``true" value of the functional of interest. 
When the data was generated from a power-law Zipf distributios, the predictive performance gets worse, with median errors less then $40\%$ for the number of unseen and the unseen prevalence and $60\%$ for the coverage of prevalences. We refer to Section~S13 of the Supplementary Materials \citep{BFN(23)} for more details. 

\begin{figure}
\includegraphics[width=0.45\textwidth]{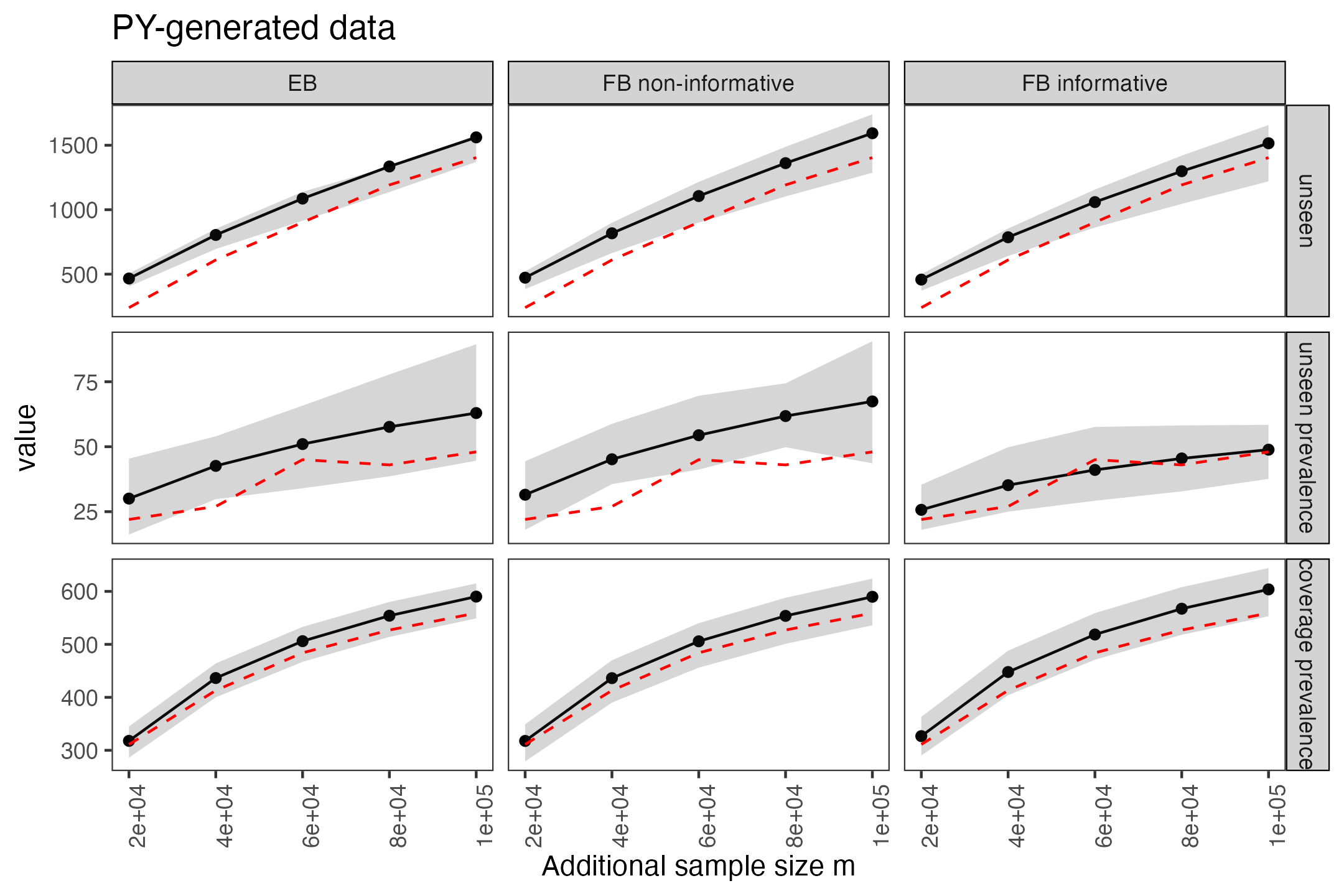}\\
\centering
\caption{Representative examples of the predicted value (black solid line) and actual value (red dashed line) for three SSP (the number of unseen $\mathfrak{u}_{n,m}$, unseen prevalence $\mathfrak{u}_{1,n,m}$, and coverage of prevalence $\mathfrak{f}_{1,n,m}$) as a function of the additional sample size $m$, for three parameter estimating method (empirical Bayes (EB), fully Bayes (FB) with non-informative priors, FB with informative priors). The gray bands represent the $95\%$ credible intervals. \label{fig:PYprediction}}
\end{figure}


\section{Some generalizations of SSPs}\label{sec8}

``Feature-sampling" problems (FSPs) generalize SSPs by allowing an individual in the population to belong to more than one species, which are referred to as features. To introduce the class of FSPs, we consider a population of individuals such that each individual is endowed with a finite set of features' labels belonging to a (possibly infinite) space of features. The ``classical framework" for FSPs assumes that $n\geq1$ observed samples from the population are modeled as a random sample $(Y_{1},\ldots,Y_{n})$, where $Y_{i}=(Y_{i,j})_{j\geq1}$ is a sequence of  independent Bernoulli r.v.s with unknown feature probabilities $(p_{j})_{j\geq1}$, such that $Y_{r}$ is independent of $Y_{s}$ for any $r\neq s$. In principle, each of SSPs discussed in this paper admits a corresponding feature sampling counterpart.  Within the broad class of FSPs, in recent years there has been a growing interest, especially biological sciences, in the estimation of
\begin{equation}\label{uns_fea}
\sum_{j\geq1}I\left(\sum_{i=1}^{n}Y_{i,j}=0\right)I\left(\sum_{i=1}^{m}Y_{n+i,j}>0\right),
\end{equation}
namely the number of hitherto unseen features that would be observed if $m$ additional samples $(Y_{n+1},\ldots,Y_{n+m})$ were collected from the same $(p_{j})_{j\geq1}$. We refer to \citet{Ion(09)}, \citet{Gra(14)}, \citet{Zou(16)}, \citet{Orl(17)} \citet{Cha(19)} for parametric and nonparametric approaches to estimate \eqref{uns_fea}. The FSP \eqref{uns_fea} provides the natural  feature sampling counterpart of the problem of the unseen-species problem; we refer to \cite{Fad(21)} for classical (frequentist) nonparametric inference of a feature sampling counterpart of the problem of estimating the missing mass. Recently \citet{Mas(20)} and \citet{Mas(20a)} proposed a BNP approach to estimate \eqref{uns_fea}, which rely on placing suitable prior distributions on the underlying probabilities $(p_{j})_{j\geq1}$. See \citet{BerF(23)} for a more general setting. Despite these works, BNP inference for FSPs remains still a mostly unexplored field for both methods and applications. See, e.g., \citet{Ber(23)} and \citet{Mas(23)}.

SSPs may be generalized to multiple populations of individuals, in such a way that populations share species. Consider $r>1$ populations of individuals, such that each individual is labeled by a symbol or species' label belonging to a (possibly infinite) space of symbols. That is, species' labels are shared among the populations. Following the ``classical framework" for SSPs, it is assumed that  $r$ observed samples of individuals from the populations, the $i$-th sample being of size $n_{i}$, are modeled as a random sample $\{(X_{i,1},\ldots,X_{i,n_{i}})\}_{i=1,\ldots,r}$ from a collection of $r$ unknown distributions $(p_{1},\ldots,p_{r})$. Then, interest is in estimating discrete functionals that encode features of additional unobserved samples from the same populations; of special interest are functionals encoding information of the number of shared species among populations. In recent years, SSPs with multiple populations have become critical in microbiome studies, i.e. microbial ecology and biology, where next generation sequencing has been applied to obtain inventories of bacteria in many different environments (populations); see \citet{Jeg(21)} and references therein. Nonparametric inference for SSPs with multiple populations pose challenging mathematical hurdles to overcome \citep{Rag(17),Hao(20b)}. In particular, the BNP approach requires to place a nonparametric prior on the underlying collection of distributions $(p_{1},\ldots,p_{r})$, in such a way to model the unknown species compositions of the populations and the dependency among these compositions. Hierarchical priors \citep{Teh(06),Cam(19)} and compound priors \citep{Gri(17)} provide a broad class of priors for $(p_{1},\ldots,p_{r})$, being mathematically tractable and flexible in terms of prior's parameters. However, to date, BNP inference for SSPs in multiple populations is mostly unexplored, being difficult to obtain posterior inferences that are analytically tractable and, most importantly, computationally efficient in applications.

In the ``classical framework" for SSPs, the observed samples are modeled as a random sample $\mathbf{X}_{n}$ from an unknown distribution $p$. The assumption of independence among the $X_{i}$'s is unrealistic in many applications, though it yields results that are interesting in themselves, and upon which more sophisticated frameworks may be built. For instance, in a natural languages the probability of appearance of a word strongly depends on the previous words, both for grammatical and semantic reasons. Likewise, the nucleotides in a DNA sequence do not form a random sample. There has been a recent interest in estimating the missing mass and coverage probabilities when observed samples are modeled as Markov chains \citep{Asa(14), Fal(16), Hao(18), Wol(19), Sko(20a), Cha(21),Pan(24)}. \citet{Bac(13)} first considered BNP analysis of SSPs under the assumption that the observed samples are modeled as a reversible Markov chain. This is motivated by the analysis of benchtop and computer experiments that produce data associated with the structural fluctuations of a protein in water, with species being protein conformational states \citep{Pan(10)}. \citet{Bac(13)} introduced a nonparametric prior for the unknown transition kernel of a reversible Markov chain, such that: i) the state space of the chain is uncountable; ii) the prediction for the next state visited by the chain is not solely a function of the number of transitions observed in and out of the last state, but transition probabilities out of different states share statistical strength. While the model of \citet{Bac(13)} can be used to predict characteristics of future unobserved trajectories of reversible Markov chains, i.e. protein dynamics, the major goals of the their paper were: i) predicting how soon the chain will return to a specific state of interest; ii) predicting the number of states that the  chain has not yet visited in the first (observed) transitions and that will appear in subsequent (unobserved) transitions. 


\section*{Acknowledgements}

The authors are grateful to the Editor (Professor Sonia Petrone), the Associate Editor and three Referees for their comments and corrections that allow to improve remarkably the paper. Stefano Favaro wishes to thank Timothy P. Daley for useful discussions on applications of species sampling problems in biological and physical sciences. Cecilia Balocchi, Stefano Favaro and Zacharie Naulet received funding from the European Research Council (ERC) under the European Union's Horizon 2020 research and innovation programme under grant agreement No 817257. Stefano Favaro gratefully acknowledge the support from the Italian Ministry of Education, University and Research (MIUR), ``Dipartimenti di Eccellenza" grant 2018-2022.


\pagebreak
\begin{flushleft}
\textbf{\Huge Supplemental Materials for \\ 
Bayesian nonparametric inference for ``species-sampling" problems}\\

\vspace{0.3cm}

\textbf{\Large Cecilia Balocchi, Stefano Favaro, Zacharie Naulet}
\end{flushleft}
\vspace{0.6cm}

\thispagestyle{empty}
\setcounter{equation}{0}
\setcounter{figure}{0}
\setcounter{table}{0}
\setcounter{page}{1}
\setcounter{section}{0}
\makeatletter
\renewcommand{\theequation}{S\arabic{equation}}
\renewcommand{\thesection}{S\arabic{section}}
\renewcommand{\thefigure}{S\arabic{figure}}
\renewcommand{\bibnumfmt}[1]{[S#1]}
\renewcommand{\citenumfont}[1]{S#1}

\numberwithin{equation}{section}

\numberwithin{equation}{section}

\section{Combinatorial numbers}\label{sup_a}
We recall some basic identities on factorial numbers, Stirling numbers, generalized factorial coefficients and generalizations thereof, which will be critical to prove the results of the present paper. We refer to Charalambides (2005) for a comprehensive account on these combinatorial numbers. For $t\in\mathbb{R}^{+}$ and $u\in\mathbb{N}_{0}$, let $(t)_{(u)}=\prod_{0\leq i\leq u-1}(t+i)$ be the rising factorial of $t$ of order $u$. For $t>0$, the $(u,v)$-th signless Stirling number of the first type, denoted by $|s(u,v)|$, is the defined as the $v$-th coefficient in the expansion of $(t)_{(u)}$ into powers, i.e., 
\begin{displaymath}
(t)_{(u)}=\sum_{v=0}^{u}|s(u,v)|t^{v}.
\end{displaymath}
For $t\in\mathbb{R}^{+}$ and $u\in\mathbb{N}_{0}$, let $(t)_{[u]}=\prod_{0\leq i\leq u-1}(t-i)$ be the falling factorial of $t$ of order $u$, such that it  holds $(t)_{(u)}=(-1)^{u}(-t)_{[u]}$. The $(u,v)$-th Stirling number of the second type, denoted by $S(u,v)$, is the defined as the $v$-th coefficient in the expansion of $t^{u}$ into falling factorials, i.e., 
\begin{displaymath}
t^{u}=\sum_{v=0}^{u}S(u,v)(t)_{[v]}.
\end{displaymath}
It is assumed: $|s(0,0)|=S(0,0)=1$, $|s(u,0)|=S(u,0)=0$ for $u>0$ and $|s(u,v)|=S(u,v)=0$ for $v>u$. An generalization of Stirling numbers is given by the coefficients of the expansion of non-centered rising factorials into powers and of powers into non-centered falling factorials, i.e., 
\begin{displaymath}
(t+b)_{(u)}=\sum_{v=0}^{u}|s(u,v;b)|t^{v}
\end{displaymath}
and
\begin{displaymath}
t^{u}=\sum_{v=0}^{u}S(u,v;b)(t-b)_{[v]}.
\end{displaymath}
It is assumed: $|s(0,0;b)|=S(0,0;b)=1$, $|s(u,0;b)|=(b)_{(u)}$ for $u>0$, $S(u,0;b)=b^{u}$ for $u>0$ and $|s(u,v;b)|=S(u,v;b)=0$ for $v>u$. The coefficients $|s(u,v;b)|$ and $S(u,v;b)$ are typically referred to as signless non-centered Stirling number of the first type and non-centered Stirling number of the second type, respectively. An explicit expression for the non-centered Stirling number of the second type can be deduced by applying its definition, that is it holds
\begin{displaymath}
S(u,v;b)=\frac{1}{v!}\sum_{j=0}^{v}(-1)^{v-j}{v\choose j}(j+b)^{u}.
\end{displaymath}
The following convolutional identities show some useful relationships between Stirling numbers of the first and of the second type and their corresponding non-centered Stirling numbers, i.e., 
\begin{itemize}
\item[i)]
\begin{displaymath}
|s(u,v;b)|=\sum_{j=v}^{u}{j\choose v}b^{j-v}|s(u,j)|=\sum_{j=v}^{u}{u\choose j}(b)_{(u-j)}|s(j,v)|;
\end{displaymath}
\item[ii)]
\begin{equation}\label{eq:stir2}
S(u,v;b)=\sum_{j=v}^{u}{u\choose j}b^{u-j}S(j,v)=\sum_{j=v}^{u}{j\choose v}(b)_{[j-v]}S(u,j).
\end{equation}
\end{itemize}
A variation of the convolutional identity \eqref{eq:stir2}, which will be useful in our context, is the following
\begin{equation}\label{eq:stir1}
S(u,v;b)=\sum_{j=v}^{u}(-1)^{u-j}{j\choose v}S(u,j)(b+v)_{(j-v)}.
\end{equation}
For $t\in\mathbb{R}^{+}$, $a\in\mathbb{R}$ and $u\in\mathbb{N}_{0}$, let us consider the rising factorial of $at$ of order $u$, i.e., $(at)_{(u)}=\prod_{0\leq i\leq u-1}(at+i)$. The $(u,v)$-th (centered) generalized factorial coefficient, denoted by $\mathscr{C}(u,v;a)$, is defined as the  $v$-th coefficient in the expansion of $(at)_{(u)}$ into rising factorials, i.e.,
\begin{equation}\label{eq:genfact3}
(at)_{(u)}=\sum_{v=0}^{u}\mathscr{C}(u,v;a)(t)_{(v)}.
\end{equation}
It is assumed: $\mathscr{C}(0,0;a)=1$, $\mathscr{C}(u,0;a)=0$ for $u>0$, $\mathscr{C}(u,v;a)=0$ for $v>u$. Similarly to the non-centered Stirling number, a generalization of the (centered) generalized factorial coefficient is defined as the $v$-th coefficient in the expansion of $(at-b)_{(u)}$ into rising factorials, i.e.,
\begin{equation}\label{eq:genfact5}
(at-b)_{(u)}=\sum_{v=0}^{u}\mathscr{C}(u,v;a,b)(t)_{(v)}.
\end{equation}
It is assumed: $\mathscr{C}(0,0;a,b)=1$, $\mathscr{C}(u,0;a,b)=(-b)_{(u)}$ for $u>0$, $\mathscr{C}(u,v;a,b)=0$ for $v>u$. 
The coefficient $\mathscr{C}(u,v;a,b)$ is referred to as the non-centered generalized factorial coefficient. An explicit expression for the non-centered generalized factorial coefficients can be deduced, i.e., 
\begin{equation}\label{eq:genfact1}
\mathscr{C}(u,v;a,b)=\frac{1}{v!}\sum_{j=0}^{v}(-1)^{j}{v\choose j}(-ja-b)_{(u)}.
\end{equation}
The following convolutional identities show the interplay between the (centered) generalized factorial coefficient and the non-centered generalized factorial coefficient, and the interplay between the non-centered generalized factorial coefficient and the non-centered Stirling number, i.e.,
\begin{itemize}
\item[i)]
\begin{equation}\label{eq:genfact4}
\mathscr{C}(u,v;a,b)=\sum_{j=v}^{u}{u\choose j}\mathscr{C}(j,v;a)(-b)_{(u-j)};
\end{equation}
\item[ii)]
\begin{displaymath}
\mathscr{C}(u,v;a_{1}a_{2},a_{1}a_{2}+b_{1})=\sum_{j=v}^{u}\mathscr{C}(u,j;a_{1},b_{1})\mathscr{C}(j,v;a_{2},b_{2});
\end{displaymath}
\item[iii)]
\begin{equation}\label{eq:genfact2}
\lim_{a\rightarrow0}\frac{\mathscr{C}(u,v;a,b)}{a^{k}}=|s(u,v;-b)|;
\end{equation}
\item[iv)]
\begin{displaymath}
\lim_{a\rightarrow+\infty}\frac{\mathscr{C}(u,v;a,ab)}{b^{n}}=(-1)^{u-v}S(u,v;b).
\end{displaymath}
\end{itemize}

\section{A compound Poisson perspective of EPSF}\label{supp_cp}

We conclude present representation of the EPSF in terms of compound Poisson sampling models \citep[Chapter 7]{Cha(05)}, thus providing an intuitive construction of the EPSF that sheds light on the sampling structure of the PYP prior. We consider a population of individuals with a random number $K$ of distinct types, and we assume that $K$ is distributed as a Poisson distribution with parameter $\lambda=z[1-(1-q)^{\alpha}]$ such that $q\in(0,1)$, $\alpha\in(0,1)$ and $z>0$. For $i\geq1$ let $N_{i}$ denote the random number of individuals of type $i$ in the population, and assume the $N_{i}$'s to be independent of $K$, independent of each other and such that for $x\in\mathbb{N}$
\begin{equation}\label{eq:negbin}
\text{Pr}[N_{i}=x]=-\frac{1}{[1-(1-q)^{\alpha}]}{\alpha\choose x}(-q)^{x}
\end{equation}
for all $i\geq1$. Let $S=\sum_{1\leq i\leq K}N_{i}$ and $M_{r}=\sum_{1\leq i\leq K} I(N_{i}=r)$ for $r=1,\ldots,S$, that is $M_{r}$ is the random number of $N_{i}$'s equal to $r$ such that $\sum_{r\geq1}M_{r}=K$ and $\sum_{r\geq1}rM_{r}=S$. If $\mathbf{M}(n,z)=(M_{1}(n,z),\ldots,M_{n}(n,z))$ is a random variable whose distribution coincides with the conditional distribution of $(M_{1},\ldots,M_{S})$ given $S=n$, then it holds \citep[Section 3]{Cha(07)}
\begin{align}\label{eq:sampling_alpha}
&\text{Pr}[\mathbf{M}(n,z)=(m_{1},\ldots,m_{n})]=\frac{n!}{\sum_{j=0}^{n}\mathscr{C}(n,j;\alpha,0)z^{j}}\prod_{i=1}^{n}\frac{\left[z\frac{\alpha(1-\alpha)_{(i-1)}}{i!}\right]^{m_{i}}}{m_{i}!}.
\end{align}
The distribution \eqref{eq:sampling_alpha} is referred to as the negative Binomial compound Poisson sampling formula. See \citet[Chapter 7]{Cha(05)} for an account on compound Poisson sampling models and their distributional properties, and to \citet{DF(20c)} for a comprehensive treatment of the large $n$ asymptotic behaviour of $\mathbf{M}(n,z)$ under the negative Binomial compound Poisson sampling model. See also \citet{DF(20a),DF(20b)} and references therein.

By letting $\alpha\rightarrow0$ the distribution \eqref{eq:negbin} reduces to the log-series distribution, whereas \eqref{eq:sampling_alpha} reduces to the well-known log-series compound Poisson sampling formula. The equivalence between the Ewens sampling formula, i.e. \eqref{eq_ewe_py} with $\alpha=0$, and the log-series compound Poisson sampling formula is known from \citet{Wat(74)}. For $\alpha\in(0,1)$, it exists an interplay  between the EPSF \eqref{eq_ewe_py} and the negative Binomial compound Poisson sampling formula \eqref{eq:sampling_alpha}. Let $G_{a,b}$ denote a Gamma random variable with scale parameter $b>0$ and shape parameter $a>0$, and let $S_{\alpha,\theta}$ be the Pitman's $\alpha$-diversity random variable in \eqref{eq:sigma_diversity}. For $\alpha\in(0,1)$ and $\theta>-\alpha$, let $G_{\theta+n,1}$ and $S_{\alpha,\theta}$ be independent random variables and define the scale mixture random variable
\begin{displaymath}
\bar{X}_{\alpha,\theta,n}=G^{\alpha}_{\theta+n,1}S_{\alpha,\theta}.
\end{displaymath}
If $\mathbf{M}_{n}$ is distributed as the EPSF \eqref{eq_ewe_py}, with $\alpha\in(0,1)$, and $\mathbf{M}(n,z)$ is distributed as the negative Binomial compound Poisson formula \eqref{eq:sampling_alpha}, then \citet[Theorem 2.2]{DF(20c)} show that
\begin{equation}\label{distr_id_epsf}
\mathbf{M}_{n}\stackrel{\text{d}}{=}\mathbf{M}(n,\bar{X}_{\alpha,\theta,n}).
\end{equation}
According to \eqref{distr_id_epsf} the EPSF admits a representation in terms of a randomized negative Binomial compound Poisson sampling formula. The randomization acts on the parameter $z$ with respect to the distribution of the scaled Pitman's $\alpha$-diversity random variable $\bar{X}_{\alpha,\theta,n}$. This is an instance of a compound mixed Poisson sampling model \citep[Chapter 7]{Cha(05)}, where the distribution of $K$ is a mixture of Poisson distributions with respect to the distribution of $\bar{X}_{\alpha,\theta,n}$. We refer to \citet{DF(20c)} for a detailed account of \eqref{distr_id_epsf}, and its role in quantifying the large $n$ asymptotic behaviours of $K_{n}$ and $M_{r,n}$ in \eqref{eq:sigma_diversity} and \eqref{eq:sigma_diversity_m}, respectively.

\section{Proof of Proposition \ref{prp1}}\label{sup_b}
We start by proving statement i), and then statement ii) by letting $\alpha\rightarrow0$. For $\alpha\in(0,1)$ and $\theta>-\alpha$, the falling factorial moment of order $d\geq0$ of $K^{\ast}_{m}$ \citep[Proposition 1]{Fav(09)} is
\begin{equation}\label{eq_momk}
\E[(K^{\ast}_{m})_{[d]}]=(-1)^{d}d!\frac{\left(\frac{\theta+n}{\alpha}\right)_{(d)}}{(\theta+n)_{(m)}}\mathscr{C}(m,d;-\alpha,-\theta-n).
\end{equation}
Let $Z_{n,p}\sim\text{Binomial}(n,p)$, and let recall the moment of order $d$ of $Z_{n,p}$ \citep[Chapter 3.3]{Joh(05)}, i.e., 
\begin{equation}\label{eq_momb}
\E[(Z_{n,p})^{d}]=\sum_{t=0}^{r}S(d,t)(n)_{[t]}p^{t}
\end{equation}
for $d\geq0$. By a direct application of Proposition 1 in \citet{Fav(09)}, the following identities hold
\begin{align*}
&\E[(\mathfrak{u}_{n,m})^{d}\,|\,\mathbf{X}_{n}]\\
&\quad=\sum_{i=0}^{d}(-1)^{d-i}\left(k+\frac{\theta}{\alpha}\right)_{(i)}S\left(d,i;k+\frac{\theta}{\alpha}\right)\frac{(\theta+n+i\alpha)_{(m)}}{(\theta+n)_{(m)}}\\
&\text{[by Equation \eqref{eq:stir1}]}\\
&\quad=\sum_{i=0}^{d}(-1)^{-i}\frac{(\theta+n+i\alpha)_{(m)}}{(\theta+n)_{(m)}}\sum_{t=i}^{d}(-1)^{t}{t\choose i}S(d,t)\left(k+\frac{\theta}{\alpha}\right)_{(t)}\\
&\quad=\sum_{t=0}^{r}S(d,t)\frac{\left(k+\frac{\theta}{\alpha}\right)_{(t)}}{\left(\frac{\theta+n}{\alpha}\right)_{(t)}}\left(\frac{\theta+n}{\alpha}\right)_{(t)}\sum_{i=0}^{t}(-1)^{t-i}{t\choose i}\frac{(\theta+n+i\alpha)_{(m)}}{(\theta+n)_{(m)}}\\
&\text{[by Equation \eqref{eq:genfact1}]}\\
&\quad=\sum_{t=0}^{d}S(d,t)\frac{\left(k+\frac{\theta}{\alpha}\right)_{(t)}}{\left(\frac{\theta+n}{\alpha}\right)_{(t)}}\frac{\left(\frac{\theta+n}{\alpha}\right)_{(t)}}{(\theta+n)_{(m)}}(-1)^{t}t!\mathscr{C}(m,t;-\alpha,-\theta-n)\\
&\text{[by Equation \eqref{eq_momk}]}\\
&\quad=\sum_{t=0}^{d}S(d,t)\frac{\left(k+\frac{\theta}{\alpha}\right)_{(t)}}{\left(\frac{\theta+n}{\alpha}\right)_{(t)}}\E[(K^{\ast}_{m})_{[t]}]\\
&\quad=\sum_{t=0}^{d}S(d,t)\frac{\Gamma\left(\frac{\theta+n}{\alpha}\right)}{\Gamma\left(\frac{\theta}{\alpha}+k\right)\Gamma\left(\frac{n}{\alpha}-k\right)}\frac{\Gamma\left(k+\frac{\theta}{\alpha}+t\right)\Gamma\left(\frac{n}{\alpha}-k\right)}{\Gamma\left(\frac{\theta+n}{\alpha}+t\right)}\E[(K^{\ast}_{m})_{[t]}]\\
&\text{[by the identity $\frac{\Gamma\left(k+\frac{\theta}{\alpha}+t\right)\Gamma\left(\frac{n}{\alpha}-k\right)}{\Gamma\left(\frac{\theta+n}{\alpha}+t\right)}=\int_{0}^{1}x^{t+\frac{\theta}{\alpha}+k-1}(1-x)^{\frac{n}{\alpha}-k-1}\ddr x$]}\\
&\quad=\sum_{t=0}^{d}S(d,t)\E[(K^{\ast}_{m})_{[t]}]\frac{\Gamma\left(\frac{\theta+n}{\alpha}\right)}{\Gamma\left(\frac{\theta}{\alpha}+k\right)\Gamma\left(\frac{n}{\alpha}-k\right)}\int_{0}^{1}x^{t+\frac{\theta}{\alpha}+k-1}(1-x)^{\frac{n}{\alpha}-k-1}\ddr x\\
&\quad=\sum_{t=0}^{d}S(d,t)\E[(K^{\ast}_{m})_{[t]}]\E\left[\left(B_{\frac{\theta}{\alpha}+k,\,\frac{n}{\alpha}-k}\right)^{t}\right]\\
&\quad=\E\left[\E\left[\sum_{t=0}^{d}S(d,t)(K^{\ast}_{m})_{[t]}\left(B_{\frac{\theta}{\alpha}+k,\,\frac{n}{\alpha}-k}\right)^{t}\right]\right]\\
&\text{[by Equation \eqref{eq_momb}]}\\
&\quad=\E\left[\left(Z_{K^{\ast}_{m},B_{\frac{\theta}{\alpha}+k,\,\frac{n}{\alpha}-k}}\right)^{d}\right].
\end{align*}
This completes the proof of statement i) of Proposition \ref{prp1}. The proof of statement ii) follows along similar lines, and by letting $\alpha\rightarrow0$. By combining Equation \eqref{eq_momk} with Equation \eqref{eq:genfact2} it holds
\begin{align}\label{eq_momk_dir}
\E[(K^{\ast}_{m})_{[d]}]&=\lim_{\alpha\rightarrow0}(-1)^{d}d!\frac{\left(\frac{\theta+n}{\alpha}\right)_{(d)}}{(\theta+n)_{(m)}}\mathscr{C}(m,d;-\alpha,-\theta-n)\\
&\notag=\frac{(\theta+n)^{d}}{(\theta+n)_{(m)}}d!|s(m,d;\theta+n)|
\end{align}
for $d\geq0$. Then, still using Proposition 1 in \citet{Fav(09)}, the following identities hold true
\begin{align*}
\E[(\mathfrak{u}_{n,m})^{d}\,|\,\mathbf{X}_{n}]&=\lim_{\alpha\rightarrow0}\sum_{i=0}^{d}(-1)^{d-i}\left(k+\frac{\theta}{\alpha}\right)_{(i)}S\left(d,i;k+\frac{\theta}{\alpha}\right)\frac{(\theta+n+i\alpha)_{(m)}}{(\theta+n)_{(m)}}\\
&\text{[by Equation \eqref{eq:stir1} and Equation \eqref{eq:genfact1}]}\\
&=\lim_{\alpha\rightarrow0}\sum_{t=0}^{d}S(d,t)\frac{\left(k+\frac{\theta}{\alpha}\right)_{(t)}}{\left(\frac{\theta+n}{\alpha}\right)_{(t)}}\frac{\left(\frac{\theta+n}{\alpha}\right)_{(t)}}{(\theta+n)_{(m)}}(-1)^{t}t!\mathscr{C}(m,t;-\alpha,-\theta-n)\\
&\text{[by Equation \eqref{eq:genfact2}]}\\
&=\sum_{t=0}^{d}S(d,t)\frac{\theta^{t}}{(\theta+n)_{(m)}}t!|s(m,t;\theta+n)|\\
&=\sum_{t=0}^{d}S(d,t)\frac{\frac{\theta^{t}}{(\theta+n)_{(m)}}}{(\theta+n)^{t}}(\theta+n)^{t}t!|s(m,t;\theta+n)|\\
&\text{[by Equation \eqref{eq_momk_dir}]}\\
&=\sum_{t=0}^{d}S(d,t)\left(\frac{\theta}{\theta+n}\right)^{t}\E[(K^{\ast}_{m})_{[t]}]\\
&\text{[by Equation \eqref{eq_momb}]}\\
&=\E\left[\left(Z_{K^{\ast}_{m},\frac{\theta}{\theta+n}}\right)^{d}\right].
\end{align*}
This completes the proof of statement ii). Now, we consider the proof of statement iii). For $\alpha\in(0,1)$ and $\theta>-\alpha$, the falling factorial moment of order $d\geq0$ of $M^{\ast}_{r,m}$ \citep[Proposition 1]{Fav(13)} is
\begin{equation}\label{eq_momm}
\E[(M^{\ast}_{r,m})_{[d]}]=(m)_{[dr]}\left(\frac{\alpha(1-\alpha)_{(r-1)}}{r!}\right)^{d}\left(\frac{\theta+n}{\alpha}\right)_{(d)}\frac{(\theta+n+d\alpha)_{(m-dr)}}{(\theta+n)_{(m)}}.
\end{equation}
Then, by combining Theorem 2 in \citet{Fav(13)} with Equation \eqref{eq:genfact1} the following identities hold
\begin{align*}
&\E[(\mathfrak{u}_{r,n,m})^{d}\,|\,\mathbf{X}_{n}]\\
&\quad=\sum_{t=0}^{d}S(d,t)(m)_{[tr]}\left(\frac{\alpha(1-\alpha)_{(r-1)}}{r!}\right)^{t}\left(k+\frac{\theta}{\alpha}\right)_{[t]}\frac{(\theta+n+t\alpha)_{(m-tr)}}{(\theta+n)_{(m)}}\\
&\quad=\sum_{t=0}^{d}S(d,t)\frac{\left(k+\frac{\theta}{\alpha}\right)_{[t]}}{\left(\frac{\theta+n}{\alpha}\right)_{(t)}}(m)_{[tr]}\left(\frac{\alpha(1-\alpha)_{(r-1)}}{r!}\right)^{t}\left(\frac{\theta+n}{\alpha}\right)_{(t)}\frac{(\theta+n+t\alpha)_{(m-tr)}}{(\theta+n)_{(m)}}\\
&\text{[by Equation \eqref{eq_momm}]}\\
&\quad=\sum_{t=0}^{d}S(d,t)\frac{\left(k+\frac{\theta}{\alpha}\right)_{[t]}}{\left(\frac{\theta+n}{\alpha}\right)_{[t]}}\E[(M^{\ast}_{r,m})_{[t]}]\\
&\quad=\sum_{t=0}^{d}S(d,t)\frac{\Gamma\left(\frac{\theta+n}{\alpha}\right)}{\Gamma\left(\frac{\theta}{\alpha}+k\right)\Gamma\left(\frac{n}{\alpha}-k\right)}\frac{\Gamma\left(k+\frac{\theta}{\alpha}+t\right)\Gamma\left(\frac{n}{\alpha}-k\right)}{\Gamma\left(\frac{\theta+n}{\alpha}+t\right)}\E[(M^{\ast}_{r,m})_{[t]}]\\
&\text{[by the identity $\frac{\Gamma\left(k+\frac{\theta}{\alpha}+t\right)\Gamma\left(\frac{n}{\alpha}-k\right)}{\Gamma\left(\frac{\theta+n}{\alpha}+t\right)}=\int_{0}^{1}x^{t+\frac{\theta}{\alpha}+k-1}(1-x)^{\frac{n}{\alpha}-k-1}\ddr x$]}\\
&\quad=\sum_{t=0}^{d}S(d,t)\E[(M^{\ast}_{r,m})_{[t]}]\frac{\Gamma\left(\frac{\theta+n}{\alpha}\right)}{\Gamma\left(\frac{\theta}{\alpha}+k\right)\Gamma\left(\frac{n}{\alpha}-k\right)}\int_{0}^{1}x^{t+\frac{\theta}{\alpha}+k-1}(1-x)^{\frac{n}{\alpha}-k-1}\ddr x\\
&\quad=\sum_{t=0}^{d}S(d,t)\E[(M^{\ast}_{r,m})_{[t]}]\E\left[\left(B_{\frac{\theta}{\alpha}+k,\frac{n}{\alpha}-k}\right)^{t}\right]\\
&\quad=\E\left[\E\left[\sum_{t=0}^{d}S(d,t)(M^{\ast}_{r,m})_{[t]}\left(B_{\frac{\theta}{\alpha}+k,\frac{n}{\alpha}-k}\right)^{t}\right]\right]\\
&\text{[by Equation \eqref{eq_momb}]}\\
&\quad=\E\left[\left(Z_{M^{\ast}_{r,m}}B_{\frac{\theta}{\alpha}+k,\frac{n}{\alpha}-k}\right)^{d}\right].
\end{align*}
This completes the proof of statement iii) of Proposition \ref{prp1}. The proof of statement iv) follows along similar lines, and by letting $\alpha\rightarrow0$. from Equation \eqref{eq_momm} we can write the following identity
\begin{align}\label{eq_momm_dir}
&\E[(M^{\ast}_{r,m})_{[d]}]\\
&\notag\quad=\lim_{\alpha\rightarrow0}(m)_{[dr]}\left(\frac{\alpha(1-\alpha)_{(r-1)}}{r!}\right)^{d}\left(\frac{\theta+n}{\alpha}\right)_{(d)}\frac{(\theta+n+d\alpha)_{(m-dr)}}{(\theta+n)_{(m)}}\\
&\notag\quad=(m)_{[dr]}\left(\frac{\theta+n}{r}\right)^{d}\frac{(\theta+n)_{(m-dr)}}{(\theta+n)_{(m)}}
\end{align}
for $d\geq0$. Then, still by using Theorem 2 in \citet{Fav(13)}, the following identities hold true
\begin{align*}
&\E[(\mathfrak{u}_{r,n,m})^{d}\,|\,\mathbf{X}_{n}]\\ 
&\quad=\lim_{\alpha\rightarrow0}\sum_{t=0}^{d}S(d,t)(m)_{[tr]}\left(\frac{\alpha(1-\alpha)_{(r-1)}}{r!}\right)^{t}\left(k+\frac{\theta}{\alpha}\right)_{[t]}\frac{(\theta+n+t\alpha)_{(m-tr)}}{(\theta+n)_{(m)}}\\
&\quad=\sum_{t=0}^{d}S(d,t)(m)_{[tr]}\left(\frac{\theta}{r}\right)^{t}\frac{(\theta+n)_{(m-tr)}}{(\theta+n)_{(m)}}\\
&\quad=\sum_{t=0}^{d}S(d,t)\left(\frac{\theta}{\theta+n}\right)^{t}(m)_{[tr]}\left(\frac{\theta+n}{r}\right)^{t}\frac{(\theta+n)_{(m-tr)}}{(\theta+n)_{(m)}}\\
&\quad=\text{[by Equation \eqref{eq_momm_dir}]}\\
&\quad=\sum_{t=0}^{d}S(d,t)\left(\frac{\theta}{\theta+n}\right)^{t}\E[(M^{\ast}_{r,m})_{[t]}]\\
&\quad=\text{[by Equation \eqref{eq_momb}]}\\
&\quad=\E\left[\left(Z_{M^{\ast}_{r,m},\frac{\theta}{\theta+n}}\right)^{d}\right].
\end{align*}
This completes the proof of statement iv).

\section{Proof of Proposition \ref{prp2}}\label{sup_c}
We consider the case $\alpha\in(0,1)$ and $\theta>-\alpha$, whereas the case $\alpha=0$ and $\theta>0$ holds by letting $\alpha\rightarrow0$. By means of the definition of $\mathfrak{f}_{r,n,m}$ and an application of the Binomial theorem, we write
\begin{align}\label{starting}
\notag\E[(\mathfrak{f}_{r,n,m})^{d}\,|\,\mathbf{X}_{n}]&=\E\left[\left(\sum_{j\geq1}\mathbbm{1}_{\{N_{j,n}=r\}}\mathbbm{1}_{\{N_{j,m}>0\}}\right)^{d}\,|\,\mathbf{X}_{n}\right]\\
&\notag\quad=\E\left[\left(\sum_{j\geq1}\mathbbm{1}_{\{N_{j,n}=r\}}(1-\mathbbm{1}_{\{N_{j,m}=0\}})\right)^{d}\,|\,\mathbf{X}_{n}\right]\\
&\notag\quad=\E\left[\left(\sum_{j\geq1}\mathbbm{1}_{\{N_{j,n}=r\}}-\sum_{j\geq1}\mathbbm{1}_{\{N_{j,n}=r\}}\mathbbm{1}_{\{N_{j,m}=0\}}\right)^{d}\,|\,\mathbf{X}_{n}\right]\\
&\notag\quad=\E\left[\left(m_{r}-\sum_{j\geq1}\mathbbm{1}_{\{N_{j,n}=r\}}\mathbbm{1}_{\{N_{j,m}=0\}}\right)^{d}\,|\,\mathbf{X}_{n}\right]\\
&\quad=\sum_{i=0}^{d}{d\choose i}m_{r}^{d-i}(-1)^{i}\E\left[\left(\sum_{j\geq1}\mathbbm{1}_{\{N_{j,n}=r\}}\mathbbm{1}_{\{N_{j,m}=0\}}\right)^{i}\,|\,\mathbf{X}_{n}\right].
\end{align}
Now, let 
\begin{equation}\label{disc_r}
\mathfrak{d}_{r,n,m}=\sum_{j\geq1}\mathbbm{1}_{\{N_{j,n}=r\}}\mathbbm{1}_{\{N_{j,m}=0\}}
\end{equation}
be the random variable whose conditional moment of order $i$, given $\mathbf{X}_{n}$ is in Equation \eqref{starting}. The first part of the proof considers the problem of computing the conditional moment of order $d$, given $\mathbf{X}_{n}$, of the random variable \eqref{disc_r}. For $u\in\mathbb{N}$ and $b,c>0$, recall that a random variable $U_{b,c,u}$ on $\{1,\ldots,u\}$ has a generalized factorial distribution if, for $x \in \{1,\ldots,u\}$ it holds
\begin{equation}\label{eq_urn_supp}
\text{Pr}[U_{b,c,u}=x]=\frac{1}{(bc)_{(u)}}\mathscr{C}(u,x;b,0)(c)_{(x)}.
\end{equation}
See Supplementary Material \ref{sup_a}. For $u,v\in\mathbb{N}$ and $a>0$ such that $a>u$, recall that a random variable $H_{a,u,v}$ on $\{0,1,\ldots,u\}$ has a (general) hypergeometric distribution if, for $x\in{\{0,1,\ldots,u\}}$ it holds
\begin{displaymath}
\text{Pr}[H_{a,u,v}=x]=\frac{{a\choose x}{v\choose u-x}}{{a+v\choose u}}.
\end{displaymath}
Recall the moment of order $d$ of the random variable $H_{a,u,v}$ \citep[Chapter 6.3]{Joh(05)}, that is
\begin{equation}\label{gen_hyp_mom}
\E[H_{a,u,v}^{d}]=\sum_{i=1}^{d}S(d,i)i!{u\choose i}\frac{\Gamma\left(a+1+v-i\right)\Gamma\left(a+1\right)}{\Gamma\left(a+1-i\right)\Gamma\left(a+1+v\right)}.
\end{equation}
for $d\geq0$. Now, let $\mathcal{C}_{n,x}$ be a combinatorial set of combinations defined as follows: $\mathcal{C}_{n,0}=\emptyset$  and $\mathcal{C}_{n,x}=\{(c_1,\ldots,c_{x})\text{ : }  c_k\in \{1,\ldots,n\}, c_k\neq c_l, \ \text{if}\ k\neq l \}$ for any $x\geq1$. Then, we can write
\begin{align}\label{eq2}
&\notag\E[(\mathfrak{d}_{r,n,m})^{d}\,|\,\mathbf{X}_{n}]\\
&\notag\quad=\E[(\mathfrak{d}_{r,n,m})^{d}\,|\,\mathbf{N}_{n}=(n_{1},\ldots,n_{k}),K_{n}=k]\\
&\notag\quad=\E\left[\left(\sum_{j=1}^{k}\mathbbm{1}_{\{N_{j,n}=r\}}\mathbbm{1}_{\{N_{j,m}=0\}}\right)^{d}\,|\,\mathbf{N}_{n}=(n_{1},\ldots,n_{k}),K_{n}=k\right]\\
&\notag\quad=\sum_{x=1}^{k}\sum_{i_{1}=1}^{d}\sum_{i_{2}=1}^{i_{1}-1}\cdots\sum_{i_{x-1}=1}^{i_{x-2}-1}{d\choose i_{1}}{i_{1}\choose i_{2}}\cdots{i_{x-2}\choose i_{x-1}}\\
&\notag\quad\quad\times\sum_{(c_{1},\ldots,c_{x})\in\mathcal{C}_{k,x}}\E\left[\prod_{t=1}^{x}(\mathbbm{1}_{\{N_{c_{t},n}=r\}}\mathbbm{1}_{\{N_{c_{t},m}=0\}})^{i_{x-t}-i_{x-t+1}}\,|\,\mathbf{N}_{n}=(n_{1},\ldots,n_{k}),K_{n}=k\right]\\
&\notag\quad=\sum_{x=1}^{d}S(d,x)x!\sum_{(c_{1},\ldots,c_{x})\in\mathcal{C}_{k,x}}\E\left[\prod_{t=1}^{x}\mathbbm{1}_{\{N_{c_{t},n}=r\}}\mathbbm{1}_{\{N_{c_{t},m}=0\}}\,|\,\mathbf{N}_{n}=(n_{1},\ldots,n_{k}),K_{n}=k\right]\\
&\notag\quad=\sum_{x=1}^{d}S(d,x)x!\sum_{(c_{1},\ldots,c_{x})\in\mathcal{C}_{k,x}}\prod_{t=1}^{x}\mathbbm{1}_{\{N_{c_{t},n}=r\}}\E\left[\prod_{t=1}^{x}\mathbbm{1}_{\{N_{c_{t},m}=0\}}\,|\,\mathbf{N}_{n}=n_{1},\ldots,n_{k},K_{n}=k\right]\\
&\quad=\sum_{x=1}^{d}S(d,x)x!\sum_{(c_{1},\ldots,c_{x})\in\mathcal{C}_{k,x}}\prod_{t=1}^{x}\mathbbm{1}_{\{N_{c_{t},n}=r\}}\P[N_{c_{1},m}=0,\ldots,N_{c_{x},m}=0\,|\,\mathbf{N}_{n}=(n_{1},\ldots,n_{k}),K_{n}=k].
\end{align}
The conditional probability in \eqref{eq2} can be computed from \citet[Lemma 1]{Fav(13)}. Let $V_{m}$ be the number of $X_{n+i}$'s that do not coincide with any of the $S_{j}^{\ast}$'s in the random sample $\mathbf{X}_{n}$. From \citet[Equation 38 and Equation 40]{Fav(13)} we can write the following distributions
\begin{itemize}
\item[i)]
\begin{align*}
&\text{Pr}[N_{c_{1},m}=0,\ldots,N_{c_{x},m}=0\,|\,\mathbf{N}_{n}=(n_{1},\ldots,n_{k}),K_{n}=k,V_{m}=v]\\
&\quad=\frac{(n-\sum_{i=1}^{x}n_{c_{i}}-(k-x)\alpha)_{(m-v)}}{(n-k\alpha)_{(m-v)}};
\end{align*}
\item[ii)]
\begin{align*}
&\text{Pr}[V_{N-n}=v\,|\,\mathbf{N}_{n}=(n_{1},\ldots,n_{k}),K_{n}=k]\\
&\quad{m\choose v}(n-k\alpha)_{(m-v)}\sum_{j=0}^{v}\frac{\frac{\prod_{i=0}^{k+j-1}(\theta+i\alpha)}{(\theta)_{(n+m)}}}{\frac{\prod_{i=0}^{k-1}(\theta+i\alpha)}{(\theta)_{(n)}}}\frac{\mathscr{C}(v,j;\alpha)}{\alpha^{j}},
\end{align*}
\end{itemize}
such that
\begin{align}\label{eq3}
&\text{Pr}[N_{c_{1},m}=0,\ldots,N_{c_{x},m}=0\,|\,\mathbf{N}_{n}=\mathbf{n}_{n},K_{n}=k]\\
&\notag\quad=\sum_{j=0}^{m}\frac{1}{\alpha^{j}}\frac{\frac{\prod_{i=0}^{k+j-1}(\theta+i\alpha)}{(\theta)_{(n+m)}}}{\frac{\prod_{i=0}^{k-1}(\theta+i\alpha)}{(\theta)_{(n)}}}\sum_{v=j}^{m}{m\choose s}\mathscr{C}(v,j;\alpha)(n-\sum_{i=1}^{x}n_{c_{i}}-(k-x)\alpha)_{(m-v)}\\
&\notag\text{[by Equation \eqref{eq:genfact4}]}\\
&\notag\quad=\sum_{j=0}^{m}\frac{\frac{\prod_{i=0}^{k+j-1}(\theta+i\alpha)}{(\theta)_{(n+m)}}}{\frac{\prod_{i=0}^{k-1}(\theta+i\alpha)}{(\theta)_{(n)}}}\frac{\mathscr{C}(m,j;\alpha,-n+\sum_{i=1}^{x}n_{c_{i}}+(k-x)\alpha)}{\alpha^{j}}\\
&\notag\quad=\frac{1}{(\theta+n)_{(m)}}\sum_{j=0}^{m}\left(\frac{\theta}{\alpha}+k\right)_{(j)}\mathscr{C}(m,j;\alpha,-n+\sum_{i=1}^{x}n_{c_{i}}+(k-x)\alpha)\\
&\notag\text{[by Equation \eqref{eq:genfact5}]}\\
&\notag\quad=\frac{(\theta+n-\sum_{i=1}^{x}n_{c_{i}}+x\alpha)_{(m)}}{(\theta+n)_{(m)}}.
\end{align}
Then, by a direct combination of Equation \eqref{eq2} with Equation \ref{eq3} we can write the following identities
\begin{align*}
&\notag\E[(\mathfrak{d}_{r,n,m})^{d}\,|\,\mathbf{X}_{n}]\\
&\quad=\sum_{x=1}^{d}S(d,x)x!\sum_{(c_{1},\ldots,c_{x})\in\mathcal{C}_{k,x}}\prod_{t=1}^{x}\mathbbm{1}_{\{N_{c_{t},n}=r\}}\text{Pr}[N_{c_{1},m}=0,\ldots,N_{c_{x},m}=0\,|\,\mathbf{N}_{n}=(n_{1},\ldots,n_{k}),K_{n}=k]\\
&\quad=\sum_{x=1}^{d}S(d,x)x!\sum_{(c_{1},\ldots,c_{x})\in\mathcal{C}_{k,x}}\prod_{t=1}^{x}\mathbbm{1}_{\{N_{c_{t},n}=r\}}\frac{(\theta+n-\sum_{i=1}^{x}n_{c_{i}}+x\alpha)_{(m)}}{(\theta+n)_{(m)}}\\
&\notag\text{[by the definition of $\mathcal{C}_{k,x}$]}\\
&\quad=\sum_{x=1}^{d}S(d,x)x!{m_{r}\choose x}\frac{(\theta+n-xr+x\alpha)_{(m)}}{(\theta+n)_{(m)}}\\
&\quad=\frac{1}{(\theta+n)_{(m)}}\sum_{x=1}^{d}S(d,x)x!{m_{r}\choose x}\left((r-\alpha)\left(\frac{\theta+n}{r-\alpha}-x\right)\right)_{(m)}\\
&\notag\text{[by Equation \eqref{eq:genfact3}]}\\
&\quad=\frac{1}{(\theta+n)_{(m)}}\sum_{x=1}^{d}S(d,x)x!{m_{r}\choose x}\sum_{i=0}^{m}\mathscr{C}(m,i;r-\alpha)\left(\frac{\theta+n}{r-\alpha}-x\right)_{(i)}\\
&\quad=\frac{1}{(\theta+n)_{(m)}}\sum_{i=0}^{m}\mathscr{C}(m,i;r-\alpha)\sum_{x=1}^{d}S(d,x)x!{m_{r}\choose x}\frac{\Gamma\left(\frac{\theta+n}{r-\alpha}+i-x\right)}{\Gamma\left(\frac{\theta+n}{l-\alpha}-x\right)}\\
&\quad=\frac{1}{(\theta+n)_{(m)}}\sum_{i=0}^{m}\mathscr{C}(m,i;r-\alpha)\frac{\Gamma\left(\frac{\theta+n}{r-\alpha}+i\right)}{\Gamma\left(\frac{\theta+n}{r-\alpha}\right)}\\
&\quad\quad\times\sum_{x=1}^{d}S(d,x)x!{m_{r}\choose x}\frac{\Gamma\left(\frac{\theta+n}{r-\alpha}+i-x\right)\Gamma\left(\frac{\theta+n}{r-\alpha}\right)}{\Gamma\left(\frac{\theta+n}{r-\alpha}-x\right)\Gamma\left(\frac{\theta+n}{r-\alpha}+i\right)}\\
&\text{[by Equation \eqref{eq_urn_supp} and \eqref{gen_hyp_mom}]}\\
&\quad=\E[(H_{\frac{\theta+n}{r-\alpha}-1,m_{r},U_{m,\frac{\theta+n}{r-\alpha},r-\alpha}})^{d}].
\end{align*}
Then, by combining the last expression with \eqref{starting}, the moment $\E[(\mathfrak{f}_{r,n,m})^{d}\,|\,\mathbf{X}_{n}]$ can be written as
\begin{align*}
\notag\E[(\mathfrak{f}_{r,n,m})^{d}\,|\,\mathbf{X}_{n}]&=\sum_{i=0}^{d}{d\choose i}m_{r}^{d-i}(-1)^{i}\E\left[\left(\mathfrak{d}_{r,n,m}\right)^{i}\,|\,\mathbf{X}_{n}\right]\\
&=\sum_{i=0}^{d}{d\choose i}m_{r}^{d-i}(-1)^{i}\E[(H_{\frac{\theta+n}{r-\alpha}-1,m_{r},U_{m,\frac{\theta+n}{r-\alpha},r-\alpha}})^{i}]\\
&=\E\left[\left(m_{r}-H_{\frac{\theta+n}{r-\alpha}-1,m_{r},U_{m,\frac{\theta+n}{r-\alpha},r-\alpha}}\right)^{d}\right].
\end{align*}
The proof for $\alpha\in(0,1)$ and $\theta>-\alpha$ is completed. The case $\alpha=0$ and $\theta>0$ follows by letting $\alpha\rightarrow0$.

\section{Proof of Equation (\ref{eq_profile_est})}\label{sup_d}
Let $\mathbf{X}_{n}$ be a random sample from $P\sim\text{PYP}(\alpha,\theta)$ with $\alpha\in[0,1)$ and $\theta>-\alpha$, and let $\mathbf{X}_{n}$ feature $K_{n}=k$ distinct symbols with frequencies $\mathbf{N}_{n}=(n_{1},\ldots,n_{k})$, and $M_{r,n}=m_{r}$ is the number of distinct species with frequency $r\geq1$. From the proof of Proposition \ref{prp2} in Supplementary Material \ref{sup_c}, the moment of order $d\geq1$ of the posterior distribution of $\mathfrak{f}_{r,n,m}$, given $\mathbf{X}_{n}$, is
\begin{equation}\label{mom_tau1}
\E[(\mathfrak{f}_{r,n,m})^{d}\,|\,\mathbf{X}_{n}]=\sum_{i=1}^{d}{d\choose i}m_{r}^{d-i}(-1)^{i}\sum_{x=1}^{i}S(i,x)x!{m_{r}\choose x}\frac{(\theta+n-xr+x\alpha)_{(m)}}{(\theta+n)_{(m)}},
\end{equation}
and hence
\begin{displaymath}
\E[\mathfrak{f}_{r,n,m}\,|\,\mathbf{X}_{n}]=m_{r}-\frac{(\theta+n-r+\alpha)_{(m)}}{(\theta+n)_{(m)}}.
\end{displaymath}
This completes the proof.

\section{Proof of Equation (\ref{eq:approx_post})}\label{sup_e}
Let $\mathbf{X}_{n}$ be a random sample from $P\sim\text{PYP}(\alpha,\theta)$ with $\alpha\in[0,1)$ and $\theta>-\alpha$, and let $\mathbf{X}_{n}$ feature $K_{n}=k$ distinct symbols with frequencies $\mathbf{N}_{n}=(n_{1},\ldots,n_{k})$, and $M_{r,n}=m_{r}$ is the number of distinct species with frequency $r\geq1$. From the proof of Proposition \ref{prp2} in Supplementary Material \ref{sup_c}, the moment of order $d\geq1$ of the posterior distribution of $\mathfrak{f}_{r,n,m}$, given $\mathbf{X}_{n}$, is
\begin{equation}\label{mom_tau}
\E[(\mathfrak{f}_{r,n,m})^{d}\,|\,\mathbf{X}_{n}]=\sum_{i=1}^{d}{d\choose i}m_{r}^{d-i}(-1)^{i}\sum_{x=1}^{i}S(i,x)x!{m_{r}\choose x}\frac{(\theta+n-xr+x\alpha)_{(m)}}{(\theta+n)_{(m)}}
\end{equation}
Recall that by means of Stirling formula $\Gamma(n + i)/\Gamma(n)\simeq n^{i}$ as $n\rightarrow+\infty$ is a first order approximation of the Gamma function. By applying Stirling formula to \eqref{mom_tau}, as $n\rightarrow+\infty$ and $m\rightarrow+\infty$.
\begin{align}\label{eq:asym_mom}
&\nonumber \E[(\mathfrak{f}_{r,n,m})^{d}\,|\,\mathbf{X}_{n}]\\
&\nonumber\quad=\sum_{i=1}^{d}{d\choose i}m_{r}^{d-i}(-1)^{i}\sum_{x=1}^{i}S(i,x)x!{m_{r}\choose x}\frac{(\theta+n-xr+x\alpha)_{(m)}}{(\theta+n)_{(m)}}\\
&\nonumber\quad=\sum_{i=1}^{d}{d\choose i}m_{r}^{d-i}(-1)^{i}\sum_{x=1}^{i}S(i,x)x!{m_{r}\choose x}\frac{\frac{\Gamma(\theta+n+m-xr+x\alpha)}{\Gamma(\theta+n-xr+x\alpha)}}{\frac{\Gamma(\theta+n+m)}{\Gamma(\theta+n)}}\\
&\quad\simeq\sum_{i=1}^{d}{d\choose i}m_{r}^{d-i}(-1)^{i}\sum_{x=1}^{i}S(i,x)x!{m_{r}\choose x}\left\{\left(\frac{n}{n+m}\right)^{r-\alpha}\right\}^{x}.
\end{align}
Formula \eqref{eq:asym_mom} is the moment of order $d$ of the random variable $m_{r}-Z_{m_{r},\,(n/(n+m))^{r-\alpha}}$ with $Z_{m_{r},\,(n/(n+m))^{r-\alpha}}\sim\text{Binomial}(m_{r},\,(n/(n+m))^{r-\alpha})$ \citep[Section 3.3]{Joh(05)}. This completes the proof.


\section{Proof of Theorem~\ref{thm:wellspecified:mmle}}
\label{sec:proof:thm:wellspecified:mmle}

In the next we let $\PP_{\alpha,\theta}$ denote the joint distribution of $(P,X_1,X_2,\dots)$ under the parameter $(\alpha,\theta)$ [\textit{ie}. $P \sim \mathrm{PYP}(\alpha,\theta)$ and $X_1,X_2,\dots \mid P \overset{iid}{\sim} P$]. We also let $\alpha_{*}$ and $\theta_{*}$ defined as in Theorem~\ref{thm:4sc-jzwh-mc1}. Recall that $\lim_{x\to 0} x^{\alpha_0} \bar{F}_P(x) = S_{\alpha_0,\theta_0}/\Gamma(1-\alpha_0)$ almost-surely under $\PP_{\alpha_0,\theta_0}$. Then by Theorem~\ref{thm:4sc-jzwh-mc1} and by Corollary~\ref{cor:pypsmooth} with $\alpha_{*} = \alpha_0$, for all $\varepsilon > 0$
\begin{equation*}
  \lim_n \PP_{\alpha_0,\theta_0}\Bigg( |\hat\alpha_n - \alpha_0| \leq \varepsilon,\ |\hat\theta_n - Z| \leq \varepsilon \,\Big\vert\, P \Bigg) = 1
\end{equation*}
$\PP_{\alpha_0,\theta_0}\textrm{-as}$. By dominated convergence, for all $\varepsilon > 0$,
  \begin{align*}
    &\lim_n \PP_{\alpha_0,\theta_0}\Bigg( |\hat\alpha_n - \alpha_0| \leq \varepsilon,\ |\hat\theta_n - Z| \leq \varepsilon\Bigg)\\
    &\quad= \lim_n \EE_{\alpha_0,\theta_0}\Bigg[\PP_{\alpha_0,\theta_0}\Bigg( |\hat\alpha_n - \alpha_0| \leq \varepsilon,\ |\hat\theta_n - Z| \leq \varepsilon \,\Big\vert\, P \Bigg) \Bigg]\\
    &\quad=  \EE_{\alpha_0,\theta_0}\Bigg[\lim_n\PP_{\alpha_0,\theta_0}\Bigg( |\hat\alpha_n - \alpha_0| \leq \varepsilon,\ |\hat\theta_n - Z| \leq \varepsilon \,\Big\vert\, P \Bigg) \Bigg]\\
    &\quad= 1.
  \end{align*}
  Hence,
\begin{equation*}
  |\hat\alpha_n - \alpha_0|
  = o_p(1),\qquad |\hat\theta_n - Z| = o_p(1)
\end{equation*}
under $\PP_{\alpha_0,\theta_0}$.

\section{Proof of Theorem~\ref{thm:consistency-wellspecified-bayes}}
\label{sec:proof:thm:consis}

In the next we let $\PP_{\alpha,\theta}$ denote the joint distribution of $(P,X_1,X_2,\dots)$ under the parameter $(\alpha,\theta)$ [\textit{ie}. $P \sim \mathrm{PYP}(\alpha,\theta)$ and $X_1,X_2,\dots \mid P \overset{iid}{\sim} P$]. We also let $\hat\alpha_n$, $\hat{V}_n$, $\alpha_{*}$, $\theta_{*}$, $\phi$ and $H_{*}$ defined as in Theorems~\ref{thm:misspecified:mle} and~\ref{thm:4sc-jzwh-mc1}. Then by Theorem~\ref{thm:4sc-jzwh-mc1} and by Corollary~\ref{cor:pypsmooth} with $\alpha_{*} = \alpha_0$, for all $\varepsilon > 0$
\begin{equation*}
  \lim_n \PP_{\alpha_0,\theta_0}\Bigg( \sup_{A,B}\Big|\Pi\big(\hat{V}_n^{1/2}(\alpha - \hat{\alpha}_n) \in A,\,\gamma \in B \mid \mathbf{X}_n\big)- \phi(A)H_{*}(B) \Big| > \varepsilon \, \Big\vert \, P \Bigg) = 0
\end{equation*}
$\PP_{\alpha_0,\theta_0}\textrm{-as}$. By dominated convergence, for all $\varepsilon > 0$,%
  \begin{align*}
    &\lim_n \PP_{\alpha_0,\theta_0}\Bigg( \sup_{A,B}\Big|\Pi\big(\hat{V}_n^{1/2}(\alpha - \hat{\alpha}_n) \in A,\,\gamma \in B \mid \mathbf{X}_n\big)- \phi(A)H_{*}(B) \Big| > \varepsilon \Bigg)\\%
    &= \lim_n\EE_{\alpha_0,\theta_0}\Bigg[\PP_{\alpha_0,\theta_0}\Bigg( \sup_{A,B}\Big|\Pi\big(\hat{V}_n^{1/2}(\alpha - \hat{\alpha}_n) \in A,\,\gamma \in B \mid \mathbf{X}_n\big)- \phi(A)H_{*}(B) \Big| > \varepsilon \, \Big\vert \, P  \Bigg) \Bigg]\\
    &= \EE_{\alpha_0,\theta_0}\Bigg[\lim_n\PP_{\alpha_0,\theta_0}\Bigg( \sup_{A,B}\Big|\Pi\big(\hat{V}_n^{1/2}(\alpha - \hat{\alpha}_n) \in A,\,\gamma \in B \mid \mathbf{X}_n\big)- \phi(A)H_{*}(B) \Big| > \varepsilon \, \Big\vert \, P  \Bigg) \Bigg]\\
    &= 0.
  \end{align*}
  Hence,
\begin{equation}
  \label{eq:posterior:limit}
  \sup_{A,B}\Big|\Pi\big(\hat{V}_n^{1/2}(\alpha - \hat{\alpha}_n) \in A,\,\gamma \in B \mid \mathbf{X}_n\big)- \phi(A)H_{*}(B) \Big|%
  = o_p(1)
\end{equation}
under $\PP_{\alpha_0,\theta_0}$. The same argument combined with Theorem~\ref{thm:misspecified:mle} shows that $\hat\alpha_n = \alpha_0 + o_p(1)$ under $\PP_{\alpha_0,\theta_0}$ and that $\hat{V}_n \to \infty$ with probability going to one, which establishes the first claim of the Theorem. Regarding the second claim, remark that $H_{*}$ is a random probability measure under $\PP_{\alpha_0,\theta_0}$, and almost-surely $H_{*}$ is continuous. Hence, there exists a random variable $W$ such that
\begin{equation*}
  H_{*}\big([\theta_0 + \alpha_0 - W,\ \theta_0+\alpha_0 +  W] \big) \leq \frac{1}{2}
\end{equation*}
$\PP_{\alpha_0,\theta_0}$-as. The result then follows by \eqref{eq:posterior:limit}.

\section{Proof of Theorem~\ref{thm:lecamtheta}}
\label{sec:proof:thm:lecamtheta}

It is a classical result of Le Cam that
\begin{multline}
    \label{eq:lecam}
    \inf_{\hat\theta_n}\sup_{-\alpha < \theta \leq t}\EE_{(\alpha,\theta)}\big( (\hat{\theta}_n(\mathbf{X}_n) - \theta)^2 \big)\\%
    \geq \frac{1}{4}\sup_{-\alpha <\theta_1 < \theta_2 \leq t} (\theta_2 - \theta_1)^2\Big(1 - \mathsf{TV}(\PP_{\alpha,\theta_1}(\mathbf{X}_n \in \cdot),\PP_{\alpha,\theta_2}(\mathbf{X}_n \in \cdot)) \Big)
\end{multline}
where $\mathsf{TV}(\PP_{\alpha,\theta_1}(\mathbf{X}_n \in \cdot),\PP_{\alpha,\theta_2}(\mathbf{X}_n \in \cdot))$ denote the total variation distance between the laws of $\mathbf{X}_n$ under the parameters $(\alpha,\theta_1)$ and $(\alpha,\theta_2)$.

We first give an upper bound on $\mathsf{KL}(\PP_{\alpha,\theta_1}(\mathbf{X}_n \in \cdot);\PP_{\alpha,\theta_2}(\mathbf{X}_n \in \cdot))$; this will be enough to obtain an upper-bound on the total variation in virtue of Pinsker' inequality. See that,
\begin{align*}
  \mathsf{KL}(\PP_{\alpha,\theta_1}(\mathbf{X}_n \in \cdot);\PP_{\alpha,\theta_2}(\mathbf{X}_n \in \cdot))%
  &= - \EE_{(\alpha,\theta_1)}\big(\ell_n(\alpha,\theta_2) - \ell_n(\alpha,\theta_1) \big)\\
  &=- \EE_{(\alpha,\theta_1)}\Big(- \frac{1}{2}\partial_{\theta}^2 \ell_n(\alpha,\bar \theta)(\theta_1 - \theta_2)^2 \Big)
\end{align*}
where the last line is true for some $\bar\theta \in [\theta_1,\theta_2]$ by a Taylor expansion, and by the classical result that $\EE_{(\alpha,\theta_1)}(\partial_{\theta}\ell_n(\alpha,\theta_1)) = 0$ in regular enough models. From the expression of the likelihood given in the main document [or from \eqref{eq:j29-mj0s-4ya} below], we deduce that
\begin{align*}
    \partial_{\theta}^2\ell_n(\alpha,\theta)%
    &=\psi'(\theta+1) + \frac{1}{\alpha}\psi'(\theta/\alpha+K_n) - \frac{1}{\alpha}\psi'(\theta/\alpha+1) - \psi'(\theta+n)
\end{align*}
where $\psi'$ the the derivative of the digamma function. The function $\psi'$ is monotone decreasing on $(0,\infty)$ and non-negative \citep{abramo}. Hence, for all $\bar{\theta}\in [\theta_1,\theta_2]$
\begin{align*}
    \partial_{\theta}^2\ell_n(\alpha,\bar\theta)%
    &\geq - \frac{1}{\alpha}\psi'(\theta_1/\alpha+1).
\end{align*}
Hence, $\mathsf{KL}(\PP_{\alpha,\theta_1}(\mathbf{X}_n \in \cdot);\PP_{\alpha,\theta_2}(\mathbf{X}_n \in \cdot)) \leq \frac{\psi'(\theta_1/\alpha+1)}{2\alpha}(\theta_2-\theta_1)^2$, and by Pinsker's inequality for all $\theta_2 > \theta_1$
\begin{equation*}
    \mathsf{TV}(\PP_{\alpha,\theta_1}(\mathbf{X}_n \in \cdot),\PP_{\alpha,\theta_2}(\mathbf{X}_n \in \cdot))%
    \leq \frac{\theta_2 - \theta_1}{2}\sqrt{\frac{\psi'(\theta_1/\alpha+1)}{\alpha} }.
\end{equation*}

Given the previous upper bound, we then deduce from \eqref{eq:lecam} that
\begin{align*}
    \inf_{\hat\theta_n}\sup_{-\alpha < \theta \leq t}\EE_{(\alpha,\theta)}\big( (\hat{\theta}_n(\mathbf{X}_n) - \theta)^2 \big)%
    &\geq \frac{1}{4}\sup_{0 < \varepsilon < t+\alpha}\varepsilon^2\Big(1 - \frac{\varepsilon}{2}\sqrt{\frac{\psi'((t-\varepsilon)/\alpha + 1)}{\alpha} } \Big)\\
    &\geq \frac{1}{8}\min\Bigg(\frac{(t+\alpha)^2}{4},\ \frac{\alpha}{\psi'[(t/\alpha+1)/2 ]} \Bigg)
\end{align*}
where the last line follows by using that the supremum over $\varepsilon$ is greater than the value at $\varepsilon = \min\big(\frac{t+\alpha}{2},\, \sqrt{\alpha/\psi'[(t/\alpha+1)/2]} \big)$, and because $\psi'$ is monotone decreasing. Finally, $\psi'(z) \leq \frac{1}{z} + \frac{1}{z^2} \leq 2\max(z^{-1},z^{-2})$ for all $z > 0$ \citep{abramo}, thus
\begin{align*}
    \frac{\alpha}{\psi'[(t/\alpha+1)/2 ]}%
    &\geq \frac{\alpha}{2}\min\Bigg( \frac{t/\alpha+1}{2},\, \frac{(t/\alpha+1)^2}{4} \Bigg)%
    = \min\Bigg(\frac{t + \alpha}{4},\, \frac{(t+\alpha)^2}{8\alpha} \Bigg).
\end{align*}
The conclusion follows since $\alpha \in (0,1)$.

\section{Proof of Theorem~\ref{thm:misspecified:mle}}
\label{sec:4d9-w5my-t85}

In order to prove these results, we need to write the likelihood
$L_n$ under a convenient form. We can obtain from the equation~\eqref{eq_ewe_py}
that $L_n$ factorizes as
\begin{equation}
  \label{eq:j29-mj0s-4ya}
  L_n(\alpha,\theta)%
  = \frac{\prod_{i=0}^{K_n-1}(\theta/\alpha + i)}{\prod_{i=0}^{n-1}(\theta + i)}L_n(\alpha,0)%
  = \frac{\Gamma(\theta+1)\Gamma(\theta/\alpha + K_n)}{\Gamma(\theta/\alpha + 1)\Gamma(\theta + n)} \frac{L_n(\alpha,0)}{\alpha}.
\end{equation}
We note that $\alpha \mapsto L_n(\alpha,0)$ is the likelihood of the $\alpha$-stable Poisson-Kingman
partition model \cite{Pit(06)} which has been investigated in \cite{FN(21)}. In the
whole proof, we denote $\ell_n = \log L_n$ the log-likelihood function of the
model.

The rest of this section is dedicated to the proof of the following Proposition. The results in Theorem~\ref{thm:misspecified:mle} are immediately deduced.

\begin{prp}
  \label{prp:9cr-dlcr-2yc}
  Let $X_1,\dots,X_n \overset{iid}{\sim} p$ with
  a distribution $p$ over a set of symbols. Then, the following item are true.
  \begin{enumerate}
    \item\label{item:rph-e4g6-jcz} The mapping $\alpha \mapsto \log L_n(\alpha,0)$ is concave on
    $(0,1)$ and it admits a unique maximizer $\hat{\alpha}_n^0$ whenever $K_n \ne n$
    and $K_n \ne 1$. Furthermore, if Assumption~\ref{ass:ucj-yl04-iyy} is
    satisfied,
    \begin{equation*}
      \log(n)| \hat{\alpha}_n^0 - \alpha_{*}| = o_p(1).
    \end{equation*}
    
    \item\label{item:6x9-2jie-dzi} Let Assumption~\ref{ass:ucj-yl04-iyy} be
    satisfied. Then the set
    $\argmax_{(\alpha,\theta)\in \Phi}L_n(\alpha,\theta)$ is not empty with
    probability $1 + o(1)$ as $n\to \infty$. Furthermore,
    $(\hat{\alpha}_n,\hat{\theta}_{n}) \in \argmax_{(\alpha,\theta)\in \Phi}L_n(\alpha,\theta)$
    must satisfy
    $\hat{\alpha}_n = \hat{\alpha}_n^0 + O_p\big[\frac{\log(n)}{n^{\alpha_{*}}}\big]$
    [thus $\hat{\alpha}_n = \alpha_{*} + o_p(1)$], and
    $\hat{\theta}_n = \theta_{*} + o_p(1)$.
  \end{enumerate}
\end{prp}

\subsection{Proof of the item~\ref{item:rph-e4g6-jcz} in Proposition~\ref{prp:9cr-dlcr-2yc}}

The concavity of $\alpha \mapsto \ell_n(\alpha,0)$, uniqueness, and existence of its maximizer
when $K_n \ne 1$ and $K_n \ne n$ is a consequence of the fact that
\begin{equation}
  \label{eq:ft9-0iz5-m0m}%
  \ell_n(\alpha,0) = (K_n-1)\log(\alpha) + \sum_{i=1}^{n-1} \Bigg(\sum_{\ell=i+1}^nM_{n,\ell} \Bigg)\log(i - \alpha).
\end{equation}
In particular, we see that $\lim_{\alpha \to 0}\partial_{\alpha}\ell_n(\alpha,0) = \infty$ and
$\lim_{\alpha\to 1}\partial_{\alpha}\ell_n(\alpha,0) =-\infty$; and also $\partial_{\alpha}^2\ell_n(\alpha,0) < 0$ for all
$\alpha \in (0,1)$. The same result has been established with more details in Theorem~1
of \cite{FN(21)}.

The second claim is a classical consistency analysis of the maximum (marginal) likelihood estimator; it is also
similar to Theorem 2 in \cite{FN(21)} but here we use weaker
conditions, which changes a bit the analysis of the misspecification bias. For
the sake of completeness, we give the complete proof.

We show that for all $\epsilon > 0$ the condition $\log(n)|\hat{\alpha}_n^0 - \alpha_{*}| > \epsilon$
contradicts $\partial_{\alpha}\ell_n(\hat{\alpha}_n^0,0) = 0$ on a event of probability $1 + o(1)$.
Since $\ell_n(\cdot,0)$ is concave, $\partial_{\alpha}\ell_n(\cdot,0)$ is monotone decreasing on $(0,1)$.
Therefore,
\begin{equation*}
  \inf_{|\alpha - \alpha_{*}| > \delta}|\partial_{\alpha}\ell_n(\alpha,0)|%
  = \min\big\{|\partial_{\alpha}\ell_n(\alpha_{*} - \delta)|,\, |\partial_{\alpha}\ell_n(\alpha_{*} + \delta)| \big\}.
\end{equation*}
Hence, the theorem follows if we show that with probability
$1 + o(1)$ we have $|\partial_{\alpha}\ell_n(\alpha_{*} - \delta)| > 0$ and $|\partial_{\alpha}\ell_n(\alpha_{*} + \delta)| > 0$ for
$\delta = \epsilon/\log(n)$ [and for all $\epsilon > 0$]. We do so by proving in Proposition~\ref{pro:6} that
$\mathcal{S}_{p,n}(\alpha) = \EE_{p}[\partial_{\alpha}\ell_n(\alpha,0)]$ satisfies $\mathcal{S}_{p,n}(\alpha \pm \delta) \gg 0$
and by proving in Proposition~\ref{pro:7} that
$\var_{p}(\partial_{\alpha}\ell_n(\alpha \pm \delta)) \ll \mathcal{S}_{p,n}(\alpha \pm \delta)^2$.

\begin{prp}
  \label{pro:6}
  Let $p$ satisfy Assumption~\ref{ass:ucj-yl04-iyy}. Then, there exists a
  constant $B > 0$ depending only on $(L,\alpha_{*})$ such that for all $\epsilon > 0$ and
  all $\alpha \in (0,1)$ as $n\to \infty$
  \begin{equation*}
    \log(n)|\alpha - \alpha_{*}| > \epsilon \implies |\mathcal{S}_{p,n}(\alpha)| \geq \frac{Bn^{\alpha_{*}}|\alpha - \alpha_{*}|}{\min(\alpha,1-\alpha)}(1 + o(1)).
  \end{equation*}
\end{prp}
\begin{proof}
  By
  Lemma~\ref{lem:2},
  \begin{equation*}
    \mathcal{S}_{p,n}(\alpha)%
    =-\frac{1}{\alpha}- \sum_{k=1}^{n-1}\frac{n-k}{k-\alpha}\binom{n}{k}\int_0^1\bar{F}_{p}(x)x^k(1-x)^{n-k-1}\intd x%
    + \frac{n}{\alpha}\int_0^1\bar{F}_{p}(x)(1-x)^{n-1}\intd x.
  \end{equation*}
  We now define $\Delta(x) = \bar{F}_{p}(x) - Lx^{-\alpha_{*}}$, so that we can decompose
  $\mathcal{S}_{p,n}$ as
  \begin{align*}
    \mathcal{S}_{p,n}(\alpha)%
    &=- \frac{1}{\alpha}  + \frac{Ln}{\alpha}\int_0^1 x^{-\alpha_{*}}(1-x)^{n-1}\intd x%
      +  \frac{n}{\alpha} \int_0^1 \Delta(x)(1-x)^{n-1}\,\intd x\\
    &\quad%
    - Ln \sum_{k=1}^{n-1} \frac{n-k}{k-\alpha}\binom{n}{k}\int_0^1
      x^{k -\alpha_{*}}(1-x)^{n-k-1}\intd x\\
    &\quad%
    + \sum_{k=1}^{n-1} \frac{n-k}{k-\alpha}\binom{n}{k}%
      \int_0^1 \Delta(x)x^k(1-x)^{n-k-1}\intd x.
  \end{align*}
  That is, after splitting the last summation in two at
  $k_n = \lfloor n^{1-\alpha_{*}}\log(n)/\epsilon \rfloor$ for an arbitrary small $\epsilon > 0$,
  \begin{align*}
    \mathcal{S}_{p,n}(\alpha)%
    &=-\frac{1}{\alpha} + \frac{L\Gamma(1-\alpha_{*})\Gamma(n+1)}{\alpha\Gamma(n-\alpha_{*}+1)}%
    - \underbrace{\frac{L n!}{\Gamma(n-\alpha_{*}+1)} \sum_{k=1}^{n-1}
      \frac{\Gamma(k-\alpha_{*})}{k!} \cdot \frac{k - \alpha_{*} }{k-\alpha} }_{f_{1}(\alpha)} \\
    &\quad%
      +  \underbrace{ \frac{n}{\alpha}\int_0^1\Delta(x)(1-x)^{n-1}\,\intd x}_{f_{2}(\alpha)}%
    - \underbrace{\sum_{k=1}^{k_n} \frac{n-k}{k-\alpha}\binom{n}{k}%
      \int_0^1 \Delta(x)x^k(1-x)^{n-k-1}\intd x}_{f_{3}(\alpha)}\\
    &\quad%
       - \underbrace{\sum_{k=k_n}^{n-1} \frac{n-k}{k-\alpha}\binom{n}{k}%
      \int_0^1 \Delta(x)x^k(1-x)^{n-k-1}\intd x}_{f_{4}(\alpha)}.
  \end{align*}
  By Lemma~\ref{lem:knc-clpq-b9h} and by Stirling's formula,%
  \begin{align*}
    |f_{2}(\alpha)|%
    &= \frac{1}{\alpha} \cdot o\Big[\frac{n\Gamma(n)}{\Gamma(n+1-\alpha_{*})\log(n)} \Big]%
      = \frac{1}{\alpha} \cdot o\Big[\frac{n^{\alpha_{*}}}{\log(n)} \Big].
  \end{align*}
  Similarly,%
  \begin{align*}
    |f_{3}(\alpha)|%
    &= o\Big[\frac{1}{\log(n)} \Big]\cdot \sum_{k=1}^{k_n}\frac{n-k}{k-\alpha}\binom{n}{k} \frac{\Gamma(n-k)\Gamma(k+1-\alpha_{*})}{\Gamma(n+1-\alpha_{*})}\\
    &\leq o\Big[\frac{n^{\alpha_{*}}}{\log(n)} \Big]\cdot \sum_{k\geq 1}\frac{1}{k-\alpha} \frac{\Gamma(k+1-\alpha_{*})}{k!}\\
    &\leq \frac{1}{1 - \alpha} \cdot o\Big[\frac{n^{\alpha_{*}}}{\log(n)} \Big].
  \end{align*}
  Also,
  \begin{align*}
    |f_4(\alpha)|%
    &\leq \frac{1}{k_n-\alpha}%
      \int_0^1|\Delta(x)| \sum_{k=0}^{n-1}(n-k)\binom{n}{k}x^k(1-x)^{n-k-1}\intd x\\
    &= \frac{n}{k_n-\alpha} \int_0^1|\Delta(x)| \sum_{k=0}^{n-1}\binom{n-1}{k}x^k(1-x)^{n-k-1}\intd x\\
    &= \frac{n}{k_n-\alpha}\int_0^1|\Delta(x)|\intd x\\
    &= o\Big[\frac{n^{\alpha_{*}}}{\log(n)} \Big]
  \end{align*}
  because $\epsilon$ can be chosen as small as we want and because
  $\int_0^1\bar{F}_p(x)\intd x =1$ and $\int_0^1x^{-\alpha_{*}}\intd x < \infty$. Finally, we
  note that,
  \begin{align*}
    f_{1}(\alpha_{*})%
    &= \frac{Ln!}{\Gamma(n-\alpha_{*} + 1)} \sum_{k=1}^{n-1} \frac{\Gamma(k-\alpha_{*})}{k!}\\
    &= \frac{L \Gamma(1-\alpha_{*})\Gamma(n+1)}{\alpha_{*}\Gamma(n-\alpha_{*}+1)} - \frac{L n \Gamma(n-\alpha_{*})}{\alpha_{*}\Gamma(n-\alpha_{*}+1)}\\
    &= \frac{L \Gamma(1-\alpha_{*})\Gamma(n+1)}{\alpha_{*}\Gamma(n-\alpha_{*}+1)} + O(1),
  \end{align*}
  and,
  \begin{align*}
    f_{1}(\alpha) - f_{1}(\alpha_{*})%
    &=%
      \frac{L n!}{\Gamma(n-\alpha_{*}+1)} \sum_{k=1}^{n-1} \frac{\Gamma(k-\alpha_{*})}{k!} \Big\{ \frac{k- \alpha_{*}}{k-\alpha} - 1  \Big\}\\
    &= (\alpha - \alpha_{*})\cdot \frac{Ln!}{\Gamma(n-\alpha_{*}+1)} \sum_{k=1}^{n-1} \frac{\Gamma(k-\alpha_{*})}{k!(k-\alpha)}.
  \end{align*}
  Combining all of the previous equations, we find that,
  \begin{align*}
    \mathcal{S}_{p,n}(\alpha)%
    &= - \frac{1}{\alpha} + \Big(\frac{\alpha_{*}}{\alpha}-1\Big)\frac{L\Gamma(1-\alpha_{*})n!}{\alpha_{*}\Gamma(n-\alpha_{*} + 1)}\\
    &\quad
      -(\alpha - \alpha_{*})\cdot \frac{Ln!}{\Gamma(n-\alpha_{*}+1)} \sum_{k=1}^{n-1} \frac{\Gamma(k-\alpha_{*})}{k!(k-\alpha)}%
      + o\Big[\frac{n^{\alpha_{*}}}{\min(\alpha,1-\alpha)\log(n)} \Big].
  \end{align*}
  The conclusion of the proposition follows from the last display.
\end{proof}

\begin{prp}
  \label{pro:7}
  The following is true for all $\alpha \in (0,1)$ and all $p$
  \begin{equation*}
    \var_{p}(\partial_{\alpha}\ell_n(\alpha,0))%
    \leq \max\Big\{\frac{1}{\alpha^2},\frac{2}{(1-\alpha)^2}\Big\}\EE_p[K_n].
  \end{equation*}
  We also recall the result from \cite{Gne(07)} that if
  $\bar{F}_p(x) \sim Lx^{-\alpha^{*}}$ as $x \to 0$ [which is the case under
  Assumption~\ref{ass:ucj-yl04-iyy}] then $\EE_p[K_n] \sim L\Gamma(1-\alpha_{*})n^{\alpha_{*}}$ as
  $n\to \infty$.
\end{prp}
\begin{proof}
  We bound the variance of $\partial_{\alpha}\ell_n(\alpha)$ using an Efron-Stein argument. For simplicity
  we write $\varphi \equiv -\alpha \partial_{\alpha}\ell_n(\alpha)$. For every $n =1,\dots,n$, we define the random variables
  $C_{n,k} = \sum_{\ell=k}^nM_{n,\ell} = \sum_{j\geq 1}\1_{Y_{n,j} \geq k}$,
  $Y_{n,j}^{(i)} = Y_{n,j} - \1_{X_i = j}$, and
  $C_{n,k}^{(i)} = \sum_{j\geq1}^n\1_{Y_{n,j}^{(i)}\geq k}$. We note that
  $Y_{n,j}^{(i)}$ does not depend on $X_i$, so does $C_{n,k}^{(i)}$. Also remark
  that $\varphi = \sum_{k=1}^{n-1}\frac{\alpha}{k-\alpha}C_{n,k+1} - C_{n,1}$. Defining
  $\varphi_i = \sum_{k=1}^{n-1}\frac{\alpha}{k-\alpha}C_{n,k+1}^{(i)} - C_{n,1}^{(i)}$,
  we have by Efron-Stein's bound
  \cite[Theorem~3.1]{boucheron:lugosi:massart:2013}%
  \begin{align*}
    \var_{p}(\varphi)%
    &\leq%
      \sum_{i=1}^n\EE_{p}[(\varphi - \varphi_i)^2].
  \end{align*}
  But,
  \begin{align*}
    C_{n,k} - C_{n,k}^{(i)}%
    &= \sum_{j\geq 1}(\1_{Y_{n,j}\geq k} - \1_{Y_{n,j}^{(i)}\geq k})
    = \sum_{j\geq 1}\1_{Y_{n,j} = k}\1_{X_i = j},
  \end{align*}
  where the second equality follows because $Y_{n,j} \geq Y_{n,j}^{(i)} \geq Y_{n,j} - 1$ almost-surely, hence
  $\1_{Y_{n,j}\geq k} - \1_{Y_{n,j}^{(i)}\geq k}$ is either zero or one, and it is one
  iff $Y_{n,j}= k$ and $Y_{n,j}^{(i)} = k-1$; that is iff $Y_{n,j}=k$ and $X_i = j$. Therefore,
  \begin{align*}
    (\varphi - \varphi_i)^2%
    &= \Big(\sum_{k=1}^{n-1} \frac{\alpha}{k-\alpha}(C_{n,k+1} - C_{n,k+1}^{(i)}) - (C_{n,1} - C_{n,1}^{(i)})  \Big)^2\\
    &= \Big\{ \sum_{j\geq 1}\1_{X_i = j}\Big(\sum_{k=1}^{n-1} \frac{\alpha'}{k -\alpha'}\1_{Y_{n,j}=k+1} - \1_{Y_{n,j}=1} \Big) \Big\}^2.
  \end{align*}
  It follows,
  \begin{align*}
    \sum_{i=1}^n (\varphi - \varphi_i)^2%
    &= \sum_{j\geq 1}\Big(\sum_{i=1}^n\1_{X_i=j}\Big)\Big(\sum_{k=1}^{n-1} \frac{\alpha}{k -\alpha}\1_{Y_{n,j}=k+1} - \1_{Y_{n,j}=1} \Big)^2\\
    &= \sum_{j\geq 1}Y_{n,j}\Big(\sum_{k=1}^{n-1} \frac{\alpha}{k -\alpha}\1_{Y_{n,j}=k+1} - \1_{Y_{n,j}=1} \Big)^2\\
    &= \sum_{j\geq 1}\1_{Y_{n,j} = 1} + \sum_{j\geq 1}Y_{n,j}\1_{Y_{n,j}\ne 1}\Big(\sum_{k=1}^{n-1} \frac{\alpha}{k -\alpha}\1_{Y_{n,j}=k+1}\Big)^2\\
    &= M_{n,1}%
      + \sum_{k=1}^{n-1}\sum_{k'=1}^{n-1} \frac{\alpha}{k-\alpha}\frac{\alpha}{k'-\alpha} \sum_{j\geq 1}Y_{n,j}\1_{Y_{n,j}\ne 1}\1_{Y_{n,j}=k+1}\1_{Y_{n,j}=k'+1}\\
    &= M_{n,1} + \sum_{k=1}^{n-1} \Big( \frac{\alpha}{k - \alpha} \Big)^2 \sum_{j\geq 1}Y_{n,j}\1_{Y_{n,j}=k+1}.
  \end{align*}
  Hence we have shown,
  \begin{align*}
    \sum_{i=1}^n \EE_{p}[(\varphi - \varphi_i)^2]%
    &= \EE_{p}[M_{n,1}] + \sum_{k=1}^{n-1} \Big( \frac{\alpha}{k - \alpha} \Big)^2 (k+1)\EE_{p}[M_{n,k+1}]\\
    &\leq \EE_p[M_{n,1}] + \frac{2\alpha^2}{(1-\alpha)^2}\sum_{k=1}^{n-1}\EE_p[M_{n,k+1}]\\
    &\leq \max\Big\{1,\, \frac{2\alpha^2}{(1-\alpha)^2} \Big\}\EE_p[K_n],
  \end{align*}
  where the second line follows because $(k+1)/(k-\alpha)^2\leq 2/(1-\alpha)^2$ for all
  $k \geq 1$; and the third line because $\sum_{\ell=1}^nM_{n,\ell} = K_n$.
\end{proof}

\subsection{Auxiliary results for the proof of the item~\ref{item:rph-e4g6-jcz}
  in Proposition~\ref{prp:9cr-dlcr-2yc}}
\label{sec:54l-mu7q-4ek}

This section gathers a series of Lemma that are used in the proof of the Proposition~\ref{pro:6}.

\begin{lem}
  \label{lem:2}
  For any $p$, any $n \geq 1$, and any $0 \leq k \leq n-1$,
  \begin{equation*}
    \EE_{p}\Big[ \sum_{\ell=k+1}^n M_{n,\ell} \Big]%
    =%
    (n-k)\binom{n}{k}\int_0^{1}\bar{F}_{p}(x) x^{k}(1-x)^{n-k -1}\intd x.
  \end{equation*}
\end{lem}
\begin{proof}
  Let $Y_{n,j} = \sum_{i\geq 1}\1_{X_i=j}$ so that
  $M_{n,\ell} \overset{d}{=} \sum_{j\geq 1}\1_{Y_{n,j}=\ell}$. Thus,
  \begin{align*}
    \EE_{p}[M_{n,\ell}]%
    &= \sum_{j\geq
      1}\EE_{p}[\1_{Y_{n,j}=\ell}]\\
    &= \binom{n}{\ell}\sum_{j\geq
      1}p_j^{\ell}(1-p_j)^{n-\ell}\\
    &= \binom{n}{\ell}\sum_{j\geq 1}\int_0^{p_j}\Big\{\ell
      x^{\ell-1}(1-x)^{n-\ell} - (n-\ell)x^{\ell}(1-x)^{n-\ell-1}\Big\}\intd x.
  \end{align*}
  Since $\bar{F}_{p}(x) = \sum_{j\geq 1}\1_{p_j > x}$  we deduce that
  \begin{equation*}
    \EE_{p}[M_{n,\ell}]%
    =%
    \binom{n}{\ell}\int_0^{1}\bar{F}_{p}(x) x^{\ell-1}(1-x)^{n-\ell -1} (\ell - nx)\intd x.
  \end{equation*}
  Finally, we observe that
  \begin{equation*}
    \sum_{\ell=k+1}^n \binom{n}{\ell} x^{\ell-1}(1-x)^{n-\ell-1}(\ell - nx)%
    = (n-k)\binom{n}{k}x^k(1-x)^{n-k -1}. \qedhere%
  \end{equation*}
\end{proof}

\begin{lem}
  \label{lem:knc-clpq-b9h}
  Suppose $(k_n)$ is a sequence such that for some $0 < \eta \leq 1$ it holds
  $k_n = O(n^{1-\eta})$ as $n \to \infty$. Then, under Assumption~\ref{ass:ucj-yl04-iyy}
  \begin{equation*}
    \sup_{k=0,\dots,k_n}%
    \frac%
    {\int_{0}^{1}|\bar{F}_p(x) - Lx^{-\alpha_{*}}| x^{k} (1-x)^{n-k-1}\intd x}%
    {\Gamma(n-k)\Gamma(k+1-\alpha_{*})/\Gamma(n+1-\alpha_{*})}%
    = o\Big[\frac{1}{\log(n)} \Big].
  \end{equation*}
\end{lem}
\begin{proof}
  For all $\epsilon > 0$ there is $x_0 > 0$ such that
  $|\bar{F}_p(x) - Lx^{-\alpha_{*}}| \leq \epsilon x^{-\alpha_{*}}/\log(1/x)$ for all $x \in (0,x_0)$.
  Further, $\bar{F}_p$ is monotone decreasing and $x \mapsto x^{-\alpha_{*}}$ too. Therefore,
  \begin{align*}
    \int_{0}^{1}|\bar{F}_p(x) - Lx^{-\alpha_{*}}| x^{k} (1-x)^{n-k-1}\intd x%
    &\leq \epsilon \int_0^{x_0}\frac{x^{k-\alpha_{*}}(1-x)^{n-k-1}\intd x}{\log(1/x)}\\
    &\quad%
      + \big(\bar{F}_p(x_0) + Lx_0^{-\alpha_{*}} \big) \int_{x_0}^1x^k(1-x)^{n-k-1}\intd x.
  \end{align*}
  But,
  \begin{align*}
    \int_{x_0}^1x^k(1-x)^{n-k-1}\intd x%
    &\leq \int_{0}^1x^k(1-x)^{n-k-1}\intd x\\
    &= \frac{\Gamma(k+1)\Gamma(n-k)}{\Gamma(n+1)}\\
    &= \frac{\Gamma(n-k)\Gamma(k+1-\alpha_{*})}{\Gamma(n+1-\alpha_{*})}\cdot%
      \frac{\Gamma(k+1)\Gamma(n+1-\alpha_{*})}{\Gamma(n+1)\Gamma(k+1-\alpha_{*})}.
  \end{align*}
  Hence, for a universal constant $A > 0$ [in particular not depending on $x_0$
  nor $k_n$]
  \begin{align*}
    \sup_{k=1,\dots,k_n}\frac{\int_{x_0}^1x^k(1-x)^{n-k-1}\intd x}%
    {\Gamma(n-k)\Gamma(k+1-\alpha_{*})/\Gamma(n+1-\alpha_{*})}%
    \leq \frac{Ak_n^{\alpha_{*}}}{n^{\alpha_{*}}}.
  \end{align*}
  On the other hand, for any $c > 1$ [since $k_n/n \to 0$ we assume without loss
  of generality that $k/n \ll x_0$]
  \begin{align*}
    \int_0^{x_0}\frac{x^{k-\alpha_{*}}(1-x)^{n-k-1}\intd x}{\log(1/x)}%
    &= \int_0^{ck/n} \frac{x^{k-\alpha_{*}}(1-x)^{n-k-1}\intd x}{\log(1/x)}%
      + \int_{ck/n}^{x_0} \frac{x^{k-\alpha_{*}}(1-x)^{n-k-1}\intd x}{\log(1/x)}\\
    &\leq \frac{\int_0^1x^{k-\alpha_{*}}(1-x)^{n-k-1}\intd x }{\log[n/(ck)]}%
      + \frac{\int_{ck/n}^1 x^{k-\alpha_{*}}(1-x)^{n-k-1}\intd x}{\log(1/x_0)}\\
    &= \frac{\Gamma(n-k)\Gamma(k+1-\alpha_{*})}{\Gamma(n+1-\alpha_{*})\log[n/(ck)]}%
      + \frac{\int_{ck/n}^1 x^{k-\alpha_{*}}(1-x)^{n-k-1}\intd x}{\log(1/x_0)}.
  \end{align*}
  But, introducing $Z$ a random variable that has a Beta distribution with
  parameters $k+1 - \alpha_{*}$ and $n-k$,
  \begin{align*}
    \int_{ck/n}^1 x^{k-\alpha_{*}}(1-x)^{n-k-1}\intd x%
    &= \frac{\Gamma(n-k)\Gamma(k+1-\alpha_{*})}{\Gamma(n+1-\alpha_{*})}\PP\big(Z > ck/n \big)
  \end{align*}
  Note that $\EE[Z] = (k+1-\alpha_{*})/(n+1-\alpha_{*})$, and thus if $c$ is taken large
  enough [but depending only on $\alpha_{*}$]
  \begin{align*}
    \sup_{k=1,\dots,k_n}\PP\big(Z > ck/n\big)%
    &= \sup_{k=1,\dots,k_n}\PP\big(Z - \EE[Z] > ck/(2n)\big)\\
    &\leq \sup_{k=1,\dots,k_n} \frac{\var(Z)}{c^2 k^2/(4n^2)}\\
    &= \sup_{k=1,\dots,k_n} \frac{4n^2 (k+1-\alpha_{*})(n-k)}{c^2k^2(n+1-\alpha_{*})^2(n+2-\alpha_{*})}\\
    &\leq \frac{8}{c^2}.
  \end{align*}
  Taking for instance $c = \sqrt{n/k}$, we obtain finally,
  \begin{align*}
    \sup_{k=1,\dots,k_n}%
    \frac%
    {\int_{0}^{1}|\bar{F}_p(x) - Lx^{-\alpha_{*}}| x^{k} (1-x)^{n-k-1}\intd x}%
    {\Gamma(n-k)\Gamma(k+1-\alpha_{*})/\Gamma(n+1-\alpha_{*})}%
    &\leq \frac{2\epsilon}{\log(n/k_n)}%
      + \frac{8\epsilon}{\log(1/x_0)}\frac{k_n}{n}\\%
    &\quad%
      + A\big(\bar{F}_p(x_0) + Lx_0^{-\alpha_{*}} \big) \frac{k_n^{\alpha_{*}}}{n^{\alpha_{*}}}.
  \end{align*}
  Since $k_n/n$ decays polynomially in $n$ and since this is true for all
  $\epsilon > 0$, the conclusion follows.
\end{proof}

\subsection{Proof of the item~\ref{item:6x9-2jie-dzi} in Proposition~\ref{prp:9cr-dlcr-2yc}}
\label{sec:6wf-urz8-luq}

We prove the item~\ref{item:6x9-2jie-dzi} in several steps. We define two
subsets of $\Phi$ by $\Phi_1 = \{(\alpha,\theta) \in \Phi\,:\, |\alpha - \hat{\alpha}_n^0| \leq b_1n^{-\delta} \}$ and
$\Phi_2 = \{(\alpha,\theta) \in \Phi\,:\, -\hat{\alpha}_n^0 + b_2 \leq \theta \leq b_2' \log(n) \}$, where
$b_1,b_2,b_2' > 0$ and $\delta > 0$ are constants to be chosen accordingly. Then, we
show that no maximizer of $L_n$ can be found in $\Phi_1^c$
[Proposition~\ref{pro:kl0-860e-jpy}], and neither in $\Phi_2^c$
[Proposition~\ref{pro:x8d-d96y-u1g}]. As a consequence, if there is a maximizer,
it must be in $\Phi_1 \cap \Phi_2$. We analyze $L_n$ over the domain $\Phi_1 \cap \Phi_2$ in
Proposition~\ref{pro:5i7-gs99-awg} which establishes the final result.

In the whole proof, we use the well-known fact \citep{Gne(07)} that $K_n \asymp n^{\alpha_{*}}$
and $M_{n,2} \asymp n^{\alpha_{*}}$ on an event of probability $1 + o(1)$. Hence, from now
we assume without loss of generality that $K_n \asymp n^{\alpha_{*}}$ and
$M_{n,2} \asymp n^{\alpha_*}$.

\begin{prp}
  \label{pro:kl0-860e-jpy}
  If the constant $b_1 > 0$ is large enough and the constant $\delta > 0$ is small
  enough, then there exists a constant $c > 0$ such that with probability
  $1 + o(1)$ as $n\to \infty$%
  \begin{equation*}
    \sup_{(\alpha,\theta)\in \Phi_1^c}\{\ell_n(\alpha,\theta) - \sup \ell_n\}%
    \leq -c n^{\alpha_{*}-\delta}.
  \end{equation*}
\end{prp}
\begin{proof}
  We first get a very rough localization and show that $\hat{\alpha}_n$ must be
  within $[c_0,c_1]$ with probability $1 + o(1)$, for suitable constants
  $0 < c_0 < c_1 < 1$. By combining the results of the
  Lemma~\ref{lem:ny9-3mjf-qpx}, \ref{lem:xbo-ffe5-o0m} and
  \ref{lem:205-743d-w1y}, we find that
  \begin{align*}
    \sup_{\alpha \in [c_0,c_1]^c}\sup_{\theta > -\alpha}\{\ell_n(\alpha,\theta) - \sup \ell_n\}%
    &\leq \sup_{\alpha\in [c_0,c_1]^c}\{-2\log(\alpha) + \ell_n(\alpha,0) - \ell_n(\hat{\alpha}_n^0,0)\}%
      + O_p(n^{\alpha_{*}}).
  \end{align*}
  For $c_0$ small enough (but bounded away from zero) and $c_1$ large enough
  (but bounded away from 1), the equation~\eqref{eq:ft9-0iz5-m0m} and the fact
  that $K_n \asymp n^{\alpha_{*}}$ and $M_{n,2} \asymp n^{\alpha_{*}}$ guarantee that
  the last display is smaller than $\lesssim -n^{\alpha_{*}}$ on an event of probability
  $1 + o(1)$.

  Then it is enough to consider $\alpha \in [c_0,c_1]$ and refine the previous result.
  We start with the region $\theta > c_2$ for some $c_2 > 0$ large enough. We can use
  Stirling's formula to obtain
 \begin{align*}
   \log \Big\{\frac{\Gamma(\theta+1)}{\Gamma(\theta/\alpha + 1)} \frac{\Gamma(\theta/\alpha+K_n)}{\Gamma(\theta + n)} \Big\}%
   &= (\theta + 1)\log(\theta + 1) + \frac{1}{2}\log \frac{\theta +n}{\theta+1}\\%
   &\quad- (\theta/\alpha + 1)\log(\theta/\alpha + 1)  - \frac{1}{2}\log \frac{\theta/\alpha + K_n}{\theta/\alpha + 1}\\
   &\quad+(\theta/\alpha + K_n)\log (\theta/\alpha + K_n) - K_n\\
   &\quad-(\theta + n)\log(\theta + n) +  n +  O\Big( \frac{1}{c_2} \Big).
 \end{align*}
 We observe that,
 \begin{align*}
   \frac{1}{2}\log \frac{\theta+n}{\theta+1} - \frac{1}{2}\log \frac{\theta/\alpha + K_n}{\theta/\alpha + 1}%
   = O\big( \log(n) \big).
 \end{align*}
 Hence
 \begin{align}
   \label{eq:krk-490w-bkj}
   \ell_n(\alpha,\theta)%
   \leq h_{n,\alpha}(\theta) + \ell_n(\alpha,0) + n - K_n + O\big(\log(n) \big),
 \end{align}
 where
 \begin{align*}
  h_{n,\alpha}(\theta)%
  &= (\theta + 1)\log(\theta + 1)%
    - (\theta/\alpha + 1)\log(\theta/\alpha + 1)\\
   &\quad%
  +(\theta/\alpha + K_n)\log (\theta/\alpha + K_n)
  -(\theta + n)\log(\theta + n).
 \end{align*}
 We now analyze the function $h_{n,\alpha}$. First,
 \begin{align}
   \notag
   h_{n,\alpha}'(\theta)%
   &= \frac{1}{\alpha} \log \frac{\theta/\alpha + K_n}{\theta/\alpha + 1}%
     - \log \frac{\theta +n}{\theta +1}\\
   \label{eq:mhs-nbkh-zxg}
   &= \frac{1}{\alpha}\log\Big[1 + \frac{K_n-1}{\theta/\alpha + 1}\Big]%
     - \log \Big[1 + \frac{n-1}{\theta+1} \Big]
 \end{align}
 We want to ensure the unique existence of a solution to $h_{n,\alpha}'(\theta) = 0$ and
 guarantee that solution maximizes $h_{n,\alpha}$. We differentiate a second time
 \begin{align*}
   h_{n,\alpha}''(\theta)%
   &= \frac{1}{\theta+1} - \frac{1}{\theta +n}%
     + \frac{1}{\alpha^2}\frac{1}{\theta/\alpha + K_n}%
     - \frac{1}{\alpha^2}\frac{1}{\theta/\alpha + 1}\\
   &= \frac{n-1}{(\theta+1)(\theta+n)}%
     - \frac{1}{\alpha^2}\frac{K_n-1}{(\theta/\alpha + K_n)(\theta/\alpha + 1)}\\
   &= \frac{n-1}{(\theta+1)(\theta+n)}%
     - \frac{K_n-1}{(\theta + \alpha K_n)(\theta + \alpha)}\\
   &= \frac{(n-1)(\theta + \alpha K_n)(\theta +\alpha) - (K_n-1)(\theta+1)(\theta+n)}{(\theta+1)(\theta+n)(\theta + \alpha K_n)(\theta+\alpha)}.
 \end{align*}
 So, the sign of $h_{n,\alpha}''(\theta)$ is the same as the sign of (recall
 $\theta > -\alpha > -1$)
 \begin{align*}
   N_{\theta}
   &= (n-1)(\theta + \alpha K_n)(\theta +\alpha) - (K_n-1)(\theta+1)(\theta+n)\\
   &= A\theta^2 %
     + B\theta%
     +C,
 \end{align*}
 where $A = n - K_n$, $B = (K_n+1)(n-1)\alpha - (n+1)(K_n-1)$, and
 $C = -n(K_n-1) + (n-1)K_n\alpha^2$. We observe that $A = n(1 + o_p(1))$,
 $B = -(1-\alpha)n K_n + O_p(n)$, and $C = -n K_n(1-\alpha^2) + O_p(n)$. By studying the
 polynomial $x \mapsto Ax^2 + Bx + C$ on $\Reals$, we find that it has two roots,
 given by
 \begin{equation*}
   r_1%
   = K_n(1-\alpha)(1 + o_p(1)),\qquad%
   r_2 = - \frac{1-\alpha^2}{1-\alpha} + o_p(1).
 \end{equation*}
 Remark that $r_2 + \alpha = -1 + o_p(1)$, so indeed with probability $1 + o(1)$ the
 root $r_2$ is strictly smaller than $-\alpha$. Hence, with probability $1 + o(1)$
 the equation $h_{n,\alpha}''(\alpha) = 0$ has only one solution on $(-\alpha,\infty)$. Further,
 $A > 0$ with probability $1 + o(1)$, then on a event of probability $1 + o(1)$
 we have $h_{n,\alpha}'' < 0$ on $(-\alpha,r_1)$ and $h_{n,\alpha}''(\theta) > 0$ on $(r_1,\infty)$.
 Also,
 \begin{align*}
   h_{n,\alpha}'(r_1)%
   &= \frac{1}{\alpha}\log\Big\{1 + \frac{K_n(1+o_p(1))}{(1-\alpha)K_n/\alpha} \Big\}%
     - \log\Big\{1 + \frac{n(1+o_p(1))}{(1-\alpha)K_n} \Big\}\\%
   &= - \log(n/K_n) + O_p(1).
 \end{align*}
 On the other hand with probability $1 + o(1)$ it is also the case that
 $\lim_{\theta \to-\alpha}h_{n,\alpha}'(\theta) = +\infty$ and $\lim_{\theta\to \infty} h_{n,\alpha}'(\theta) = 0$. Therefore, we
 have established that with probability $1 + o(1)$ we have $h_{n,\alpha}'$ is
 monotonically decreasing on $(-\alpha,r_1)$ with $\lim_{\theta \to -\alpha}h_{n,\alpha}'(\theta) = +\infty$ and
 $h_{n,\alpha}'(r_1) < 0$; and $h_{n,\alpha}'$ is monotonically increasing on $(r_1,\infty)$
 with $h_{n,\alpha}'(r_1) < 0$ and $\lim_{\theta \to \infty}h_{n,\alpha}'(\theta) = 0$. Thus, in
 probability, $h_{n,\alpha}'(\theta) = 0$ has a unique solution $\bar{\theta}$, and this
 solution must be within $(-\alpha,r_1)$. Clearly $\bar{\theta}$ is the global maximizer
 of $h_{n,\alpha}$ and it must be the case that $\bar{\theta} = o_p(K_n)$. Then, by
 equation~\eqref{eq:mhs-nbkh-zxg}, we deduce that $\bar{\theta}$ must satisfy
 \begin{align}
   \label{eq:c5y-097y-63j}
   \frac{\bar{\theta} + \alpha}{(\bar{\theta}+1)^{\alpha}} = \frac{\alpha K_n}{n^{\alpha}}(1 + o_p(1)).
 \end{align}
 So letting $E$ large enough, either we have $K_n \leq E n^{\alpha}$ and it must be that
 $\bar{\theta} \in (-\alpha,D)$ for some constant $D$, or $K_n > E n^{\alpha}$ and we must have
 that $\bar{\theta}^{1-\alpha} \asymp n^{-\alpha} K_n = (K_n/n)^{\alpha} K_n^{1-\alpha}$. This establishes that
 $\bar{\theta} = O_p\big[\max(1, K_n(K_n/n)^{\alpha/(1-\alpha)} )\big]$. It follows that,
 \begin{align*}
   \sup_{\theta > -\alpha}h_{n,\alpha}(\theta)%
   &\leq (\bar{\theta}/\alpha + K_n)\log(\bar{\theta}/\alpha + K_n) - (\bar{\theta} + n)\log(\bar{\theta} + n)\\
   &= K_n \log K_n - n \log n + O_p\Big[\bar{\theta} \max\Big(1, \log \frac{K_n^{1/\alpha}}{n} \Big) \Big]\\
   &= K_n \log K_n - n \log n + O_p\big( n^{\alpha_{*}-\delta} \big)
 \end{align*}
 for some $\delta > 0$, because $K_n \asymp n^{\alpha_*}$. Then, by combining the
 last display with equation~\eqref{eq:krk-490w-bkj} and the result of the
 Lemma~\ref{lem:ny9-3mjf-qpx} [after applying Stirling's formula to both the
 $\log\Gamma(K_n)$ and $\log \Gamma(n)$ terms], we find
 \begin{align*}
   \sup_{\substack{(\alpha,\theta) \in \Phi_1^c\\\theta > c_2}}\{\ell_n(\alpha,\theta) - \sup \ell_n\}%
   \leq \sup_{|\alpha - \hat{\alpha}_n^0| > b_1n^{-\delta}}\{\ell_n(\alpha,0) - \ell_n(\hat{\alpha}_n,0)\}%
   + O_p(n^{\alpha_{*} - \delta}).
 \end{align*}
 The equation~\eqref{eq:ft9-0iz5-m0m} and the fact that $K_n \asymp n^{\alpha_{*}}$ and
 $M_{n,2} \asymp n^{\alpha_{*}}$ guarantee that the last display is smaller than
 $\lesssim -n^{\alpha_{*}-\delta}$ on an event of probability $1 + o(1)$ if $b_1$ is large
 enough.

 It remains to consider the region $-\alpha < \theta \leq c_2$ and $\alpha \in [c_0,c_1]$. But in this
 case $\log \Gamma(\theta + 1) \lesssim 1$, $\log \Gamma(\theta/\alpha + 1) \gtrsim 1$, and by Stirling's formula
 $\log \Gamma(\theta/\alpha + K_n) = \log \Gamma(K_n) + (\theta/\alpha)\log(K_n) + o_p(1)$ and
 $\log(\theta / n) = \log \Gamma(n) + \theta \log(n) + o(1)$. By Lemma~\ref{lem:ny9-3mjf-qpx}
 \begin{align*}
   \sup_{\substack{(\alpha,\theta) \in \Phi_1^c\\\theta \leq c_2}}\{\ell_n(\alpha,\theta) - \sup \ell_n\}%
   &\leq \sup_{|\alpha - \alpha_n^0| > b_1n^{-\delta}}\{\ell_n(\alpha,0) - \ell_n(\hat{\alpha}_n,0)\}%
   + O_p\big( \log(n)\big).
 \end{align*}
 Hence the conclusion follows by the same argument as before.
\end{proof}

\begin{prp}
  \label{pro:x8d-d96y-u1g}
  There exists a constant $c > 0$ such that for all $0 < b_2 \ll \hat{\alpha}_n^0$
  small enough and $b_2' > 0$ large
  enough
  \begin{equation*}
    \sup_{(\alpha,\theta)\in \Phi_1 \cap \Phi_2^c}\{\ell_n(\alpha,\theta) - \sup \ell_n\}%
    \leq -c\min\Big\{ b_2' \log(n),\, \log\Big(\frac{\hat{\alpha}_n^0}{b_2}\Big) \Big\}.
  \end{equation*}
  with probability $1 + o(1)$ as $n \to \infty$
\end{prp}
\begin{proof}
  We first consider the region where $\theta > b_2' \log(n)$, we note
  that the bound of the equation~\eqref{eq:krk-490w-bkj} in the proof of
  Proposition~\ref{pro:kl0-860e-jpy} holds here too. Further, we have proven
  that $h_{n,\alpha}$ has a unique maximizer $\bar{\theta}$ which must satisfy the
  equation~\eqref{eq:c5y-097y-63j}. Here, we have in addition that
  $|\alpha - \hat{\alpha}_n^0| \lesssim n^{-\delta}$, and by the item~\ref{item:rph-e4g6-jcz} of the
  Theorem $\log(n)|\hat{\alpha}_n^0 - \alpha_*| = o_p(1)$. Since $K_n \asymp n^{\alpha_*}$,
  it must be that $\bar{\theta} = O_p(1)$. We also proved that $h_{n,\alpha}'(\theta) < 0$ for
  all $\theta > \bar{\theta}$, guaranteeing for all $z > 0$
  \begin{align*}
    \sup_{\theta > \bar{\theta} + z} h_{n,\alpha}(\theta)%
    &= h_{n,\alpha}(\bar{\theta} + z)%
      = h_{n,\alpha}(\bar{\theta}) + \frac{1}{2}h_{n,\alpha}''(\bar{\theta} + \bar{z})z^2
  \end{align*}
  for some $\bar{z} \in (0,z)$, by a Taylor expansion and because $\bar{\theta}$
  maximizes $h_{n,\alpha}$. By definition of $h_{n,\alpha}$, and since $\bar{\theta} = O_p(1)$
  and $K_n/n^{\alpha} \asymp 1$
  \begin{align*}
    h_{n,\alpha}(\bar{\theta})%
    &= K_n \log K_n - n \log(n)%
      + O_p(1).
  \end{align*}
  By the computations in the proof of Proposition~\ref{pro:kl0-860e-jpy}, if $\theta \leq 1$ we
  have
  \begin{align*}
    h_{n,\alpha}''(\theta)%
    &= \frac{A \theta^2 + B\theta +C}{(\theta+1)(\theta+n)(\theta+\alpha K_n)(\theta + \alpha)}\\
    &=(1 + o_p(1)) \frac{n \theta^2 - (1-\alpha)n K_n \theta - (1-\alpha^2)n K_n}{(\theta+1)(\theta+n)(\theta+\alpha K_n)(\theta + \alpha)}\\
    &\leq (1 + o_p(1)) \frac{- (1-\alpha^2)n K_n}{(\theta+1)(\theta+n)(\theta+\alpha K_n)(\theta + \alpha)}%
      \lesssim -1,
  \end{align*}
  while if $1 < \theta \leq o(K_n)$
  \begin{align*}
    h_{n,\alpha}''(\theta)%
    &\leq (1 + o_p(1)) \frac{- (1-\alpha)n K_n \theta}{(\theta+1)(\theta+n)(\theta+\alpha K_n)(\theta + \alpha)}
    \lesssim \frac{-1}{\theta}.
  \end{align*}
  Taking $z = D\log(n)$ for some constant $D > 0$ large enough, we necessarily
  have that $\bar{\theta} + \bar{z} = O(\log(n)) = o_p(K_n)$. It follows the existence
  of a universal constant $E > 0$ (in particular not depending on $D$), such
  that with probability $1 + o(1)$
  \begin{align*}
    \sup_{\theta > \bar{\theta} + D\log(n)}h_{n,\alpha}(\theta)%
    &\leq K_n \log K_n  -  n \log(n) - E\cdot D\log(n).
  \end{align*}
  By combining the last display with equation~\eqref{eq:krk-490w-bkj} and the
  result of the Lemma~\ref{lem:ny9-3mjf-qpx} [after applying Stirling's formula
  to both the $\log\Gamma(K_n)$ and $\log \Gamma(n)$ terms], when $b_2' > 0$ is large
  enough there is a constant $c > 0$ such that with probability $1 + o(1)$
  \begin{equation*}
    \sup_{\substack{(\alpha,\theta)\in \Phi_1\\\theta > b_2'\log(n)}}\{\ell_n(\alpha,\theta) - \sup \ell_n\}%
    \leq -c b_2' \log(n).
  \end{equation*}

  It remains to consider the region where $\theta < -\hat{\alpha}_n^0 + b_2$ and
  $(\alpha,\theta) \in \Phi_1$. But by Stirling's formula,
  \begin{align*}
    \log \Gamma(\theta/\alpha + K_n)%
    &= (\theta/\alpha + K_n)\log(\theta/\alpha + K_n) - (\theta/\alpha + K_n)%
      + \frac{1}{2}\log \frac{2\pi}{\theta/\alpha + K_n} + O(K_n^{-1})\\
    &= \frac{\theta \log K_n}{\alpha}
      + K_n \log K_n - K_n + \frac{1}{2}\log \frac{2\pi}{K_n}%
      + O\Big[\frac{\max(1,\theta^2)}{K_n} \Big],
  \end{align*}
  and
  \begin{align*}
    \log \Gamma(\theta + n)%
    &= \theta \log(n) + n \log(n) - n +\frac{1}{2}\log \frac{2\pi}{n}%
      + O\Big[\frac{\max(1,\theta^2)}{n} \Big].
  \end{align*}
  So we have because of Lemma~\ref{lem:ny9-3mjf-qpx} and another application of
  Stirling's formula%
  \begin{multline}
    \label{eq:12n-q8jt-x2q}
    \ell_n(\alpha,\theta) - \sup \ell_n%
    \leq \Psi_n(\alpha,\theta) - \Psi_n(\hat{\alpha}_n^0,\hat{T}_n)\\%
     - \log \frac{\alpha}{\hat{\alpha}_n^0}
      + \ell_n(\alpha,0) - \ell_n(\hat{\alpha}_n^0,0)%
      + O\Big[ \frac{\max(1,\theta^2)}{K_n} \Big].
  \end{multline}
  But $\Psi_n(\hat{\alpha}_n^0,\hat{T}_n) = O_p(1)$ by
  Lemma~\ref{lem:ny9-3mjf-qpx}, $\ell_n(\alpha, 0) - \ell_n(\hat{\alpha}_n^0,0) \leq 0$ by
  construction, and $\lim_{\theta\to -\alpha}\Psi_n(\alpha,\theta) = -\infty$. Furthermore,
  $\log \Gamma(z) \sim \log(1/z)$ as $z\to 0$.  Hence, there exists a constant $c > 0$
  such that for small enough $0 < b_2 \ll \hat{\alpha}_n^0$
  \begin{equation*}
    \sup_{\substack{(\alpha,\theta)\in \Phi_1\\\theta < -\hat{\alpha}_n^0 + b_2}}\{\ell_n(\alpha,\theta) - \sup \ell_n\}%
    \leq -c \log \Big(\frac{\hat{\alpha}_n^0}{b_2}\Big)
  \end{equation*}
  with probability $1+o(1)$.
\end{proof}

To facilitate the analysis of $L_n$ over $\Phi_1 \cap \Phi_2$ we introduce an
auxiliary estimator $\hat{T}_n$ of $\theta_{*}$ as in the next Proposition.

\begin{prp}
  \label{pro:auxiliary-est-theta}
  Define a function $\Psi_n : \Phi \to \Reals$ by
  $\Psi_n(\alpha,\theta) = \log \Gamma(\theta + 1) - \log \Gamma(\theta/\alpha + 1) + \theta \log \frac{K_n^{1/\alpha}}{n}$.
  Then for each fixed $\alpha \in (0,1)$ the map
  $\theta \mapsto \Psi_n(\alpha,\theta)$ is concave and has a unique maximizer
  on $(-\alpha,\infty)$. Call this maximizer $\hat{T}_{n,\alpha}$ and let
  $\hat{T}_n = \hat{T}_{n,\hat{\alpha}_n^0}$. If the
  Assumption~\ref{ass:ucj-yl04-iyy} is satisfied, then the sequence
  $(\hat{T}_n)_{n\geq 1}$ converges in probability to $\theta_{*} \in (-\alpha_*,\infty)$.
\end{prp}
\begin{proof}
  Let $\alpha \in (0,1)$ arbitrary. Recall that $\psi$ is the derivative of
  $\log \Gamma$ by definition. Hence, by differentiating $\Psi_n$ with respect
  to $\theta$, we find that
  \begin{equation}
    \label{eq:jyg-7i8q-8zd}
    \partial_{\theta}\Psi_n(\alpha,\theta)%
    = \psi(\theta + 1) - \alpha^{-1}\psi(\theta/\alpha + 1) + \log \frac{K_n^{1/\alpha}}{n}.
  \end{equation}
  Standard analysis of the digamma function shows that
  $\lim_{z\to 0}\psi(z) = -\infty$ and $\psi(z) = \log(z) + O(z^{-1})$ as
  $z\to \infty$. Thus
  $\lim_{\theta\to -\alpha}\partial_{\theta}\Psi_n(\alpha,\theta) = + \infty$
  and $\lim_{\theta \to \infty}\partial_{\theta}\Psi_n(\alpha,\theta) = \infty$
  \citep{abramo}. Since $\theta \mapsto \partial_{\theta}\Psi_n(\alpha,\theta)$
  is continuous, there exists a solution in $\theta$ to
  $\partial_{\theta}\Psi_n(\alpha,\theta) = 0$. Furthermore, it is well-known
  that \citep{abramo}

  \begin{equation*}
    \psi(z + 1) = -\gamma + \sum_{k=1}^{\infty}\Big(\frac{1}{k} - \frac{1}{k+z}\Big),
  \end{equation*}
  where $\gamma$ is Euler's gamma constant, $\gamma \approx 0.577$. It follows
  that
  \begin{align}
    \label{eq:xt4-uy9j-ren}
    \partial_{\theta}^2\Psi_n(\alpha,\theta)%
    &= \sum_{k=1}^{\infty}\Big[ \frac{1}{(k+\theta)^2} - \frac{1}{(\alpha k + \theta)^2} \Big]%
      < 0.
  \end{align}
  This establishes that $\theta \mapsto \Psi_n(\alpha,\theta)$ has a unique
  maximizer on $(-\alpha,\infty)$ as well as the (strict) concavity of this map.
  By computations analoguous to the one made in the proof of
  Proposition~\ref{pro:6}, we find that under Assumption~\ref{ass:ucj-yl04-iyy}
  we have
  \begin{align*}
    \EE_p[K_n]
    &= n \int_0^1\bar{F}_p(x)(1-x)^{n-1}dx\\
    &= L\Gamma(1-\alpha_{*})n^{\alpha_{*}}\Big[1 + o\Big(\frac{1}{\log(n)}\Big)\Big],
  \end{align*}
  and by Lemma~A.3 in \cite{FN(21)} we have $\var_p(K_n) \leq \EE_p[K_n]$. Since
  $\log(n)|\hat{\alpha}_n - \alpha_0| = o_p(1)$ by the
  Item~\ref{item:rph-e4g6-jcz}, we find that
  $K_n/n^{\hat{\alpha}_n} = L\Gamma(1-\alpha_{*}) + o_p(1)$. Since
  $\partial_{\theta}\Psi_n(\hat{\alpha}_n^0,\hat{T}_{n}) = 0$, by
  equation~\eqref{eq:jyg-7i8q-8zd},
  \begin{align*}
    L \Gamma(1 - \hat{\alpha}_n^0)%
    &= \exp\{\psi(\hat{T}_{n}/\hat{\alpha}_n^0+1) - \hat{\alpha}_n^0\psi(\hat{T}_n + 1) \} + o_p(1).
  \end{align*}
  The final claim follows by a continuity argument.
\end{proof}

\begin{prp}
  \label{pro:5i7-gs99-awg}
  Any maximizer $(\hat{\alpha}_n,\hat{\theta}_n)$ of $\ell_n$ must satisfy
  $\hat{\alpha}_n = \hat{\alpha}_n^0 + O_p\big[\frac{\log(n)}{n^{\alpha_{*}}} \big]$ and
  $\hat{\theta}_n = \hat{T}_n + O_p\big[\frac{\log(n)^2}{n^{\alpha_{*}}} \big]$.
\end{prp}
\begin{proof}
  By Propositions~\ref{pro:kl0-860e-jpy} and \ref{pro:x8d-d96y-u1g}, it is
  enough to consider $(\alpha,\theta) \in \Phi_1 \cap \Phi_2$. It is easily seen that the bound
  of equation~\eqref{eq:12n-q8jt-x2q} also holds true here. Then, by a
  Taylor expansion, there is $\bar{\alpha}$ in a neighborhood of $\hat{\alpha}_n^0$ such
  that%
  \begin{align*}
    \Psi_n(\alpha,\theta) - \Psi_n(\hat{\alpha}_n^0,\hat{T}_{n})%
    &=\partial_{\alpha}\Psi_n(\bar{\alpha},\theta)(\alpha - \hat{\alpha}_n^0)%
      + \Psi_n(\hat{\alpha}_n^0,\theta) - \Psi_n(\hat{\alpha}_n^0,\hat{T}_n).
  \end{align*}
  But,%
  \begin{align*}
    \partial_{\alpha}\Psi_n(\bar{\alpha},\theta)%
    &= \frac{\theta}{\bar{\alpha}^2}\{-\log K_n + \psi(\theta/\bar{\alpha} + 1)\}%
    = O\big[ \theta\max(\log K_n,\log \theta  ) \big],
  \end{align*}
  where the second estimate follows because $\theta/\bar{\alpha}$ is bounded away from
  $-1$ with probability $1 + o(1)$ by the previous results, and hence
  $\psi(\theta/\bar{\alpha} + 1)$ is of order constant for small $-\hat{\alpha}_n^0 + b_2 \leq \theta \lesssim 1$ and
  never more than $O(\log(\theta))$ for $\theta\gtrsim 1$. So on $\Phi_1 \cap \Phi_2$, we
  must have
  \begin{align*}
    \Psi_n(\alpha,\theta) - \Psi_n(\hat{\alpha}_n^0,\hat{T}_n)%
    = \Psi_n(\hat{\alpha}_n^0,\theta) - \Psi_n(\hat{\alpha}_n^0,\hat{T}_{n}) + O_p\big[ \theta \log(n)\cdot(\alpha - \hat{\alpha}_n^0)  \big].
  \end{align*}
  By concavity of $\theta \mapsto \Psi_n(\hat{\alpha}_n^0,\theta)$
  \begin{align*}
    \sup_{|\theta - \hat{T}_{n}| > \varepsilon}\{\Psi_n(\hat{\alpha}_n^0,\theta) - \Psi_n(\hat{\alpha}_n^0,\hat{T}_{n}) \}%
    &\leq\sup_{x \in (0,\varepsilon)}\partial_{\theta}^2\Psi_n(\hat{T}_n + x,\hat{\alpha}_n^0)\cdot \frac{\varepsilon^2}{2}.
  \end{align*}
  But by the equation~\eqref{eq:xt4-uy9j-ren},
  \begin{align*}
    \partial_{\theta}^2\Psi_n(\alpha,\theta)%
    &\leq \frac{1}{(1+\theta)^2} - \frac{1}{(\alpha + \theta)^2}\\%
    &= \frac{-(1-\alpha)^2 - 2\theta(1-\alpha)}{(\alpha + \theta)(1 + \theta)}\\
    &\leq \frac{-(1-\alpha)^2 + 2\alpha(1-\alpha)}{(\alpha + \theta)(1 + \theta)}\\
    &= - \frac{(1-\alpha)^2}{(\alpha  + \theta)(1 + \theta)}.
  \end{align*}
  Hence, for $\varepsilon > 0$ small enough
  $\ell_n(\alpha,\theta) - \sup \ell_n \lesssim -\varepsilon^2 + O_p(n^{-\delta}\log(n)^2)$ whenever
  $|\theta - \hat{T}_n| \geq \varepsilon$ with $\theta \leq b_2'\log(n)$. So we have shown
  that any maximizer of $\ell_n$ must lie in the region
  $|\alpha - \hat{\alpha}_n^0| \leq b_0n^{-\delta}$ and
  $|\theta - \hat{T}_n| \leq b_3 n^{-\delta/2}\log(n)$ for some $b_3 > 0$ large
  enough. In this region, and by the same arguments as before, we can refine the
  bound of equation~\eqref{eq:12n-q8jt-x2q} onto
  \begin{align*}
    \ell_n(\alpha,\theta) - \sup \ell_n%
    &\leq \frac{1}{2}\partial_{\theta}^2\Psi_n(\bar{T},\hat{\alpha}_n^0)(\theta - \hat{T}_n)^2%
      + \frac{1}{2} \partial_{\alpha}^2\ell_n(\bar{\alpha},0)(\alpha - \hat{\alpha}_n^0)^2\\
    &\quad%
      + O_p\big[\log(n)\cdot (\alpha - \hat{\alpha}_n^0) + n^{-\alpha_0} \big]%
  \end{align*}
  for some $\bar{\alpha}$ in a neighborhood of $\hat{\alpha}_n^0$ and some $\bar{T}$ in a
  neighborhood of $\hat{T}_{n}$. The conclusion follows because with
  probability $1+o(1)$ it is true that $\partial_{\theta}^2\Psi_n(\bar{T},\hat{\alpha}_n^0) \lesssim -1$ by
  the previous, and $\frac{1}{2} \partial_{\alpha}^2\ell_n(\bar{\alpha},0) \lesssim -n^{\alpha_{*}}$ by
  equation~\eqref{eq:ft9-0iz5-m0m} and classical arguments.
\end{proof}

\subsection{Auxiliary results for the proof of the item~\ref{item:6x9-2jie-dzi}
  in Proposition~\ref{prp:9cr-dlcr-2yc}}
\label{sec:1ij-jg4e-sts}

This section gathers a series of Lemma that are used in the proof of the
Propositions~\ref{pro:kl0-860e-jpy}, \ref{pro:x8d-d96y-u1g}, and
\ref{pro:5i7-gs99-awg} of the previous section.

\begin{lem}
  \label{lem:ny9-3mjf-qpx}
  The following lower bound is true.
  \begin{equation*}
    \sup \ell_n%
    \geq \Psi_n(\hat{\alpha}_n^0,\hat{T}_n) - \log \hat{\alpha}_n^0 +\ell_n(\hat{\alpha}_n^0) + \log \Gamma(K_n) - \log \Gamma(n) + O_p(K_n^{-1}).
  \end{equation*}
  Furthermore, observe that $\Psi_n(\hat{\alpha}_n^0,\hat{T}_n) = O_p(1)$
  because of the item~\ref{item:6x9-2jie-dzi} of the Theorem.
\end{lem}
\begin{proof}
  Obviously, $\sup \ell_n \geq \ell_n(\hat{\alpha}_n^0,\hat{T}_n)$. Then, the
  conclusion follows by equation~\eqref{eq:j29-mj0s-4ya}, the definition of
  $\Psi_n$, by Stirling's formula applied to
  $\Gamma(\hat{T}_{n}/\hat{\alpha}_n^0 + K_n)$, to
  $\Gamma(\hat{T}_{n} + n)$, to $\Gamma(K_n)$, and to $\Gamma(n)$, and because
  $K_n = O_p(n^{\alpha_{*}})$, and $\hat{T}_{n} > -\hat{\alpha}_n^0$.
\end{proof}

\begin{lem}
  \label{lem:xbo-ffe5-o0m}
  For all $\theta > 0$ and for all $\alpha \in (0,1)$
  \begin{align*}
    L_n(\alpha,\theta)%
    &\leq \frac{\Gamma(K_n) e^{K_n}}{\alpha \Gamma(n)} L_n(\alpha,0).
  \end{align*}
\end{lem}
\begin{proof}
  The starting point is the following observation. If
  $Z \sim \gammaDist(\theta/\alpha,1)$ and $W \sim \gammaDist(K_n)$ are independent random
  variables, then
  \begin{align*}
    \EE\Big[e^Z\Big(1 - \frac{Z^{1/\alpha}}{W^{1/\alpha}}\Big)^{n-1}\1_{W > Z}\Big]%
    &= \frac{1}{\Gamma(\theta/\alpha)\Gamma(K_n)}\int_0^{\infty} w^{K_n-1}e^{-w}\int_0^w z^{\theta/\alpha-1}\Big(1 - \frac{z^{1/\alpha}}{w^{1/\alpha}}\Big)^{n-1}\intd z \intd w\\
    &= \frac{\alpha}{\Gamma(\theta/\alpha)\Gamma(K_n)}\int_0^{\infty}w^{K_n + \theta/\alpha - 1}e^{-w}\intd w \int_0^1u^{\theta-1}(1-u)^{n-1}\intd u\\
    &= \frac{\alpha \Gamma(\theta/\alpha + K_n)}{\Gamma(\theta/\alpha)\Gamma(K_n)} \frac{\Gamma(n) \Gamma(\theta)}{\Gamma(\theta+n)}.
  \end{align*}
  Therefore, by Chernoff's bound on the Gamma distribution
  \begin{align*}
    \frac{\Gamma(\theta)}{\Gamma(\theta/\alpha)} \frac{\Gamma(\theta/\alpha + K_n)}{\Gamma(\theta+n)}%
    &= \frac{\Gamma(K_n)}{\alpha \Gamma(n)}\EE\Big[e^Z\Big(1 - \frac{Z^{1/\alpha}}{W^{1/\alpha}}\Big)^{n-1}\1_{W > Z}\Big]\\
    &\leq \frac{\Gamma(K_n)}{\alpha \Gamma(n)}\EE[e^Z\PP(W > Z)]\\
    &\leq \frac{\Gamma(K_n)}{\alpha\Gamma(n)}\EE\Big[e^Z\min\Big\{1, e^{K_n-Z}\Big(\frac{K_n}{Z}\Big)^{K_n} \Big\} \Big]\\
    &= \frac{\Gamma(K_n)}{\alpha \Gamma(n)}\EE\Big[\min\Big\{e^Z,\, e^{K_n}\Big(\frac{K_n}{Z}\Big)^{K_n}
      \Big\} \Big]\\
    &\leq \frac{\Gamma(K_n) e^{K_n}}{\alpha \Gamma(n)}.
  \end{align*}
  Then, since $\theta > 0$,
  \begin{align*}
    L_n(\alpha,\theta)%
    &=%
      \frac{\Gamma(\theta)}{\Gamma(\theta/\alpha)} \frac{\Gamma(\theta/\alpha + K_n)}{\Gamma(\theta+n)} L_n(\alpha,0)
    \leq \frac{\Gamma(K_n) e^{K_n}}{\alpha \Gamma(n)} L_n(\alpha,0).
  \end{align*}
\end{proof}

\begin{lem}
  \label{lem:205-743d-w1y}
  Let $K_n \geq 2$. Then, there is a generic constant $C> 0$ such that for all
  $\alpha \in (0,1)$ and all $-\alpha < \theta \leq 0$
  \begin{align*}
    L_n(\alpha,\theta)%
    &\leq \frac{Cn \Gamma(K_n)}{\alpha^2\Gamma(n)} L_n(\alpha,0).
  \end{align*}
\end{lem}
\begin{proof}
  We note that there exist constants $A,B> 0$ such that for all $0 < z \leq 1$ it
  holds $A \leq z \Gamma(z) \leq B$. Hence, we can find a generic constant $C > 0$ such
  that
  \begin{align*}
    \frac{\Gamma(\theta+1)}{\Gamma(\theta/\alpha + 1)}%
    &\leq C \frac{\theta/\alpha + 1}{\theta+1}%
      = \frac{C}{\alpha} \cdot \frac{\theta + \alpha}{\theta+1}%
      \leq \frac{C}{\alpha}.
  \end{align*}
  Then, by Gautschi's inequality, we can bound,
  \begin{align*}
    \frac{\Gamma(\theta/\alpha + K_n)}{\Gamma(\theta+n)}%
    &= \frac{\Gamma(K_n)}{\Gamma(n)} \cdot \frac{\Gamma(\theta/\alpha + K_n)}{\Gamma(K_n)} \cdot \frac{\Gamma(n)}{\Gamma(n+\theta)}\\
    &\leq \frac{\Gamma(K_n)}{\Gamma(n)} \cdot (K_n-1)^{\theta/\alpha} \cdot n^{-\theta}\\
    &\leq \frac{n \Gamma(K_n)}{\Gamma(n)}.
  \end{align*}
\end{proof}

\section{Proof of Theorem~\ref{thm:4sc-jzwh-mc1}}
\label{sec:f8u-2dds-0h7}

In the whole proof we let $Z_{n,\alpha} = \hat{V}_n^{1/2}(\alpha - \hat{\alpha}_n)$ for
simplicity. We define $A_n = \{x \in \Reals \,:\, x^2 \leq C \log(n) \}$ and
$B_n = \{x \in \Reals\,:\, r_n^{-1} \leq x \leq r_n \}$, with $C>0$ a constant and
$(r_n)$ a slowly increasing sequence to be chosen accordingly later. We also let
$C_n = A_n\times B_n$. Then, we decompose
\begin{align*}
  \sup_{A,B}\big|\Pi( Z_{n,\alpha} \in A,\, \gamma \in B \mid \mathbf{X}_n )%
  - \phi(A)H_{*}(B)
  \big|%
  &\leq R_1 + R_2 + R_3,
\end{align*}
with
\begin{align*}
  R_1%
  &= \sup_{A,B}\Big|%
    \frac{\Pi( Z_{n,\alpha} \in A \cap A_n,\, \gamma \in B \cap B_n \mid \mathbf{X}_n )}%
    {\Pi( Z_{n,\alpha} \in A_n,\, \gamma \in B_n \mid \mathbf{X}_n )}%
    - \frac{\phi(A \cap A_n)H_{*}(B \cap B_n)}{\phi(A_n)H_{*}(B_n)} \Big|,\\
  R_2%
  &= \sup_{A,B}\Big|%
    \frac{\Pi( Z_{n,\alpha} \in A \cap A_n,\, \gamma \in B \cap B_n \mid \mathbf{X}_n )}%
    {\Pi( Z_{n,\alpha} \in A_n,\, \gamma \in B_n \mid \mathbf{X}_n )}%
    - \Pi( Z_{n,\alpha} \in A,\, \gamma \in B \mid \mathbf{X}_n )
    \Big|,\\
  R_3%
  &= \sup_{A,B}\Big| \frac{\phi(A \cap A_n)H_{*}(B \cap B_n)}{\phi(A_n)H_{*}(B_n)}%
    - \phi(A)H_{*}(B)\Big|.
\end{align*}
We bound each of the terms $R_1$, $R_2$ and $R_3$ in the following subsections.

\subsubsection*{Bound on $R_1$.}

We let for simplicity
$\tilde{\Pi}_n(A,B) = \frac{\Pi( Z_{n,\alpha} \in A \cap A_n,\, \gamma \in B \cap B_n \mid \mathbf{X}_n )}{\Pi( Z_{n,\alpha} \in A_n,\, \gamma \in B_n \mid \mathbf{X}_n )}$
and we write $h_{*}$ the density of $H_{*}$, $\tilde{\pi}_n$ the density
of $\tilde{\Pi}_n$, and $f$ the density of $\phi$. Then,
\begin{align*}
  R_1%
  &= \sup_{A,B}\Big|\int_{A\cap A_n}\int_{B\cap B_n}\Big(1 - \frac{f(x)h_{*}(y)}{\phi(A_n)H_{*}(B_n)\tilde{\pi}_n(x,y)} \Big)\tilde{\pi}_n(x,y)\intd x\intd y\Big|\\
  &\leq \int_{C_n}\Big|1 - \frac{f(x)h_{*}(y)}{\phi(A_n)H_{*}(B_n)\tilde{\pi}_n(x,y)} \Big| \tilde{\pi}_n(x,y)\intd x \intd y,
\end{align*}
Introducing the shorthand notation $\hat{x}_n \equiv \hat{V}_n^{-1/2}x$ to
ease the reading, we have by Bayes' rule that
\begin{equation*}
  \tilde{\pi}_n(x,y)%
  = \frac{e^{\mathcal{L}_n(x,y)}g_{\alpha}(\hat{\alpha}_n + \hat{x}_n)g_{\gamma}(y)\1_{A_n}(x)\1_{B_n}(y)}{\int_{C_n} e^{\mathcal{L}_n(x',y')}g_{\alpha}(\hat{\alpha}_n + \hat{x}_n')g_{\gamma}(y')\intd x' \intd y' }
\end{equation*}
with
$\mathcal{L}_n(x,y) = \ell_n(\hat{\alpha}_n + \hat{x}_n, y - \hat{\alpha}_n - \hat{x}_n)$.
Here we note that on an event of probability $1 + o(1)$ we always have
$\hat{\alpha}_n + \hat{x}_n \in (0,1)$ so $\tilde{\pi}_n(x,y)$ is well-defined on this
event. We assume without loss of generality that $\tilde{\pi}_n$ is always
well-defined. Then,
\begin{align*}
  R_1%
  &\leq \int_{C_n}\Big|1 - \int_{C_n} \frac{f(x)h_{*}(y)e^{\mathcal{L}_n(x',y')} g_{\alpha}(\hat{\alpha}_n + \hat{x}_n')g_{\gamma}(y')\intd x' \intd y'}{\phi(A_n)H_{*}(B_n)e^{\mathcal{L}_n(x,y)} g_{\alpha}(\hat{\alpha}_n + \hat{x}_n)g_{\gamma}(y) } \Big| \tilde{\pi}_n(x,y)\intd x \intd y\\
  &=\int_{C_n}\Big|1 - \int_{C_n} \frac{f(x)h_{*}(y)e^{\mathcal{L}_n(x',y')} g_{\alpha}(\hat{\alpha}_n + \hat{x}_n')g_{\gamma}(y')}{f(x')h_{*}(y')e^{\mathcal{L}_n(x,y)} g_{\alpha}(\hat{\alpha}_n + \hat{x}_n)g_{\gamma}(y) } \frac{f(x')h_{*}(y')\intd x' \intd y'}{\phi(A_n)H_{*}(B_n)} \Big| \tilde{\pi}_n(x,y)\intd x \intd y\\
  &\leq \int_{C_n^2}\Big|1 - \frac{f(x)h_{*}(y)e^{\mathcal{L}_n(x',y')} g_{\alpha}(\hat{\alpha}_n + \hat{x}_n')g_{\gamma}(y')}{f(x')h_{*}(y')e^{\mathcal{L}_n(x,y)} g_{\alpha}(\hat{\alpha}_n + \hat{x}_n)g_{\gamma}(y) }\Big| \frac{f(x')h_{*}(y')\intd x' \intd y'}{\phi(A_n)H_{*}(B_n)}\tilde{\pi}_n(x,y)\intd x \intd y\\
  &\leq \sup_{\substack{x,x'\in A_n\\y,y'\in B_n}}\Big|1 - \frac{f(x)h_{*}(y)e^{\mathcal{L}_n(x',y')} g_{\alpha}(\hat{\alpha}_n + \hat{x}_n')g_{\gamma}(y')}{f(x')h_{*}(y')e^{\mathcal{L}_n(x,y)} g_{\alpha}(\hat{\alpha}_n + \hat{x}_n)g_{\gamma}(y) }\Big|.
\end{align*}
Let define
$\mathcal{H}(y) = \log \Gamma(1 - \alpha_{*} + y) - \log \Gamma(y/\alpha_{*}) + \frac{y}{\alpha_{*}}\log[L \Gamma(1-\alpha_{*})]$,
so that $h_{*}(y) \propto e^{\mathcal{H}(y)}g_{\gamma}(y)$. Since $f(x) \propto e^{-x^2/2}$, we
obtain from the last display that
\begin{equation*}
  R_1%
  \leq \sup_{\substack{x,x'\in A_n\\y,y'\in B_n}}\Big|1 - \frac{e^{-\frac{1}{2}x^2 -\mathcal{L}_n(x,y) + \mathcal{H}(y)} g_{\alpha}(\hat{\alpha}_n + \hat{x}_n')}{e^{- \frac{1}{2}(x')^2 -\mathcal{L}_n(x',y') + \mathcal{H}(y')} g_{\alpha}(\hat{\alpha}_n + \hat{x}_n) }\Big|
\end{equation*}
Since $\sup_{x\in A_n}\hat{x}_n = o_p(1)$, since
$\hat{\alpha}_n = \alpha_{*} + o_p(1)$, and since $g_{\alpha}$ has a continuous and positive density
in a neighborhood of $\alpha_{*}$, it is enough to establish that there exists a
sequence $(E_n)$ [not depending on $x$ or $y$] such that
\begin{equation}
  \label{eq:gqe-u2qn-08y}
  \sup_{\substack{x\in A_n\\y\in B_n}}\Big|\frac{x^2}{2} + \mathcal{L}_n(x,y) - \mathcal{H}(y) + E_n\Big| = o_p(1)
\end{equation}
to obtain that $R_1 = o_p(1)$. We dedicate the rest of this section to prove
that equation~\eqref{eq:gqe-u2qn-08y} holds true. First, we note that by
Stirling's formula [see also the computations above the
equation~\eqref{eq:12n-q8jt-x2q}]
\begin{align*}
  \sup_{\substack{x\in A_n\\y\in B_n}}\Big| \mathcal{L}_n(x,y) %
  - \Psi_n(\hat{\alpha}_n + \hat{x}_n,y - \hat{\alpha}_n - \hat{x}_n)
    - \log \frac{\Gamma(K_n)}{\Gamma(n)}%
  + \log(\hat{\alpha}_n) - \ell_n(\hat{\alpha}_n + \hat{x}_n,0)\Big|
  &= o_p(1).
\end{align*}
Also,
\begin{align*}
  \Psi_n(\hat{\alpha}_n + \hat{x}_n, y - \hat{\alpha}_n - \hat{x}_n) - \mathcal{H}(y)%
  &= \log \Gamma(y - \hat{\alpha}_n - \hat{x}_n + 1) - \log \Gamma(y - \alpha_{*} + 1)\\
  &\quad - \log \Gamma \Big(\frac{y}{\hat{\alpha}_n + \hat{x}_n} \Big)%
    + \log \Gamma\Big(\frac{y}{\alpha_{*}} \Big)\\
  &\quad + \frac{y - \hat{\alpha}_n - \hat{x}_n}{\hat{\alpha}_n + \hat{x}_n}\log \frac{K_n}{n^{\hat{\alpha}_n + \hat{x}_n}}%
    - \frac{y}{\alpha_{*}}\log[L\Gamma(1-\alpha_{*})].
\end{align*}
Since $\sup_{x\in A_n}\log(n)\hat{x}_{n} = o_p(1)$,
$\log(n)|\hat{\alpha}_n - \alpha_{*}| = o_p(1)$ by Proposition~\ref{prp:9cr-dlcr-2yc}, and
$K_n/n^{\alpha_{*}} = L\Gamma(1-\alpha_{*}) + o_p(1)$ [see for instance the proof of
Proposition~\ref{prp:9cr-dlcr-2yc}], if the sequence $(r_n)$ increases slowly
enough, it must be that
\begin{align*}
  \sup_{\substack{x \in A_n\\y\in B_n}}\Big|\frac{y}{\hat{\alpha}_n + \hat{x}_n}\log \frac{K_n}{n^{\hat{\alpha}_n + \hat{x}_n}} - \frac{y}{\alpha_{*}}\log[L\Gamma(1-\alpha_{*})] \Big| = o_p(1).
\end{align*}
Next, $\hat{\alpha}_n + \hat{x}_n$ are in a small neighborhood of $\alpha_{*}$ with
probability $1 + o(1)$ and $y - \alpha_{*}$ is bounded away from $1$ for all $y > 0$.
Hence, by a suitable Taylor expansion
\begin{multline*}
 \sup_{\substack{x \in A_n\\y\in B_n} }\Big|%
  \log \Gamma(y - \hat{\alpha}_n - \hat{x}_n + 1) - \log \Gamma(y - \alpha_{*} + 1)%
  \Big|\\%
  \leq \sup_{\substack{x \in A_n\\y\in B_n} }\sup_{|\delta| \leq |\hat{\alpha}_n + \hat{x}_n - \alpha_{*}}\Big|%
  \psi(y - \alpha_{*} +1 + \delta)(\hat{\alpha}_n + \hat{x}_n - \alpha_{*}) \Big| = o_p(1)
\end{multline*}
because $y - \alpha_{*} + 1 + \delta$ remains bounded away from zero and is never more
than $\lesssim r_n$, so the last display is indeed a $o_p(1)$ is the sequence $(r_n)$
grows slowly enough [recall $\psi(z) \sim \log(z)$ as $z \to \infty$ and $\psi(z) \sim -1/z$ as
$z \to 0$]. With a similar reasoning, there is $u$ between
$y/(\hat{\alpha}_n + \hat{x}_n)$ and $y/\alpha_{*}$ such that
\begin{align*}
  \log \Gamma\Big(\frac{y}{\hat{\alpha}_n + \hat{x}_n} \Big)%
  - \log \Gamma\Big(\frac{y}{\alpha_{*}}\Big)%
  &= \psi(u)\Big[\frac{y}{\hat{\alpha}_n + \hat{x}_n} - \frac{y}{\alpha_{*}}  \Big]\\
  &= y \psi(u)\frac{\alpha_{*} - (\hat{\alpha}_n - \hat{x}_n)}{\alpha_{*}(\hat{\alpha}_n + \hat{x}_n)}
\end{align*}
and thus provided $(r_n)$ grows slowly enough
\begin{equation*}
  \sup_{\substack{x\in A_n\\y\in B_n}}\Big|\log \Gamma\Big(\frac{y}{\hat{\alpha}_n + \hat{x}_n} \Big)%
  - \log \Gamma\Big(\frac{y}{\alpha_{*}}\Big)\Big|%
  = o_p(1).
\end{equation*}
Finally, $\hat{\alpha}_n^0$ is the maximizer of $\ell_n(\cdot, 0)$, so there is $\bar{\alpha}$ in a
neighborhood of $\hat{\alpha}_n^0$ [hence in a neighborhood of $\hat{\alpha}_n$ and
$\alpha_{*}$ too] such that
\begin{align*}
  \ell_n(\hat{\alpha}_n + \hat{x}_n,0)%
  &= \ell_n(\hat{\alpha}_n^0,0)%
    + \frac{1}{2}\partial_{\alpha}^2\ell_n(\hat{\alpha}_n^0,0)\big(\hat{\alpha}_n - \hat{\alpha}_n^0 + \hat{x}_n\big)^2%
    + \frac{1}{6}\partial_{\alpha}^3\ell_n(\bar{\alpha},0)\big(\hat{\alpha}_n - \hat{\alpha}_n^0 + \hat{x}_n\big)^3.
\end{align*}
Thus, by Proposition~\ref{prp:9cr-dlcr-2yc}, by definition of $\hat{V}_n$, and
because $\partial_{\alpha}^2\ell_n(\hat{\alpha}_n^0,0) = O_p(n^{\alpha_{*}})$ and
$\partial_{\alpha}^3\ell_n(\hat{\alpha}_n^0,0) = O_p(n^{\alpha_{*}})$ (this follows by direct
computations, see for instance in \cite{FN(21)}), it
follows
\begin{equation*}
  \sup_{\substack{x\in A_n\\y\in B_n}}\Big|\ell_n(\hat{\alpha}_n + \hat{x}_n,0) - \ell_n(\hat{\alpha}_n^0,0)%
  + \frac{x^2}{2}\Big| = o_p(1).
\end{equation*}
Combining all the estimates, we have shown that if $(r_n)$ grows slowly enough,
then equation~\eqref{eq:gqe-u2qn-08y} is proved with
\begin{equation*}
  E_n = \log(\hat{\alpha}_n) - \ell_n(\hat{\alpha}_n^0,0) - \log \frac{\Gamma(K_n)}{\Gamma(n)} + \log \frac{K_n}{n^{\hat{\alpha}_n}}.
\end{equation*}

\subsubsection*{Bound on $R_2$.}

To prove that $R_2 = o_p(1)$ it is enough to show that
$\Pi(Z_{n,a} \notin A_n \mid \mathbf{X}_n) = o_p(1)$ and $\Pi(\gamma \notin B_n \mid \mathbf{X}_n) = o_p(1)$. We have to be
careful enough to ensure that these results hold for any slowly increasing
sequence $(r_n)$. First, following the same
steps as in the proofs of Proposition~\ref{pro:x8d-d96y-u1g} and
\ref{pro:5i7-gs99-awg}, for every $c > 0$ we can choose $C> 0$ such that
\begin{equation*}
  \sup_{\substack{\hat{V}_n(\alpha - \hat{\alpha}_n)^2 > C\log(n)\\\gamma > 0}}\{\ell_n(\alpha,\gamma - \alpha) - \sup \ell_n\}%
  \leq - c\log(n).
\end{equation*}
We deduce that (with $\hat{\gamma}_n = \hat{\theta}_n + \hat{\alpha}_n$)
\begin{align*}
  \Pi\big( Z_{n,\alpha} \notin A_n \mid \mathbf{X}_n  )%
  &= \frac{\iint \1_{\{\hat{V}_n(\alpha - \hat{\alpha}_n)^2 > C \log(n)\} } e^{\ell_n(\alpha,\gamma -\alpha)  - \sup \ell_n}G_{\alpha}(\intd \alpha) G_{\gamma}(\intd \gamma)}{\iint e^{\ell_n(\alpha,\gamma -\alpha) - \sup \ell_n}G_{\alpha}(\intd \alpha) G_{\gamma}(\intd \gamma)}\\
  &\leq n^{-c} \Big( \iint \1_{\hat{V}_n(\alpha - \hat{\alpha}_n)^2 \leq 1,\, |\gamma - \hat{\gamma}_n| \leq 1\} } e^{\ell_n(\alpha,\gamma-\alpha) - \sup \ell_n}G_{\alpha}(\intd \alpha) G_{\gamma}(\intd \gamma)\Big)^{-1}\\
  &\lesssim \frac{n^{-c}}{G_{\alpha}(\hat{V}_n(\alpha -\hat{\alpha}_n)^2 \leq 1  )G_{\gamma}((\gamma - \hat{\gamma}_n)^2 \leq 1)},
\end{align*}
where the last line follows by a Taylor expansion of $\ell_n$ near its maximizer. Next, $G_{\alpha}$ has a positive density in a
neighborhood of $\alpha_{*}$ and $\hat{\alpha}_n = \alpha_{*} + o_p(1)$ by
Proposition~\ref{prp:9cr-dlcr-2yc}; and similarly $G_{\gamma}$ has a positive density in
a neighborhood of $\gamma_{*} = \theta_{*} + \alpha_{*}$ and $\hat{\gamma}_n = \gamma_{*} + o_p(1)$.
By standard arguments and by taking $c$ sufficiently large, we find that
\begin{equation*}
  \Pi( Z_{n,\alpha} \notin A_n \mid \mathbf{X}_n) = o_p(1).
\end{equation*}

We now prove that $\Pi(\gamma \notin B_n \mid \mathbf{X}_n) = o_p(1)$. In fact, in view of the previous,
it is enough to consider $\Pi(Z_{n,\alpha}\in A_n,\, \gamma \notin B_n \mid \mathbf{X}_n)$. We start with a
rough concentration result. We let
$D_n = \{\gamma > 0\,:\, n^{-a} \leq \gamma \leq a \log(n) \}$ for an arbitrarily large constant
$a > 0$. Then, in view of Proposition~\ref{pro:x8d-d96y-u1g}, for every
$c > 0$ we can choose $a$ such that
\begin{equation*}
  \sup_{\substack{Z_{n,\alpha}\in A_n\\ \gamma \notin D_n}}\{\ell_n(\alpha,\gamma - \alpha) - \sup \ell_n\} \leq -c \log(n).
\end{equation*}
Hence, with the same reasoning as before, we deduce that
$\Pi(Z_{n,\alpha} \in A_n, \gamma \notin D_n \mid \mathbf{X}_n) = o_p(1)$. It is now enough to consider
$\Pi(Z_{n,\alpha} \in A_n, \gamma \in B_n^c \cap D_n \mid \mathbf{X}_n) $. Consider the function
$f_{\gamma}(\alpha) = \ell_n(\alpha,\gamma - \alpha) + \log(\alpha) - \ell_n(\alpha,0)$. By a Taylor expansion, for all
$\alpha \in (0,1)$ there exists $\bar{\alpha}$ between $\alpha$ and $\hat{\alpha}_n^0$ such that for
all $\gamma > 0$
\begin{align*}
  f_{\gamma}(\alpha) - f_{\gamma}(\hat{\alpha}_n^0)%
  &= f_{\gamma}'(\bar{\alpha})(\alpha - \hat{\alpha}_n^0)\\
  &= \frac{f_{\gamma}'(\bar{\alpha}) Z_{n,\alpha}}{\hat{V}_n^{1/2}}%
    + f_{\gamma}(\bar{\alpha})(\hat{\alpha}_n - \hat{\alpha}_n^0).
\end{align*}
Observe that,
\begin{align*}
  f_{\gamma}'(\alpha)%
  &= -\psi(\gamma + 1 - \alpha) + \frac{\gamma \psi(\gamma/\alpha)}{\alpha^2}%
    - \frac{\gamma \psi(\gamma/\alpha + K_n-1)}{\alpha^2}%
    + \psi(\gamma - \alpha + n).
\end{align*}
By classical properties of the digamma function $\gamma \psi(\gamma/\alpha) \to \alpha$ as $\gamma \to 0$ and
$\psi(z) \sim \log(z)$ as $z\to \infty$ \citep{abramo}. Furthermore
$\hat{\alpha}_n - \hat{\alpha}_n^0 = O_p[n^{-\alpha_{*}}\log(n)]$ by
Proposition~\ref{prp:9cr-dlcr-2yc}, so we deduce that
\begin{align*}
  \sup_{\substack{Z_{n,\alpha}\in A_n\\\gamma \in D_n}}|f_{\gamma}(\alpha) - f_{\gamma}(\hat{\alpha}_n^0)| = o_p(1).
\end{align*}
Since $\log(\alpha/\hat{\alpha}_n^0) = o_p(1)$ uniformly over $Z_{n,\alpha}\in A_n$, we have
shown that
\begin{align*}
  \sup_{\substack{Z_{n,\alpha}\in A_n\\\gamma \in D_n}}|\ell_n(\alpha,\gamma-\alpha)  - \ell_n(\hat{\alpha}_n^0,\gamma - \hat{\alpha}_n^0)
  - \{\ell_n(\alpha,0) - \ell_n(\hat{\alpha}_n^0,0)\} | = o_p(1).
\end{align*}
It follows that
\begin{align*}
  &\Pi(Z_{n,\alpha} \in A_n, \gamma \in B_n^c \cap D_n \mid \mathbf{X}_n)\\%
  &\qquad= \frac{\iint \1_{A_n}(Z_{n,\alpha})\1_{B_n^c\cap D_n}(\gamma)e^{\ell_n(\alpha,\gamma- \alpha)}G_{\alpha}(\intd \alpha)G_{\gamma}(\intd \gamma) }{\iint e^{\ell_n(\alpha,\gamma- \alpha)}G_{\alpha}(\intd \alpha)G_{\gamma}(\intd \gamma)}\\
  &\qquad\leq \frac{\iint \1_{A_n}(Z_{n,\alpha})\1_{B_n^c\cap D_n}(\gamma)e^{\ell_n(\alpha,\gamma- \alpha)}G_{\alpha}(\intd \alpha)G_{\gamma}(\intd \gamma) }{\iint \1_{A_n}(Z_{n,\alpha})\1_{D_n}(\gamma)e^{\ell_n(\alpha,\gamma- \alpha)}G_{\alpha}(\intd \alpha)G_{\gamma}(\intd \gamma)}\\
  &\qquad= \frac{(1+o_p(1))%
    \int_{\{Z_{n,\alpha}\in A_n\} }e^{\ell_n(\alpha,0) - \ell_n(\hat{\alpha}_n^0,0)}G_{\alpha}(\intd \alpha)\int_{B_n^c \cap D_n}
    e^{\ell_n(\hat{\alpha}_n^0,\gamma- \hat{\alpha}_n^0)}G_{\gamma}(\intd \gamma) }{\int_{\{Z_{n,\alpha}\in A_n\} }e^{\ell_n(\alpha,0) - \ell_n(\hat{\alpha}_n^0,0)}G_{\alpha}(\intd \alpha)\int_{D_n}
    e^{\ell_n(\hat{\alpha}_n^0,\gamma- \hat{\alpha}_n^0)}G_{\gamma}(\intd \gamma)}\\
  &\qquad= (1 + o_p(1))\frac{\int_{B_n^c \cap D_n}
    e^{\ell_n(\hat{\alpha}_n^0,\gamma- \hat{\alpha}_n^0)}G_{\gamma}(\intd \gamma) }{\int_{D_n}
    e^{\ell_n(\hat{\alpha}_n^0,\gamma- \hat{\alpha}_n^0)}G_{\gamma}(\intd \gamma)}.
\end{align*}
Thus,
\begin{align*}
  \Pi(Z_{n,\alpha} \in A_n, \gamma \in B_n^c \cap D_n \mid \mathbf{X}_n)%
  &\leq (1+o_p(1)) \frac{\int_{B_n^c \cap D_n}
    e^{\ell_n(\hat{\alpha}_n^0,\gamma- \hat{\alpha}_n^0) - \sup \ell_n}G_{\gamma}(\intd \gamma) }{\int_{|\gamma - \hat{\gamma}_n| \leq 1}
    e^{\ell_n(\hat{\alpha}_n^0,\gamma- \hat{\alpha}_n^0)-\sup \ell_n}G_{\gamma}(\intd \gamma)}
\end{align*}
By a Taylor expansion of $\ell_n$ near its maximizer, it can be seen that the denominator in the
previous expression is always greater than a constant with probability
$1 + o(1)$, while the numerator goes to zero because by the same arguments that
we used for instance in the proof of Proposition~\ref{pro:5i7-gs99-awg}. In
summary, we have shown that for every sequence $(r_n)$ going to infinity
\begin{align*}
  \Pi(\gamma \notin B_n \mid \mathbf{X}_n) = o_p(1).
\end{align*}

\subsubsection*{Bound on $R_3$.}

It is enough to show that $\phi(A_n^c) = o_p(1)$ and $H_{*}(B_n^c) = o_p(1)$. Both
facts are immediate.

\section{An Auxiliary result about the Pitman-Yor process}
\label{sec:lil-pyp}

Here we establish a law of the iterated logarithm for the Pitman-Yor process. We use this result to deduce Theorems~\ref{thm:wellspecified:mmle} and~\ref{thm:consistency-wellspecified-bayes} from Theorems~\ref{thm:misspecified:mle} and~\ref{thm:4sc-jzwh-mc1}, but we believe the result is of independent interest.

\begin{prp}
  \label{prp:pypnice}
  Suppose
  $P \sim \mathrm{PYP}(\alpha,\theta)$ with $\alpha \in (0,1)$ and
  $\theta > -\alpha$. Let $\bar{\rho}(t) = \frac{1}{\Gamma(1-\alpha)t^{\alpha}}$ and $f_{\alpha}$ denote the density of the $\alpha$-stable distribution on $\mathbb{R}_+$. Then there exists a random variable $T$ having density proportional to $t^{-\theta}f_{\alpha}(t)$ and such that
  \begin{equation*}
    \limsup_{x\to 0} \frac{\big|\bar{F}_P(x) - \bar{\rho}(T)x^{-\alpha} \big|}{\sqrt{2\bar{\rho}(Tx)\log(\log(\bar{\rho}(Tx))) }}%
    = 1\qquad \mathrm{a.s}.
  \end{equation*}
\end{prp}
\begin{proof}
  We recall for the next that $P \sim \mathrm{PYP}(\alpha,\theta)$
  is equal in distribution to $\sum_{j\geq 1}V_j\delta_{S_j}$ where
  $V_1,V_2,\dots$ has a $\mathrm{PD}(\alpha,\theta)$ 
  distribution (Poisson-Dirichlet) and $S_1,S_2,\dots$ is a sequence of iid symbols.
  
  We first establish the result for $P\sim \mathrm{PYP}(\alpha,0)$. Let $\Gamma_1 \leq \Gamma_2 \leq \dots$ be the arrival times of a unit rate Poisson process on $\mathbb{R}_+$ and let $T = \sum_{k\geq 1}\bar{\rho}^{-1}(\Gamma_k)$. Also let $\QQ_{\alpha}$ be the joint law of $(T,\Gamma_1,\Gamma_2,\dots)$. Classical results \citep{Fer(72),Pit(97)} give that $T^{-1}\bar{\rho}^{-1}(\Gamma_1),T^{-1}\bar{\rho}^{-1}(\Gamma_2),\dots$ has a $\mathrm{PD}(\alpha,0)$ distribution under $\QQ_{\alpha}$. Thus $\bar{F}_P(x) \overset{d}{=} \sum_{k\geq 1}I\big(\bar{\rho}^{-1}(\Gamma_k) > Tx \big)$, and
  \begin{align*}
    \limsup_{x\to 0} \frac{\big|\bar{F}_P(x) - \bar{\rho}(T)x^{-\alpha} \big|}{\sqrt{2\bar{\rho}(Tx)\log(\log(\bar{\rho}(Tx))) }}%
    &\overset{d}{=}\limsup_{x\to 0}\frac{\big|\sum_{k\geq 1}I\big(\Gamma_k < \bar{\rho}(Tx) \big) - \bar{\rho}(Tx) \big|}{\sqrt{2\bar{\rho}(Tx)\log(\log(\bar{\rho}(Tx))) }}\\
    &= \limsup_{y\to \infty} \frac{\big|\sum_{k\geq 1}I\big(\Gamma_k < y \big) - y \big|}{\sqrt{2y\log(\log(y)) }}
  \end{align*}
  Then, see that $\sum_{k\geq 1}I(\Gamma_k < y) \overset{as}{=} 1 + \nu(y)$ with $\nu(y) = \min\{n \;:\; \Gamma_n > y\}$. Since $\Gamma_n \overset{d}{=} \xi_1 + \dots + \xi_n$ with $\xi_1,\xi_2,\dots \overset{iid}{\sim} \mathrm{Exp}(1)$, the conclusion follows by \cite[Theorem~11.1 in Section~3.11]{Gut(10)}.

  Now we establish the result for $\theta > -\alpha$ but not necessarily $\theta = 0$. In this case $\bar{F}_P$ can be constructed as above from $(T,\Gamma_1,\Gamma_2,\dots)$ \citep{Pit(97)}. This time, however, $(T,\Gamma_1,\Gamma_2,\dots)$ has the law $\QQ_{\alpha,\theta}$ such that $\QQ_{\alpha,\theta}\circ T^{-1}$ has density $\propto t^{-\theta}$ with respect to $\QQ_{\alpha} \circ T^{-1}$ and $\QQ_{\alpha,\theta}((\Gamma_1,\Gamma_2,\dots)\in \cdot \mid T) = \QQ_{\alpha}((\Gamma_1,\Gamma_2,\dots)\in \cdot \mid T)$ $\QQ_{\alpha}$-as. Thus, the results above guarantee that
  \begin{equation*}
    \QQ_{\alpha,\theta}\Bigg(\limsup_{x\to 0} \frac{\big|\bar{F}_P(x) - \bar{\rho}(T)x^{-\alpha} \big|}{\sqrt{2\bar{\rho}(Tx)\log(\log(\bar{\rho}(Tx))) }} = 1 \ \Big\vert \ T\Bigg) = 1\qquad \PP_{\alpha,\theta}\textrm{-as}.\qedhere
  \end{equation*}
\end{proof}

\begin{cor}
  \label{cor:pypsmooth}
  If $P \sim \mathrm{PYP}(\alpha,\theta)$ with $\alpha \in (0,1)$ and $\theta > -\alpha$, then there exists a random variable $0 < L < \infty$ such that
  \begin{equation*}
    \lim_{x\to 0} x^{\alpha}\log\Big(\frac{1}{x} \Big)\big|\bar{F}_P(x) - L x^{-\alpha} \big| = 0\quad \mathrm{a.s.}
  \end{equation*}
\end{cor}


\section{Additional details on numerical illustrations}\label{sec:sup:illustration}

In Section~\ref{sec7} of the main manuscript, we provide numerical illustrations to display the performance of the proposed estimators on synthetic and real data. In the following, we provide additional details on the data generating process and further results.

\subsection{Parameter estimation}\label{sec:sup:param_est}

To analyse the recovery of the true parameters $\theta$ and $\alpha$ when the data is generated from a PYP, we compare an Empirical Bayes (EB) approach (which is achieved by finding $(\hat{\theta},\hat{\alpha})$ that maximize the EPPF \eqref{eppf_pit}), and a Full Bayes (FB) approach using several priors distributions on $(\theta,\alpha)$: (a) a non-informative prior for both parameters, with $\theta \sim \textrm{Gamma}(0.1,0.1)$ and $\alpha \sim \textrm{Beta}(1,1)$; (b) an informative prior for both parameters, with $\theta \sim \textrm{Gamma}(10,1)$ and $\alpha \sim \textrm{Beta}(4,1)$; (c) an informative prior for only $\theta$ and non-informative for $\alpha$, with $\theta \sim \textrm{Gamma}(10,1)$ and $\alpha \sim \textrm{Beta}(1,1)$; (d) an informative prior for only $\alpha$ and non-informative for $\theta$, with $\theta \sim \textrm{Gamma}(0.1,0.1)$ and $\alpha \sim \textrm{Beta}(4,1)$. 

The synthetic data is generated with different randomly generated values of $\alpha$ and $\theta$, the former sampled from a $\textrm{Gamma}(1,1)$ distribution and the latter from a $\textrm{Unif}(0,1)$ distribution. Moreover the inference is performed for several increasing values of the sample size ($10^3$, $10^4$, $10^5$ and $10^6$). We consider the absolute percentage or relative error, $l(\gamma_0, \hat{\gamma}) =  \frac{\vert\gamma_0 - \hat{\gamma}\vert}{\gamma_0}$, where $\gamma$ represents either $\alpha$ or $\theta$ and display in Figure~\ref{fig:par_estimation} the mean error across the generated datasets.

\subsection{Inference of SSP with synthetic data}\label{sec:sup:synthetic}

We study the performance of the proposed BNP estimators for the described species sampling problems for synthetic data. We consider a correctly-specified framework, where the data is generated from the PYP model, and a mis-specified framework, where the data generated from a power-law Zipf distribution, where $p(z) = z^{-s}/\zeta(s)$, $\zeta(\cdot)$ is the Riemann zeta function, and the parameter $s>1$ influences the power law behavior and corresponds to $1/\alpha$.

\begin{figure}[t]
\begin{flushleft}
\includegraphics[width=0.95\textwidth]{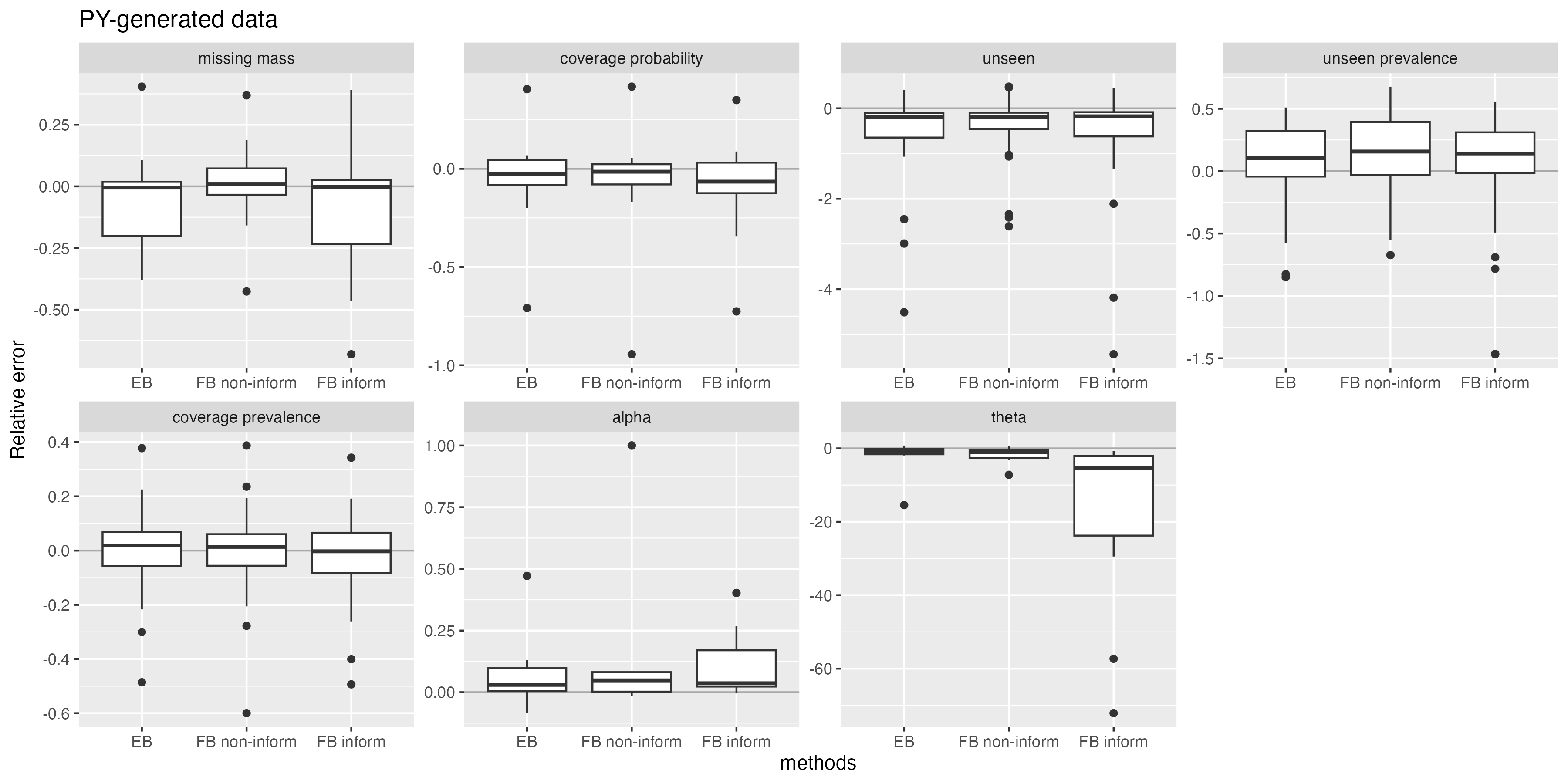}\\
\includegraphics[width=0.95\textwidth]{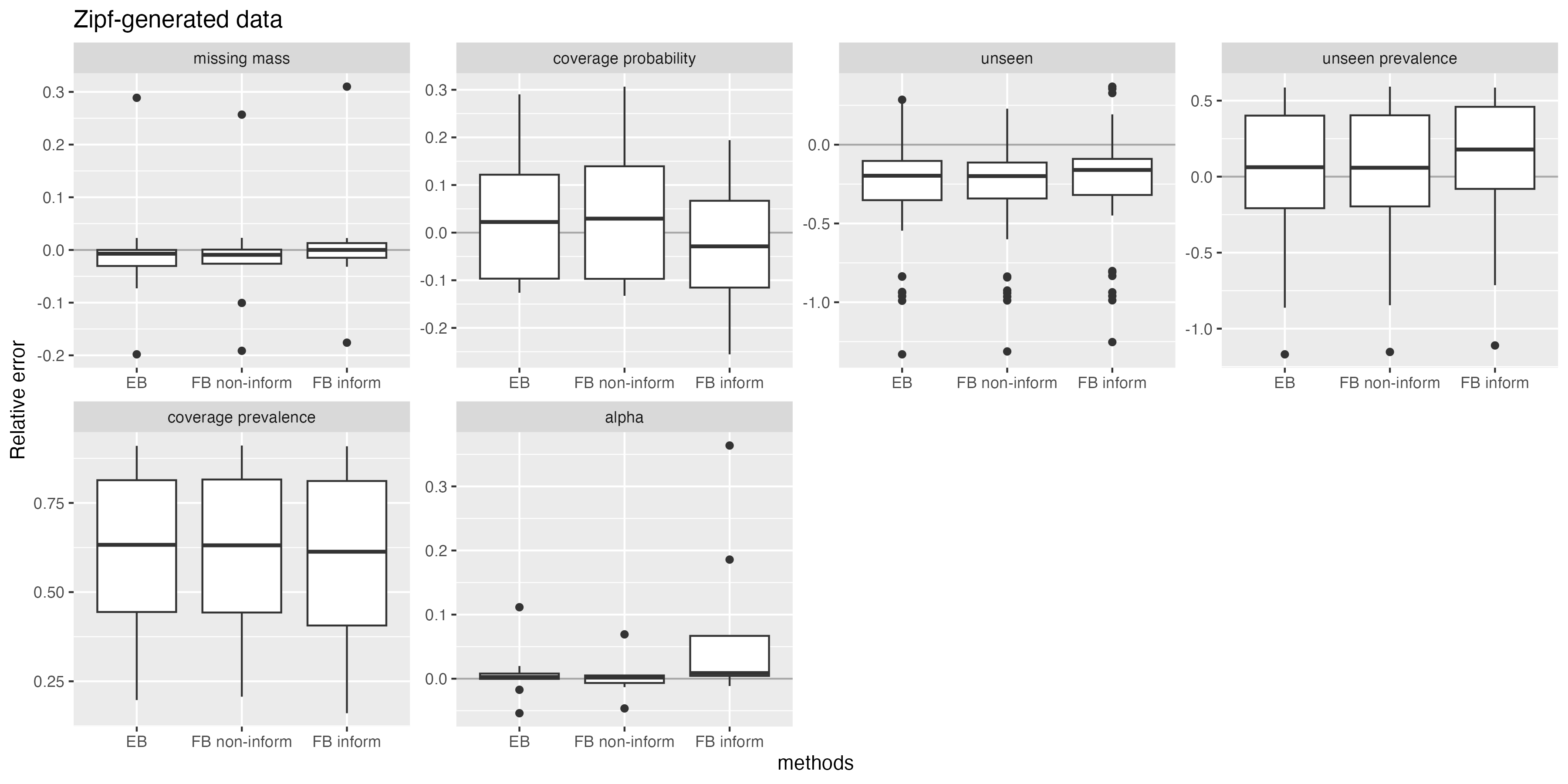}
\end{flushleft}
\caption{Relative (percentage) error for estimation of parameters and several SSP under the Empirical Bayes (EB) approach and the Full Bayes (FB) approach with different prior distributions, when the data is generated from the PYP (top panels) and the Zipf distribution (bottom panels). \label{fig:PYstatic}}
\end{figure}

We focus our analysis on the following functionals: the missing mass, the coverage probability $\mathfrak{p}_{r,n}$ for $r=1$, the number of unseen species, the unseen prevalence $\mathfrak{u}_{r,n,m}$ for $r=1$ and the coverage of prevalence $\mathfrak{f}_{r,n,m}$ for $r=1$. 
In the top panels of Figure~\ref{fig:PYstatic} we display the relative error across several synthetic datasets generated from the PIP. Across most of the datasets, the estimation of all functionals is good for all of the methods, with the worst recovery for the number of unseen, where the percentage error for several of the datasets is greater than $100\%$ in absolute value. 

Similarly, the bottom panels of Figure~\ref{fig:PYstatic} represents the relative error across several synthetic datasets generated from the Zipf distribution. 
Despite the model being mis-specified, most of the functionals are recovered quite accurately, with the median error being close to zero, except for the coverage of prevalence and the number of unseen species. For example, the first and third quartiles of the error for the missing mass are within $3\%$ in absolute value for all methods, and $11\%$ for the coverage probability.


\end{document}